\pdfoutput=1  
\documentclass[10pt, a4paper]{amsart}
\usepackage{amscd}
\usepackage{amssymb}
\usepackage[all]{xy}
\usepackage{lmodern}
\usepackage[T1]{fontenc}
\usepackage{graphicx}
\title[The Squaring Operation]{The Squaring Operation for Commutative DG Rings}
\date{24 October 2015}
\usepackage{hyperref}
\hypersetup{colorlinks=false}

\author{Amnon Yekutieli}
\address{Department of  Mathematics,
Ben Gurion University, Be'er Sheva 84105, Israel}
\email{amyekut@math.bgu.ac.il}

\thanks{
Primary: 16E45. Secondary: 16E35, 18G10, 13D03, 18E30.}

\thanks{Supported by the Israel Science Foundation grant no.\ 253/13.}

\newtheorem{thm}[equation]{Theorem}
\newtheorem{cor}[equation]{Corollary}
\newtheorem{prop}[equation]{Proposition}
\newtheorem{lem}[equation]{Lemma}
\theoremstyle{definition}
\newtheorem{dfn}[equation]{Definition}
\newtheorem{rem}[equation]{Remark}
\newtheorem{exa}[equation]{Example}

\newtheorem{cond}[equation]{Condition}

\newtheorem{setup}[equation]{Setup}

\numberwithin{equation}{section}
\newcommand{\iso}{\xrightarrow{\simeq}}

\newcommand{\xar}{\xrightarrow}
\newcommand{\opn}{\operatorname}
\newcommand{\cat}[1]{\operatorname{\mathsf{#1}}}

\newcommand{\cd}{\,{\cdot}\,}

\newcommand{\rmitem}[1]{\item[\text{\textup{(#1)}}]}

\newcommand{\mrm}[1]{\mathrm{#1}}

\newcommand{\la}{\lambda}

\renewcommand{\th}{\theta}

\newcommand{\al}{\alpha}
\newcommand{\be}{\beta}
\newcommand{\ga}{\gamma}
\newcommand{\ep}{\epsilon}
\newcommand{\ze}{\zeta}

\newcommand{\K}{\mathbb{K}}

\newcommand{\Q}{\mathbb{Q}}
\newcommand{\Z}{\mathbb{Z}}
\newcommand{\N}{\mathbb{N}}

\newcommand{\ali}[1]{\begin{aligned} #1 \end{aligned}}

\newcommand{\tup}[1]{\textup{#1}}
\newcommand{\bsym}[1]{\boldsymbol{#1}}

\newcommand{\ot}{\otimes}

\newcommand{\til}[1]{\tilde{#1}}

\newcommand{\bra}[1]{\langle #1 \rangle}
\renewcommand{\d}{\mathrm{d}}

\newcommand{\lb}{\linebreak}

\newcommand{\sbmat}[1]{\left[ \begin{smallmatrix} #1
\end{smallmatrix} \right]}
\newcommand{\bmat}[1]{\begin{bmatrix} #1 \end{bmatrix}}
\newcommand{\twoto}{\Rightarrow}
\newcommand{\centover}{ /_{\! \mrm{ce}}\,}

\begin{document}

\begin{abstract}
Let $A \to B$ be a homomorphism of commutative rings. The {\em squaring 
operation} is a functor $\opn{Sq}_{B / A}$ from the derived category 
$\cat{D}(B)$ of complexes of $B$-modules into itself. This  
operation is needed for the definition of {\em rigid complexes} (in the sense of 
Van den Bergh), that in turn leads to a new approach to Grothendieck duality for 
rings, schemes and even DM stacks. 

In our paper with J.J. Zhang from 2008 we introduced the squaring operation, 
and explored some of its properties. Unfortunately some of the proofs in that 
paper had severe gaps in them. 

In the present paper we reproduce the construction of the squaring operation. 
This is done in a more general context than in the first paper: here 
we consider a homomorphism $A \to B$ of {\em commutative DG rings}. Our 
first main result is that 
{\em the square $\opn{Sq}_{B / A}(M)$ of a DG $B$-module $M$ is independent 
of the resolutions used to present it}. 
Our second main result is on {\em the trace functoriality of the squaring 
operation}. We give precise statements and complete correct proofs. 

In a subsequent paper we will reproduce the remaining parts of the 2008 paper 
that require fixing. This will allow us to proceed with the other papers, 
mentioned in the bibliography, on the rigid approach to Grothendieck duality. 

The proofs of the main results require a substantial amount of foundational 
work on commutative and noncommutative DG rings, including a study of {\em 
semi-free DG rings}, their lifting properties, and their homotopies. This part 
of the paper could be of independent interest. 
\end{abstract}

\maketitle
\tableofcontents

\setcounter{section}{-1}
\section{Introduction}
\numberwithin{equation}{subsection}

\subsection{Background: Rigid Dualizing Complexes}
The concept of {\em rigid dualizing complex} was introduced by M. Van den Bergh 
in his influential paper \cite{VdB} from 1997. This was done for a 
noncommutative ring $A$ 
over a base field $\K$. Let us recall the definition; but to simplify matters 
(and because here we are interested in the commutative situation), we 
shall state it for a commutative ring $A$.

So let $\K$ be a field and $A$ a commutative $\K$-ring. We denote by 
$\cat{D}(A)$ the derived category of $A$-modules. Given a complex 
$M \in \cat{D}(A)$, its {\em square} is the complex 
\begin{equation} \label{eqn:490}
\opn{Sq}_{A / \K}(M) := 
\opn{RHom}_{A \ot_{\K} A}(A, M \ot_{\K} M) \in \cat{D}(A) .
\end{equation}
The $A$-module structure on $\opn{Sq}_{A / \K}(M)$ comes from the first 
argument of $\opn{RHom}$. 
Note that if $M$ is a single module, then 
\[ \opn{H}^q \bigl( \opn{Sq}_{A / \K}(M) \bigr) = 
\opn{HH}^q(A, M \ot_{\K} M) , \]
the $q$-th Hochschild cohomology of the $A$-bimodule $M \ot_{\K} M$.

Now assume $A$ is finitely generated over $\K$ (and hence it is noetherian). A 
{\em rigid dualizing complex} over $A$ relative 
to $\K$ is a pair $(R, \rho)$, where $R \in \cat{D}(A)$ is a dualizing complex 
(in the sense of \cite{RD}), and 
\begin{equation} \label{eqn:499}
\rho : R \iso \opn{Sq}_{A / \K}(R) 
\end{equation}
is an isomorphism in $\cat{D}(A)$. 
Van den Bergh proved that a rigid dualizing complex $(R, \rho)$ exists, and 
it is unique up to isomorphism. Further work in this direction was done 
by J.J. Zhang and the author in a series of papers; see \cite{YZ1} and its 
references. (These papers dealt with the noncommutative 
situation, which is significantly more complicated.) 

In the current paper we are interested in {\em commutative rings}, but in a 
{\em relative situation}: instead of a base field $\K$, we have a homomorphism 
$A \to B$ of commutative rings, and we want to produce a useful theory of 
squaring and rigidity. 

If the homomorphism $A \to B$ is {\em flat}, then there is a pretty easy way to 
generalize (\ref{eqn:490}), as follows. Given a complex $M \in \cat{D}(B)$, we 
may define 
\begin{equation} \label{eqn:491}
\opn{Sq}_{B / A}(M) := 
\opn{RHom}_{B \ot_A B}(B, M \ot^{\mrm{L}}_A M) \in \cat{D}(B) . 
\end{equation}
However, when $A \to B$ is not flat, formula (\ref{eqn:491})  is 
meaningless, since  there is no way to interpret $M \ot^{\mrm{L}}_A M$ as an 
object of $\cat{D}(B \ot_A B)$~!

In our paper with Zhang \cite{YZ2}, we proposed to solve the flatness problem 
by replacing the ring $B$ with a {\em flat DG ring resolution} $\til{B} \to B$ 
over $A$. Further work on rigid dualizing complexes over commutative rings was 
done in the paper \cite{YZ3}, and more is outlined in the survey paper 
\cite{Ye1}. 

There is work in progress (\cite{Ye6} 
and \cite{Ye7}) about rigid dualizing on schemes and Deligne-Mumford stacks. 
Indeed, the rigid approach to Grothendieck duality allows, for the first time, 
to state and prove global duality for a proper map between DM stacks. 
The ideas are outlined in the lecture notes \cite{Ye3}. 

Unfortunately, there were {\em serious flaws in the proofs in \cite{YZ2}}, 
as explained in Subsection \ref{subsec:discussion} of the Introduction. 
The discovery of these flaws forced us to go back and repair the foundations. 
Thus, {\em in the present paper we provide a comprehensive and correct 
treatment of the squaring operation}, using the DG ring method. 
The companion paper \cite{Ye5} is destined to repair further problems in 
\cite{YZ2}, and to enhance the results of \cite{YZ3}. 

Once the repairs to the foundations are done, we intend to proceed with the 
geometric application of rigid dualizing complexes (namely the papers in 
progress \cite{Ye6} and \cite{Ye7}).

\subsection{DG Rings and their Resolutions} \label{subsec:DGR-res}
Let $A = \bigoplus_{i \in \Z} A^i$ be a DG ring, where as usual ``DG'' is short 
for ``differential graded''. (Most other texts would call $A$ a unital 
associative DG algebra over $\Z$; see Remark \ref{rem:1025} for a discussion 
of nomenclature.) We say that $A$ is {\em nonpositive} if  
$A^i = 0$ for $i > 0$. The DG ring $A$ is called {\em strongly commutative} if 
$a \cd b = (-1)^{i  j} \cd b \cd a$ for $a \in A^i$ and $b \in A^j$, and 
$a \cd a = 0$ if $i$ is odd. 
We call $A$ a {\em commutative DG ring} if it is both nonpositive and strongly 
commutative. This is the term appearing in the title of the paper. 

We denote by $\cat{DGR}_{\mrm{sc}}^{\leq 0}$ the category of commutative 
DG rings. Observe that the category of commutative rings is a full subcategory 
of $\cat{DGR}_{\mrm{sc}}^{\leq 0}$, since a ring 
can be seen as a DG ring concentrated in degree $0$. 
All DG rings appearing in the Introduction are commutative (with the 
exception of Subsections \ref{subsec:proofs} and \ref{subsec:discussion}). 
However, in the paper itself we must deal with noncommutative DG rings, as 
explained in Subsection  \ref{subsec:proofs}.

Even though our motivation is the squaring operation for commutative rings, 
we develop the squaring operation more generally for commutative DG rings, 
and this is the way we shall present the results in the Introduction. 
The reason is twofold: first, there is no added difficulty in working with 
commutative DG rings (as compared to commutative rings); and second, the 
presentation is cleaner when working totally in the commutative DG framework. 

Let $A \xar{u} B$ be a homomorphism of commutative DG rings, i.e.\ a morphism 
in $\cat{DGR}_{\mrm{sc}}^{\leq 0}$. We refer to $A \xar{u} B$ as a {\em pair of 
commutative DG rings}, with notation $B / A$. These pairs form a category: 
given another pair 
$B' / A' = (A' \xar{u'} B')$, a morphism of pairs 
$w / v : B' / A' \to B / A$
is a commutative diagram 
\begin{equation} \label{eqn:985}
\UseTips \xymatrix @C=6ex @R=6ex {
A'
\ar[r]^{u'}
\ar[d]_{v}
&
B'
\ar[d]^{w}
\\
A
\ar[r]^{u}
&
B
} 
\end{equation}
in $\cat{DGR}_{\mrm{sc}}^{\leq 0}$.
The resulting category is denoted by 
$\cat{PDGR}_{\mrm{sc}}^{\leq 0}$.

Consider a pair $B / A$ in $\cat{PDGR}_{\mrm{sc}}^{\leq 0}$.
A {\em K-flat resolution} of $B / A$ is a morphism 
$s / r : \til{B} / \til{A} \to B / A$
in $\cat{PDGR}_{\mrm{sc}}^{\leq 0}$, such that $\til{B}$ is K-flat as a DG 
$\til{A}$-module, $r : \til{A} \to A$ is a quasi-isomorphism, and 
$s : \til{B} \to B$ is a surjective quasi-isomorphism.
(See Section \ref{sec:resolutions-modules} for a review of various kinds of DG 
module resolutions, including K-flat ones.) 

\begin{equation} \label{eqn:482}
\UseTips \xymatrix @C=8ex @R=6ex {
\til{A}
\ar[r]^{\til{u}} ="tilu"
\ar[d]_{r} ="v"
{} \save 
[]+<-12ex,-1ex> *+[F-:<3pt>]{\scriptstyle \text{qu-isom}} 
\ar@(r,l)@{..} "v" 
\restore
&
\til{B}
\ar[d]^{s} ="w"
{} \save 
[]+<10ex,4ex> *+[F-:<3pt>]{\scriptstyle \text{K-flat}} 
\ar@(l,ur)@{..} "tilu" 
\restore
&
{} \save 
[]+<4ex,-1ex> *+[F-:<3pt>]{\scriptstyle \text{surj qu-isom}} 
\ar@(l,r)@{..} "w" 
\restore
\\
A
\ar[r]_{u} ="u"
&
B
}
\end{equation}

Let $w / v : B' / A' \to B / A$
be a morphism in $\cat{PDGR}_{\mrm{sc}}^{\leq 0}$. A K-flat resolution of 
$w / v$ is a commutative diagram
\begin{equation} \label{eqn:986}
\UseTips \xymatrix @C=8ex @R=6ex {
\til{B}' / \til{A}'
\ar[r]^{\til{w} / \til{v}}
\ar[d]_{s' / r'}
&
\til{B} / \til{A}
\ar[d]^{s / r}
\\
B' / A'
\ar[r]^{w / v}
&
B / A
}
\end{equation}
in $\cat{PDGR}_{\mrm{sc}}^{\leq 0}$, where the vertical arrows $s / r$ and 
$s' / r'$ are K-flat resolutions.
We also say that $\til{w} / \til{v}$ is a morphism of resolutions above 
$w / v$. In case $B' / A' = B / A$, and  
$w / v$ is the identity automorphism of this pair, 
we say that $\til{w} / \til{v}$ is a morphism of resolutions of $B / A$. 
According to Proposition \ref{prop:1010}, 
K-flat resolutions of pairs, and of morphisms between pairs, exist. 

Let $A  \xar{u} B \xar{v} C$ be homomorphisms in 
$\cat{DGR}_{\mrm{sc}}^{\leq 0}$.
We refer to this data as a {\em triple of commutative DG rings}. 
A K-flat resolution of this triple is a commutative diagram 
\begin{equation} \label{eqn:966}
\UseTips \xymatrix @C=6ex @R=6ex {
\til{A}
\ar[r]^{\til{u}}
\ar[d]^{r}
&
\til{B}
\ar[d]^{s}
\ar[r]^{\til{v}}
&
\til{C}
\ar[d]^{t}
\\
A
\ar[r]^{u}
&
B
\ar[r]^{v}
&
C
} 
\end{equation}
in $\cat{DGR}_{\mrm{sc}}^{\leq 0}$,
such that 
$\til{B} / \til{A}$ is a K-flat resolution of the pair $B / A$, and 
$\til{C} / \til{A}$ is a K-flat resolution of the pair $C / A$.
The first triple above can be viewed as a morphism of pairs 
$v / \mrm{id}_A : B / A \to C / A$; and 
$\til{v} / \mrm{id}_{\til{A}} : \til{B} / \til{A}  \to \til{C} / \til{A}$
can be viewed as a K-flat resolution of $v / \mrm{id}_A$. 
Such resolutions exist by Proposition \ref{prop:1011}.

\subsection{The Squaring Operation} \label{subsec:squaring}
Recall that our DG rings are all commutative. 
To a DG ring $A$ we associate the category of DG $A$-modules 
$\cat{M}(A)$, and its derived category $\cat{D}(A)$. There is an additive 
functor $\opn{Q} : \cat{M}(A) \to  \cat{D}(A)$, which is the identity on 
objects, and which sends quasi-isomorphisms to isomorphisms. 
In Sections \ref{sec:facts-DG-mods} and \ref{sec:resolutions-modules} we recall 
some facts about DG modules, their resolutions, and related derived functors.

Let $A \xar{u} B$ be a homomorphism of DG rings, namely an object $B / A$ of 
$\cat{PDGR}_{\mrm{sc}}^{\leq 0}$. 
Given a DG $B$-module $M$, and a K-flat resolution 
$\til{B} / \til{A}$ of $B / A$, 
we let 
\begin{equation} \label{eqn:815}
\opn{Sq}_{B / A}^{\til{B} / \til{A}} (M) := 
\opn{RHom}_{\til{B} \ot_{\til{A}} \til{B}}(B, M \ot_{\til{A}}^{\mrm{L}} M)
\in \cat{D}(B) . 
\end{equation}
In Proposition \ref{prop:915} and Definition \ref{dfn:950} we provide an 
explicit presentation of \lb 
$\opn{Sq}_{B / A}^{\til{B} / \til{A}} (M)$, 
in terms of {\em compound resolutions}. 
See Remark \ref{rem:965} regarding symmetric compound resolutions.  

Consider a morphism of pairs 
$w / v : B' / A' \to B / A$, a DG module $M \in \cat{D}(B)$, a DG module
$M' \in \cat{D}(B')$, and a morphism 
$\th : M \to M'$ in $\cat{D}(B')$.
Clarification: actually $\th$ is a morphism $\opn{For}_w(M) \to M'$, where 
$\opn{For}_w : \cat{D}(B) \to \cat{D}(B')$
is the forgetful functor (restriction of scalars) corresponding to the DG ring 
homomorphism $w : B' \to B$. However, most of the time 
we shall suppress these forgetful functors, for the sake of clarity. 

Given a K-flat resolution
$\til{w} / \til{v} : \til{B}' / \til{A}' \to \til{B} / \til{A}$
of $w / v$, there is a morphism
\begin{equation} \label{eqn:961}
\opn{Sq}_{w / v}^{\til{w} / \til{v}} (\th) : 
\opn{Sq}_{B / A}^{\til{B} / \til{A}} (M) \to
\opn{Sq}_{B' / A'}^{\til{B}' / \til{A}'} (M')
\end{equation}
in $\cat{D}(B')$. The morphism 
$\opn{Sq}_{w / v}^{\til{w} / \til{v}} (\th)$
is constructed using compound resolutions -- see Proposition \ref{prop:916} and 
Definition \ref{dfn:956}. 
This construction is functorial in all arguments. 
By this we mean that if
$w' / v' : B'' / A'' \to B' / A'$
is another morphism in $\cat{PDGR}_{\mrm{sc}}^{\leq 0}$,
with K-flat resolution 
$\til{w}' / \til{v}' : \til{B}'' / \til{A}'' \to \til{B}' / \til{A}'$, 
and if $\th' : M' \to M''$
is a morphism in $\cat{D}(B'')$, then 
\[ \opn{Sq}_{(w / v) \circ (w' / v')}^{(\til{w} / \til{v}) \circ 
(\til{w}' / \til{v}')} (\th' \circ \th) = 
\opn{Sq}_{w' / v'}^{\til{w}' / \til{v}'} (\th') \circ
\opn{Sq}_{w / v}^{\til{w} / \til{v}} (\th) , \]
as morphisms 
\[ \opn{Sq}_{B / A}^{\til{B} / \til{A}} (M) \to
\opn{Sq}_{B'' / A''}^{\til{B}'' / \til{A}''} (M'') \]
in $\cat{D}(B'')$. See Proposition \ref{prop:940}. 

Here is the key technical result of our paper. It is repeated as Theorem 
\ref{thm:950} in the body of the paper, and proved there. 
A brief discussion of the proof can be found in Subsection \ref{subsec:proofs} 
of the Introduction. 

\begin{thm}[Homotopy Invariance] \label{thm:960}
Let $w / v : B' / A' \to B / A$ be a morphism in  
$\cat{PDGR}_{\mrm{sc}}^{\leq 0}$, let 
$M \in \cat{D}(B)$, let $M' \in \cat{D}(B')$, and let 
$\th : M \to M'$ be a morphism in $\cat{D}(B')$.
Suppose
$\til{B} / \til{A}$ and $\til{B}' / \til{A}'$
are K-flat resolutions of 
$B / A$ and $B' / A'$ respectively in $\cat{PDGR}_{\mrm{sc}}^{\leq 0}$,
and  
\[ \til{w}_0 / \til{v}_0,  \, \til{w}_1 / \til{v}_1 :
\til{B}' / \til{A}' \to \til{B} / \til{A} \]
are morphisms of resolutions above $w / v$. Then the morphisms 
\[ \opn{Sq}_{w / v}^{\til{w}_0 / \til{v}_0}(\th), \,
\opn{Sq}_{w / v}^{\til{w}_1 / \til{v}_1}(\th) :
\opn{Sq}_{B / A}^{\til{B} / \til{A}}(M) \to 
\opn{Sq}_{B' / A'}^{\til{B}' / \til{A}'}(M') \]
in $\cat{D}(B')$ are equal. 
\end{thm}

Here is the first main theorem of our paper. 

\begin{thm}[Existence of Squares] \label{thm:965}
Let $A \to B$ be a homomorphism of commutative DG rings,
and let $M$ be a DG $B$-module. There is a DG $B$-module 
$\opn{Sq}_{B / A}(M)$, 
unique up to a unique isomorphism in $\cat{D}(B)$, together with 
an isomorphism
\[ \opn{sq}^{\til{B} / \til{A}} : 
\opn{Sq}_{B / A}(M) \iso \opn{Sq}_{B / A}^{\til{B} / \til{A}}(M) \]
in $\cat{D}(B)$ for any K-flat resolution 
$\til{B} / \til{A}$ of $B / A$, 
satisfying the following condition. 
\begin{enumerate}
\item[($*$)] For any morphism 
$\til{w} / \til{v} : \til{B}' / \til{A}' \to \til{B} / \til{A}$
of resolutions of $B / A$, the diagram 
\[ \UseTips \xymatrix @C=18ex @R=8ex {
\opn{Sq}_{B / A}(M)
\ar[d]_{ \opn{sq}^{\til{B} / \til{A}} }
\ar[dr]^{ \opn{sq}^{\til{B}' / \til{A}'} }
\\
\opn{Sq}_{B / A}^{\til{B} / \til{A}}(M)
\ar[r]_{ \opn{Sq}_{\mrm{id} / \mrm{id}}^{\til{w} / \til{v}}(\mrm{id}_M) }
&
\opn{Sq}_{B / A}^{\til{B}' / \til{A}'}(M)
} \]
of isomorphisms in $\cat{D}(B)$ is commutative. 
\end{enumerate}
\end{thm}

The DG module $\opn{Sq}_{B / A}(M)$ is called the {\em square of $M$ over $B$ 
relative to $A$}. 

This theorem is repeated as Theorem \ref{thm:860} in the body of the paper. 
As we already mentioned, this assertion (for rings only, not DG rings) appeared 
as \cite[Theorem 2.2]{YZ2}; but the proof there had a large gap in it. 
For a discussion, and a comparison to \cite[Theorem 3.2]{AILN}, see
Subsection \ref{subsec:discussion} below.

Here is the second main theorem of the paper. 

\begin{thm}[Trace Functoriality] \label{thm:966}
Let $A \to B \xar{v} C$ 
be homomorphisms of commutative DG rings, 
let $M \in \cat{D}(B)$, let $N \in \cat{D}(C)$, and let 
$\th : N \to M$ be a morphism in $\cat{D}(B)$.
There is a unique morphism 
\[ \opn{Sq}_{v / \mrm{id}_A}(\th) :
\opn{Sq}_{C / A}(N) \to \opn{Sq}_{B / A}(M) \]
in $\cat{D}(B)$, satisfying the condition\tup{:}
\begin{enumerate}
\item[($**$)] For any K-flat resolution
$\til{A} \to \til{B} \xar{\til{v}} \til{C}$
of the triple $A \to B \xar{v} C$, the diagram 
\[ \UseTips \xymatrix @C=14ex @R=8ex {
\opn{Sq}_{C / A}(N)
\ar[d]_{ \opn{sq}^{\til{C} / \til{A}} }
\ar[r]^(0.5){ \opn{Sq}_{v / \mrm{id}}(\th) }
&
\opn{Sq}_{B / A}(M) 
\ar[d]^{ \opn{sq}^{\til{B} / \til{A}} }
\\
\opn{Sq}_{C / A}^{\til{C} / \til{A}}(N)
\ar[r]^(0.5){ \opn{Sq}_{v / \mrm{id}}^{\til{v} / \mrm{id}}(\th) }
&
\opn{Sq}_{B / A}^{\til{B} / \til{A}}(M) 
} \]
in $\cat{D}(B)$ is commutative. 
\end{enumerate}
\end{thm}

This theorem is repeated as Theorem \ref{thm:871}. The statement already 
appeared as \cite[Theorem 2.3]{YZ2} (for rings only, not DG rings). But the 
proof in loc.\ cit.\ was also incorrect. 

The known functoriality of the operation 
$\opn{Sq}_{B / A}^{\til{B} / \til{A}}$
(from Section \ref{sec:pairs-DG-modules}) now implies that the assignments 
$M \mapsto \opn{Sq}_{B / A}(M)$ and
$\th \mapsto \opn{Sq}_{\mrm{id}_B / \mrm{id}_A}(\th)$
are a functor 
\[ \opn{Sq}_{B / A} : \cat{D}(B) \to  \cat{D}(B) , \]
called the {\em squaring operation  over $B$ relative to $A$}. 

The functor $\opn{Sq}_{B / A}$ is not linear; in fact it is a 
{\em quadratic functor}, in the following sense. Given a morphism 
$\th : N \to M$ in $\cat{D}(B)$ and an element 
$b \in \opn{H}^0(B)$, we have 
\[ \opn{Sq}_{B / A}(b \cd \th) = b^2 \cd \opn{Sq}_{B / A}(\th) , \]
as morphisms 
$\opn{Sq}_{B / A}(N) \to \opn{Sq}_{B / A}(M)$ 
in $\cat{D}(B)$. See Theorem \ref{thm:1010} for a slightly more general 
statement.

\subsection{On the Proofs} \label{subsec:proofs}
Theorems \ref{thm:965} and \ref{thm:966} are rather easy to prove, once we know 
Theorem \ref{thm:960} (and the existence of resolutions). 

However, the proof of Theorem \ref{thm:960} is quite long and difficult. 
As its name suggests, it involves a homotopy between the homomorphisms 
$\til{w}_0, \til{w}_1 : \til{B}' \to \til{B}$.
The only way we know to produce a homotopy 
$\til{w}_0 \twoto \til{w}_1$ is when the DG ring $\til{B}'$ is 
{\em noncommutative semi-free} over $\til{A}'$, and there is equality 
$\til{v}_0 = \til{v}_1$.  See Definition \ref{dfn:70} and 
Theorem \ref{thm:300}.

We are thus forced to deal with {\em noncommutative DG rings}. 
Instead of the pairs of commutative DG rings that were discussed above, 
the relevant noncommutative object is a {\em central pair of DG rings}, namely a 
central homomorphism $u : A \to B$ with commutative source $A$. See 
Definitions \ref{dfn:750} and \ref{dfn:910}. 

Let $B / A$ be a central pair of DG rings. A K-flat resolution of it is a 
central pair of DG rings $\til{B} / \til{A}$, such that $\til{B}$ is a K-flat 
DG $\til{A}$-module, together with a quasi-isomorphism 
$\til{B} / \til{A} \to B / A$
as in diagram (\ref{eqn:482}).
Now $B$ is a DG bimodule over 
$\til{B}$, namely it is a DG module over the {\em enveloping DG ring}
$\til{B} \ot_{\til{A}} \til{B}^{\mrm{op}}$. 
Because of this, we must replace the single DG module $M$ by a pair 
$(M^{\mrm{l}}, M^{\mrm{r}})$, consisting of a left DG $B$-module 
$M^{\mrm{l}}$ and a right DG $B$-module $M^{\mrm{r}}$.
In this situation, instead of a ``square'' we have a ``rectangle'':
\begin{equation} \label{eqn:965}
\opn{Rect}_{B / A}^{\til{B} / \til{A}}(M^{\mrm{l}}, M^{\mrm{r}})  := 
\opn{RHom}_{ \til{B} \ot_{\til{A}} \til{B}^{\mrm{op}} }
(B, M^{\mrm{l}} \ot^{\mrm{L}}_{\til{A}} M^{\mrm{r}}) . 
\end{equation}
This object lives in the derived category 
$\cat{D}(B^{\mrm{ce}})$, where $B^{\mrm{ce}}$ is the {\em center} of $B$ (in 
the graded sense, see Definition \ref{dfn:750}). 

In the noncommutative situation, the Homotopy Theorem (the noncommutative 
version of Theorem \ref{thm:960}) is Theorem \ref{thm:840}.
It is proved in two stages. First, an important special case 
($\til{B}' / \til{A}'$ is noncommutative semi-free and $\til{v}_0 = \til{v}_1$) 
is done in Lemma \ref{lem:827}. The proof of this 
lemma relies on Theorem \ref{thm:300}, and uses the cylinder DG ring. 
The general case is then reduced to that special case.
Theorem \ref{thm:985} (Existence of Rectangles) is the noncommutative version 
of Theorem \ref{thm:965}. 

The rectangle operation is interesting on its own, even when $B / A$ is a 
commutative pair of rings. This is due to its connection with the monoidal 
operation $\ot^!$. See Remark \ref{rem:720}.

\subsection{Discussion of Related Papers} \label{subsec:discussion}
Early versions of Theorems \ref{thm:965} and \ref{thm:966} had already appeared 
in our paper \cite{YZ2} with Zhang, as Theorems 2.2 and 2.3 respectively. 
However, to our great embarrassment, the proofs of these results in \cite{YZ2} 
had severe gaps in them. The gaps are well-hidden in the text; but one clear 
error is this: the homomorphism 
$\sbmat{\phi_0 & 0 \\ 0 & \phi_1}$
in the middle of page 3225 (in the proof  of Theorem 2.2) does not make any 
sense (unless $\phi_0 = \phi_1$). In order to give a correct treatment, it is 
necessary to work with noncommutative semi-free DG ring resolutions, as 
explained in Subsection \ref{subsec:proofs} above. 

The mistake in the proof of \cite[Theorem 2.2]{YZ2} was discovered by Avramov, 
Iyengar, Lipman and Nayak \cite{AILN}. They also found a way to fix it, and 
their \cite[Theorem 3.2]{AILN} is a generalization of our 
\cite[Theorem 2.2]{YZ2}. Indeed, \cite[Theorem 3.2]{AILN} establishes the 
rectangle operation in the case when $A \to B$ is a central pair of rings. 
The proof of \cite[Theorem 3.2]{AILN} relies on the Quillen 
model structure on the category $\cat{DGR} \centover A$ of noncommutative DG 
rings, central over a commutative base ring $A$, following \cite{BP}. 
More on this aspect can be found in Remark \ref{rem:780}.
Note that our Theorem \ref{thm:840} (the noncommutative version of Theorem 
\ref{thm:965}) is stronger than \cite[Theorem 3.2]{AILN}, because we allow $A 
\to B$ be a central pair of nonpositive DG rings. 

Our goals are different from those of \cite{AILN}, and hence we adopt a 
different strategy. For us the relative situation of a homomorphism of 
commutative rings $A \to B$ is crucial, and we must consider triples of 
commutative rings $A \to B \to C$. For this reason we stick to commutative 
K-flat DG ring resolutions, relying on Theorem \ref{thm:960}. 
We do not know whether the methods of \cite{AILN}
can be adapted to yield the trace functoriality of the squaring operation 
(Theorem \ref{thm:966}).

\medskip \noindent
{\em Acknowledgments}. 
I wish to thank James Zhang,  Bernhard Keller, Vladimir Hinich, 
Liran Shaul, Rishi Vyas, Asaf Yekutieli and Sharon Hollander for their help in 
writing this paper.

\section{Some Facts on DG Rings and Modules}
\label{sec:facts-DG-mods}
\numberwithin{equation}{section}

In this section we review some known facts about DG rings and modules. 
We talk about central homomorphisms of DG rings. Finally, we introduce 
the cylinder construction for DG rings and modules.  

A {\em differential graded ring}, or {\em DG ring} for short, is a 
graded ring $A = \bigoplus_{i \in \Z} A^i$, together with an additive 
endomorphism $\d_A$ of degree $1$ called the {\em differential}. 
The differential $\d_A$ satisfies $\d_A \circ \d_A = 0$, and it is a graded 
derivation of $A$, namely it satisfies the graded Leibniz rule
\begin{equation} \label{eqn:320}
\d_A(a \cd b) = \d_A(a) \cd b + (-1)^{i} \cd a \cd \d_A(b) 
\end{equation}
for $a \in A^{i}$ and $b \in A^{j}$. (Most other texts would call such 
$A$ an associative unital DG algebra.)   

\begin{dfn}
We denote by $\cat{DGR}$ the category of all DG rings,
where the morphisms are the graded ring homomorphisms $u : A \to B$
that commute with the differentials.
\end{dfn}

We consider rings as DG rings concentrated in degree $0$. 
Thus the category of rings becomes a full subcategory of $\cat{DGR}$.

\begin{dfn} \label{dfn:400}
Let $A$ be a DG ring. A {\em DG ring over $A$}, or a {\em DG $A$-ring}, 
is a pair $(B, u)$, where $B$ is a DG ring and $u : A \to B$ is a DG ring
homomorphism. 

Suppose $(B, u)$ and $(B', u')$ are DG rings over $A$. 
A {\em homomorphism of DG rings over $A$}, or a {\em DG $A$-ring homomorphism}, 
is a homomorphism of DG rings $v : B \to B'$ such that $u' = v \circ u$.

We denote by $\cat{DGR} / A$ the category of DG rings over $A$.
\end{dfn}
 
Thus $\cat{DGR} / A$ is the usual ``coslice category'', except that we say 
``over'' instead of ``under''. See Remark \ref{rem:1016} regarding the 
direction of arrows. There is an obvious forgetful functor 
$\cat{DGR} / A \to \cat{DGR}$.
A morphism $v$ in $\cat{DGR} / A$ is shown in this commutative diagram:
\[ \UseTips \xymatrix @C=6ex @R=6ex {
B
\ar[r]^{v}
&
B'
\\
A
\ar[u]^{u}
\ar[ur]_{u'}
} \]

\begin{rem} \label{rem:1025}
We should probably justify the use of ``DG ring'', instead of the traditional 
``DG algebra''. We maintain that more generally, whenever we have a 
homomorphism $A \to B$ in the category of rings, then $B$ ought to be called 
an {\em $A$-ring}. Below is an explanation. 

Consider a commutative base ring $\K$. Given a ring homomorphism 
$\K \to A$ such that the target $A$ is commutative, we simply 
say that ``$A$ is a commutative $\K$-ring''. In the noncommutative context 
things are more delicate. When one says that ``$A$ is a $\K$-algebra'' 
(associative and unital), the meaning is that the ring homomorphism 
$\K \to A$ is central; namely the image of $\K$ is in the center of $A$. We 
suggest instead to use ``$A$ is a central $\K$-ring''. 

In the DG context, one says that ``$A$ is a DG $\K$-algebra'' (again 
associative and unital) if the image of $\K$ lies inside the subring of degree 
$0$ central cocycles of $A$. We propose to use ``central DG $\K$-ring'' 
instead, as in Definition \ref{dfn:750} below. 

Notice that the name ``ring'' allows us to drop the adjectives ``associative'' 
and ``unital''; and it also separates DG rings from DG Lie algebras etc. 

In the present paper we actually need to consider DG ring 
homomorphisms $u : A \to B$, where the source $A$ is a genuine DG ring, and 
in some cases it is even noncommutative. Thus $u(A)$ is sometimes not in the  
center of $B$, nor in the subring of cocycles, and we can't say that ``$B$ is a 
DG $A$-algebra''. But we can, and do, say that ``$B$ is a DG $A$-ring''.  
The homomorphism $u$ might be central nonetheless, as in Definition 
\ref{dfn:750}.
\end{rem}

Consider a graded ring $A$. For homogeneous elements $a \in A^i$ and 
$b \in A^j$ we define the graded commutator to be 
\begin{equation} \label{eqn:750}
\opn{ad}(a)(b) := a \cd b - (-1)^{i j} \cd b \cd a . 
\end{equation}
This extends to all $a, b \in A$ using bilinearity. Namely, if 
$a = \sum_k a_k$ and $b = \sum_l b_l$, with $a_k \in A^{i_k}$ and 
$b_l \in A^{j_l}$, then 
\[ \opn{ad}(a)(b) := \sum_{k, l} \opn{ad}(a_k)(b_l) . \]
Often the notation $[a, b]$ is used for the graded commutator. 

\begin{dfn} \label{dfn:750}
Let $A$ be a DG ring. 
\begin{enumerate}
\item The {\em center} of $A$ is the DG subring 
\[ A^{\mrm{ce}} := \{ a \in A \mid \opn{ad}(a)(b) = 0 \ \, \tup{for all} \ \,
b \in A \} . \]

\item The DG ring $A$ is called {\em weakly commutative} if 
$A^{\mrm{ce}} = A$. In other words, if 
\[ a \cd b = (-1)^{i j} \cd b \cd a \]
for all $a \in A^i$ and $b \in A^j$.

\item A homomorphism $u : A \to B$ in $\cat{DGR}$ is called {\em central} if 
$u(A^{\mrm{ce}}) \subseteq B^{\mrm{ce}}$. 
In this case, the resulting DG ring homomorphism $A^{\mrm{ce}} \to B^{\mrm{ce}}$
is denoted by $u^{\mrm{ce}}$.

\item We denote by $\cat{DGR} \centover A$ to be the the full subcategory of 
$\cat{DGR} / A$ consisting of pairs $(B, u)$ such that $u : A \to B$ is 
central.  
\end{enumerate}
\end{dfn}

Implicit in item (1) of the definition above is the (easy to check) fact that 
$A^{\mrm{ce}}$ is indeed a DG subring of $A$. The DG ring $A^{\mrm{ce}}$ itself 
is weakly commutative. 
As for item (4): a morphism $v : B \to B'$ in the category
$\cat{DGR} \centover A$ is not required to be central; the centrality condition 
is only on $A \xar{u} B$ and $A \xar{u'} B'$. 

If $u : A \to B$ and $v : B \to C$ are central homomorphisms 
in $\cat{DGR}$, then $v \circ u : A \to C$ is also central.
Thus we get a subcategory $\cat{DGR}_{\vec{\mrm{ce}}}$ of $\cat{DGR}$,
consisting of all objects, but only the central homomorphisms. 
Letting $\cat{DGR}_{\mrm{wc}}$ be the full subcategory of  $\cat{DGR}$ on the 
weakly commutative DG rings, we get a functor 
\[ (-)^{\mrm{ce}} : \cat{DGR}_{\vec{\mrm{ce}}} \to \cat{DGR}_{\mrm{wc}} . \]

Let $A$ be a DG ring. Recall that a left DG $A$-module is a left graded
$A$-module $M = \bigoplus_{i \in \Z} M^i$, with 
differential $\d_M$ of degree $1$, that satisfies $\d_M \circ \d_M = 0$ and
the graded Leibniz rule (equation (\ref{eqn:320}), but with $m \in M^{j}$ 
instead of $b$). Right DG $A$-modules are defined similarly. 
By default all DG modules in this paper are left DG modules. 

Let $A$ be a DG ring, let $M$ and $N$ be left DG $A$-modules, and let 
$\phi : M \to N$ be a degree $k$ 
additive (i.e.\ $\Z$-linear) homomorphism. We say that $\phi$ is $A$-linear if 
\begin{equation} \label{eqn:310}
\phi(a \cd m) =  (-1)^{i k} \cd a \cd \phi(m)
\end{equation}
for all $a \in A^i$ and $m \in M^j$. 
(This is the sign convention of \cite{ML}.)
The abelian group of $A$-linear homomorphisms of degree $k$ is denoted by 
$\opn{Hom}_{A}(M, N)^k$, and we let 
\begin{equation} \label{eqn:316}
\opn{Hom}_{A}(M, N) := \bigoplus_{k \in \Z} \,
\opn{Hom}_{A}(M, N)^{k} .
\end{equation}
The graded abelian group $\opn{Hom}_{A}(M, N)$ has a  differential $\d$ given by
\begin{equation} \label{eqn:690}
\d(\phi) := \d_N \circ \phi - (-1)^{k} \cd \phi \circ \d_M 
\end{equation}
for $\phi \in \opn{Hom}_A(M, N)^k$.
In particular, for $N = M$ we get a DG ring 
\begin{equation} \label{eqn:664}
\opn{End}_{A}(M) := \opn{Hom}_{A}(M, M) . 
\end{equation}

Suppose $A$ is a DG ring and $M$ a DG abelian group (i.e.\ a DG $\Z$-module). 
From (\ref{eqn:310}) we see that a DG $A$-module structure on $M$ is the same 
as a DG ring homomorphism $A \to \opn{End}_{\Z}(M)$. 

If $M$ is a right DG $A$-module and $N$ is a left DG $A$-module, then 
the usual tensor product $M \ot_A N$ is a graded abelian group: 
\[ M \ot_A N = \bigoplus_{k \in \Z} \, (M \ot_A N)^k , \]
where $(M \ot_A N)^k$ is the subgroup generated by the tensors $m \ot n$ 
such that $m \in M^i$ and $n \in N^{k - i}$. 
The graded abelian group 
$M \ot_A N$ has a differential $\d$ satisfying
\begin{equation} \label{eqn:220}
\d(m \ot n) := \d_M(m) \ot n + (-1)^i \cd m \ot \d_N(n)
\end{equation}
for $m, n$ as above.
  
Let $A$ and $B$ be DG rings. The tensor product $A \ot_{\Z} B$ is a DG ring, 
with multiplication 
\begin{equation} \label{eqn:820}
(a_1 \ot b_1) \cd (a_2 \ot b_2) := 
(-1)^{i_2 j_1} \cd (a_1 \cd a_2) \ot (b_1 \cd b_2)
\end{equation}
for $a_k \in A^{i_k}$ and $b_k \in B^{j_k}$, and with differential like in 
(\ref{eqn:220}). 

Let $B$ be a DG ring. Its  opposite DG ring $B^{\mrm{op}}$ is the 
same DG abelian group, but we denote its elements by 
$b^{\mrm{op}}$, $b \in B$. The multiplication in $B^{\mrm{op}}$ is
\[ b_1^{\mrm{op}} \cd b_2^{\mrm{op}} := 
(-1)^{j_1 j_2} \cd (b_2 \cd b_1)^{\mrm{op}} \]
for $b_k \in B^{j_k}$. 
Note that the centers satisfy
$(B^{\mrm{op}})^{\mrm{ce}} = (B^{\mrm{ce}})^{\mrm{op}}$,
and $B^{\mrm{op}} = B$ iff $B$ is weakly commutative.

If $M$ is a right DG $B$-module, then it can be seen as a left DG 
$B^{\mrm{op}}$-module, with action 
\[ b^{\mrm{op}} \cd m := (-1)^{k j} \cd m \cd b \]
for $b \in B^j$ and $m \in M^k$. 
Because of this observation, working only with left DG modules is 
not a limitation. We shall switch between notations according to convenience. 

Now consider a DG $A$-$B$-bimodule $M$, namely $M$ has a left $A$ action and a 
right $B$ action, and they commute, i.e.\ 
$a \cd (m \cd b) = (a \cd m) \cd b$. Then $M$ can be 
viewed as a left 
$A \ot_{\Z} B^{\mrm{op}}$ -module with action 
\begin{equation} \label{eqn:321}
(a \ot b^{\mrm{op}}) \cd m := (-1)^{k j} \cd a \cd m \cd b
\end{equation}
for $a \in A^{i}$, $b \in B^{j}$ and $m \in M^k$.
Thus a DG $A$-$B$-bimodule structure on a DG abelian group $M$ is the 
same as a homomorphism of DG rings  
\[ A \ot_{\Z} B^{\mrm{op}} \to \opn{End}_{\Z}(M) . \]

Let $A$ be a weakly commutative DG ring. Any $B \in \cat{DGR} \centover A$ is a 
DG $A$-bimodule.
If $B, C \in \cat{DGR} \centover A$, then the DG $A$-module
$B \ot_A C$ has an obvious DG ring structure, and moreover
$B \ot_A C \in \cat{DGR} \centover A$. Given a DG $B$-module $M$ 
and and a DG $C$-module $N$, the tensor product $M \ot_A N$ 
is a DG module over $B \ot_A C$, with action 
\[ (b \ot c) \cd (m \ot n) := (-1)^{i j} \cd (b \cd m) \ot (c \cd n) \]
for $c \in C^i$ and $m \in M^j$.

\begin{dfn}
Let $A$ be a DG ring. 
We denote by $\cat{DGMod} A$ the category of left DG $A$-modules,
where the morphisms are the $A$-linear homomorphisms
$\phi : M \to N$ of degree $0$ that commute with the differentials. 
Since the name is too long, we mostly use the abbreviation
$\cat{M}(A) := \cat{DGMod} A$.
\end{dfn}

The category $\cat{M}(A)$ is abelian. 
The set $\cat{M}(A)$ also has a DG category structure on it, in which the 
set of morphisms $M \to N$ is the DG abelian group 
$\opn{Hom}_{A}(M, N)$. The relation between these structures is 
\begin{equation} \label{eqn:1015}
\opn{Hom}_{\cat{M}(A)}(M, N) = \mrm{Z}^0 \bigl( 
\opn{Hom}_A(M, N) \bigr) ,
\end{equation}
the group of $0$-cocycles.

Let $M = \bigoplus_{i \in \Z} \, M^i$ be a graded abelian group. Recall that 
the shift (or twist, or translation, or suspension) of $M$ is the graded 
abelian group $\opn{T}(M) = \bigoplus_{i \in \Z} \, \opn{T}(M)^i$ defined as 
follows. The graded component of degree $i$ of $\opn{T}(M)$ is 
$\opn{T}(M)^i := M^{i + 1}$.
Note that as (ungraded) abelian groups we have $\opn{T}(M) = M$, but the 
gradings are different. Thus the identity automorphism of $M$ becomes a 
degree $-1$ invertible homomorphism of graded abelian groups
$\opn{t} : M \to \opn{T}(M)$. 
We shall usually represent elements of $\opn{T}(M)^i$ as 
$\opn{t}(m)$, with $m \in M^{i + 1}$.

If $M$ is a DG $A$-module, then $\opn{T}(M)$ has the following structure of DG 
$A$-module. The differential $\d_{\opn{T}(M)}$ is 
\begin{equation} \label{eqn:680}
\d_{\opn{T}(M)}(\opn{t}(m)) := - \opn{t}(\d_M(m)) =  \opn{t}(- \d_M(m))
\end{equation}
for $m \in M^k$. The action of $A$ is 
\begin{equation} \label{eqn:681}
a \cd \opn{t}(m) := (-1)^i \cd \opn{t}(a \cd m) = \opn{t}((-1)^i \cd a \cd m)
\end{equation}
for $a \in A^i$. 
If $M$ is a DG right module, or a DG bimodule, then the structure of 
$\opn{T}(M)$ is determined by (\ref{eqn:680}), (\ref{eqn:681}) and 
(\ref{eqn:321}). Concretely, 
\[ a \cd \opn{t}(m) \cd b = (-1)^i \cd \opn{t}(a \cd m \cd b) \]
for $a \in A^i$, $m \in M^k$ and $b \in A^j$.
Note that our sign conventions for the shift are such that the canonical 
bijection
\[ \opn{T}(\Z) \ot_{\Z} M \to \opn{T}(M) , \
\opn{t}(c) \ot m \mapsto \opn{t}(c \cd m) = c \cd \opn{t}(m) , \]
for $c \in \Z$ and $m \in M$, is an isomorphism of DG $A$-bimodules. 

For any $k \in \Z$, the $k$-th power of $\opn{T}$ is denoted by $\opn{T}^k$. 
There is a corresponding degree $-k$ homomorphism of graded abelian groups 
\begin{equation} \label{eqn:1025}
\opn{t}^k : M \to \opn{T}^k(M) , 
\end{equation}
which is the identity on the underlying ungraded abelian group. 
Note that \lb $\opn{T}^{l}(\opn{T}^k(M)) = \opn{T}^{k + l}(M)$,
and $\opn{T}^k$ is an automorphism of the category $\cat{M}(A)$.
As usual, we write $M[k] := \opn{T}^k(M)$. 

We now recall the {\em cone} construction, as it looks using the operator 
$\opn{t}$. Let $A$ be a DG ring, and let 
$\phi : M \to N$ be a homomorphism in $\cat{M}(A)$. The cone of $\phi$ is 
the DG $A$-module 
$\opn{Cone}(\phi) := N \oplus M[1]$, 
whose differential, when we express this module as the column
$\sbmat{N \\ M[1]}$, is left multiplication by the matrix of degree $1$ 
abelian group homomorphisms
$\sbmat{\d_N & \phi \circ \opn{t}^{-1} \\ 0 & \d_{M[1]}}$.

We learned the concept of {\em cylinder DG ring}, defined below, from B. 
Keller. It appears (without this name) in his \cite{Ke2}. 

Consider the matrix DG ring 
\begin{equation} \label{eqn:817}
\opn{Cyl}(\Z) := \bmat{\Z & \Z[-1] \\ 0 & \Z} 
= \Bigl\{ \bmat{a_0 & b \cd \xi \\ 0 & a_1} 
\big| \, a_0, a_1, b \in \Z \Bigr\} ,
\end{equation}
where the element $\xi := \opn{t}^{-1}(1) \in \Z[-1]$ has degree $1$. 
The multiplication is 
\begin{equation} \label{eqn:704}
\bmat{a_0 & b \cd \xi \\ 0 & a_1} \cd 
\bmat{a'_0 & b' \cd \xi \\ 0 & a'_1} := 
\bmat{a_0 \cd a'_0 & (a_0 \cd b' + b \cd a'_1) \cd \xi 
\\ 0 & a_1 \cd a'_1}
\end{equation}
where $a'_0, a'_1, b' \in \Z$. 
The differential is 
$\d_{\mrm{cyl}} := \opn{ad}( \sbmat{0 & \xi \\ 0 & 0} )$,
the graded commutator with the element
$\sbmat{0 & \xi \\ 0 & 0}$. Explicitly we have  
\begin{equation} \label{eqn:210}
\d_{\mrm{cyl}} ( \bmat{a_0 & b \cd \xi \\ 0 & a_1} ) 
= \bmat{ 0 & (-a_0 + a_1) \cd \xi \\ 0 & 0 } .
\end{equation}
There are DG ring homomorphisms 
$\Z \xar{\ep} \opn{Cyl}(\Z) \xar{\eta_i} \Z$, for $i = 0, 1$, with formulas 
\begin{equation} \label{eqn:703}
\ep(a) := \bmat{a & 0 \\ 0 & a} \quad \text{and} \quad 
\eta_i(\bmat{a_0 & b \cd \xi \\ 0 & a_1}) := a_i .
\end{equation}
It is easy to see that $\ep$ and $\eta_i$ are quasi-isomorphisms, and 
$\eta_i \circ \ep = \opn{id}_{\Z}$, the identity automorphism of $\Z$.

\begin{dfn} \label{dfn:700}
Let $A$ be a DG ring. The {\em cylinder} of $A$ is the DG ring 
\[ \opn{Cyl}(A) := \opn{Cyl}(\Z) \ot_{\Z} A 
\cong \bmat{A & A[-1] \\ 0 & A} , \] 
with multiplication induced from (\ref{eqn:704}), and differential 
induced from (\ref{eqn:210}), using  formulas (\ref{eqn:820}) and
(\ref{eqn:220}).
\end{dfn}

It will be convenient to denote elements of $\opn{Cyl}(A)$ as matrices 
$\sbmat{a_0 & \xi \cd b \\ 0 & a_0}$, rather than as tensors. 
Note that the element $\xi$ is central and has degree $1$, so 
$b \cd \xi = (-1)^k \cd \xi \cd b$ for $b \in A^k$. 

\begin{prop} \label{prop:700}
Let $A$ be a DG ring. Consider the induced DG ring homomorphisms 
$A \xar{\ep} \opn{Cyl}(A) \xar{\eta_i} A$. 
\begin{enumerate}
\item The homomorphisms $\ep$ and $\eta_i$ are DG ring quasi-isomorphisms.

\item The homomorphisms $\ep$ and $\eta_i$ are central, and moreover the induced
homomorphisms 
$\ep^{\mrm{ce}} : A^{\mrm{ce}} \xar{} \opn{Cyl}(A)^{\mrm{ce}}$
and
$\eta_i^{\mrm{ce}} : \opn{Cyl}(A)^{\mrm{ce}} \to A^{\mrm{ce}}$
are bijective. 
\end{enumerate}
\end{prop}

\begin{proof}
Since $\opn{Cyl}(\Z)$ is a free graded $\Z$-module (via $\ep$, and forgetting 
the differential), these are  direct consequences of the corresponding 
assertions for $\Z$; and those are clear.
\end{proof}

\begin{dfn} \label{dfn:815}
Let $A$ be a DG ring and $M$ a DG $A$-module. The {\em cylinder} of $M$ is 
the DG $\opn{Cyl}(A)$-module  
\[ \opn{Cyl}(M) := \opn{Cyl}(\Z) \ot_{\Z} M 
\cong \opn{Cyl}(A) \ot_{A} M
\cong \bmat{M & M[-1] \\ 0 & M} , \] 
with multiplication induced from (\ref{eqn:704}), and differential 
induced from (\ref{eqn:210}).
\end{dfn}

The definition above relates also to right DG modules and to bimodules, as 
explained in (\ref{eqn:321}). 

\begin{prop} \label{prop:701}
Let $A$ be a DG ring and let $M$ be a DG $A$-module.
The induced DG module homomorphisms $\ep : M \to \opn{Cyl}(M)$ and 
$\eta_i : \opn{Cyl}(M) \to M$ are quasi-isomorphisms. 
\end{prop}

When considering the homomorphism 
$\ep : M \to \opn{Cyl}(M)$, we view $\opn{Cyl}(M)$ as 
a DG $A$-module via the DG ring homomorphism $\ep : A \to \opn{Cyl}(A)$.
Likewise, when considering the homomorphism 
$\eta_i : \opn{Cyl}(M) \to M$, we view $M$ as 
a DG $\opn{Cyl}(A)$-module via the DG ring homomorphism 
$\eta_i : \opn{Cyl}(A) \to A$.

\begin{proof}
As in Proposition \ref{prop:700}(1), this is immediate from the fact that 
 $\ep : \Z \to \opn{Cyl}(\Z)$ and 
$\eta_i : \opn{Cyl}(\Z) \to \Z$ are quasi-isomorphisms. 
\end{proof}

The cylinder operation (on DG rings and modules) has the following 
functoriality. If $u : A \to B$ is a DG ring homomorphism, then there is an 
induced DG ring homomorphism 
$\opn{Cyl}(u) : \opn{Cyl}(A) \to \opn{Cyl}(B)$, 
$\opn{Cyl}(u) := \opn{id}_{\opn{Cyl}(\Z)} \ot \, u$. 
Explicitly the formula is
\begin{equation} \label{eqn:702}
\opn{Cyl}(u)( \bmat{a_0 & \xi \cd b \\ 0 & a_1}) =
\bmat{u(a_0) & \xi \cd u(b) \\ 0 & u(a_1)} . 
\end{equation}
Likewise for modules.

\begin{rem} \label{rem:1016}
In the abstract categorical sense, the DG ring $\opn{Cyl}(A)$
plays a role dual to that of a cylinder; so perhaps we should have called it  
``the path object of $A$''. 
Indeed, this is precisely the path object described on page 503 of \cite{ScSh}.

We decided to adhere to the name ``cylinder'' for two reasons. First, this is 
the name used in \cite{YZ2}. The second reason is that we are dealing with 
rings here, and so (as happens in algebraic geometry) arrows tend to be 
reversed. The same reversal occurs in Definition \ref{dfn:400}: we talk about 
``DG rings over $A$'', rather than about ``DG rings under $A$'', as the 
categorical convention would dictate. 
\end{rem}

\section{Resolutions of DG Modules}
\label{sec:resolutions-modules}

Let $A$ be a DG ring. 
We already mentioned that $\cat{M}(A) = \cat{DGMod} A$ is a DG category, 
equipped with a translation automorphism $\opn{T}$. The homotopy category of 
$\cat{DGMod} A$ is 
$\til{\cat{K}}(\cat{DGMod} A)$, where by definition 
\begin{equation} \label{eqn:1016}
\opn{Hom}_{\til{\cat{K}}(\cat{DGMod} A)}(M, N) := 
\mrm{H}^0 \bigl( \opn{Hom}_A(M, N) \bigr) .
\end{equation}
This is a triangulated category. 
The derived category $\til{\cat{D}}(\cat{DGMod} A)$ is  gotten
from $\til{\cat{K}}(\cat{DGMod} A)$ by inverting the 
quasi-isomorphisms. The localization functor (identity on objects) is
\begin{equation} \label{eqn:450}
\opn{Q} : \cat{DGMod} A \to \til{\cat{D}}(\cat{DGMod} A) . 
\end{equation}
If $A$ is a ring, i.e.\ $A = A^0$, with module category $\cat{Mod} A$, 
then $\cat{DGMod} A = \cat{C}(\cat{Mod} A)$,
$\til{\cat{K}}(\cat{DGMod} A) = \cat{K}(\cat{Mod} A)$
and $\til{\cat{D}}(\cat{DGMod} A) = \cat{D}(\cat{Mod} A)$.
Here are some references for the derived category of DG modules: 
\cite[Section 10]{BL}, 
\cite[Section 2]{Ke1}, and 
\cite[\href{http://stacks.math.columbia.edu/tag/09KV}{Section 09KV}]{SP}.
Another reference is our course 
notes \cite{Ye4}, but there (in Sections 4 and 9) the DG category 
$\cat{C}(\cat{M})$, associated to an abelian category $\cat{M}$, should be 
replaced by the DG category $\cat{DGMod} A$. 

\begin{dfn}   \label{dfn:875}   
We shall use the following abbreviations for the triangulated categories 
related 
to a DG ring $A$: 
$\cat{K}(A) := \til{\cat{K}}(\cat{DGMod} A)$ and 
$\cat{D}(A) := \til{\cat{D}}(\cat{DGMod} A)$. 
\end{dfn}

Let us recall a few facts about resolutions of DG $A$-modules. 
A DG module $N$ is called acyclic if $\opn{H}(N) = 0$. A DG $A$-module 
$P$ (resp.\ $I$) is called {\em K-projective} (resp.\ {\em K-injective})
if for any acyclic DG $A$-module $N$, the DG $\Z$-module 
$\opn{Hom}_{A}(P, N)$ (resp.\ $\opn{Hom}_{A}(N, I)$) is also acyclic. 
The DG module $P$ is called {\em K-flat} if for every acyclic 
DG $A^{\mrm{op}}$-module $N$, the $\Z$-module  $N \ot_A P$ is acyclic. 
As in the ``classical'' situation, K-projective implies K-flat. 
These definitions were introduced in \cite{Sp}. 
In \cite[Section 3]{Ke1} it is shown that ``K-projective'' is the same as 
``having property (P)'', and ``K-injective'' is the same as ``having property 
(I)''. Every $M \in \cat{M}(A)$ admits K-projective resolutions $P \to M$ and 
K-injective resolutions $M \to I$. 
See also \cite{BL}, \cite{BN}, \cite{AFH} and
\cite[\href{http://stacks.math.columbia.edu/tag/09JD}{Chapter 09JD}]{SP}.
(Semi-free DG modules will be discussed in Section \ref{sec:resolutions-rings}.)

Recall that a diagram  
\[ \UseTips \xymatrix @C=12ex @R=6ex {
M
\ar @/^1.5em/ [r]^{ \phi_0 }
\ar @/_1.5em/ [r]_{ \phi_1 }
&
N
} \]
in $\cat{M}(A)$ is called {\em commutative up to homotopy} if the homomorphisms 
$\phi_0$ and $\phi_1$ are homotopic. Or in other words, if the corresponding 
diagram in  $\cat{K}(A)$ is commutative. 

The next proposition is essentially a standard fact. We state it here because 
it will play a key role, especially in Section \ref{sec:pairs-DG-modules}.  

\begin{prop} \label{prop:460} 
Let $P, I \in \cat{M}(A)$. Assume either $P$ is K-projective or $I$ is 
K-injective. 
\begin{enumerate}
\item The additive homomorphism 
\[ \opn{Q} : \opn{Hom}_{\cat{M}(A)}(P, I) \to 
\opn{Hom}_{\cat{D}(A)}(P, I) \]
is surjective, and its kernel is the group 
$\opn{B}^0(\opn{Hom}_A(P, I))$ of null-homotopic homomorphisms.

\item Let $\psi : P \to I$ be a morphism in $\cat{D}(A)$. There exists a 
homomorphism $\phi : P \to I$ in $\cat{M}(A)$, unique up to homotopy, such that 
$\opn{Q}(\phi) = \psi$. 

\item Let $\phi_0, \phi_1 : P \to I$ be homomorphisms in $\cat{M}(A)$. The 
homomorphisms $\phi_0$ and $\phi_1$ are homotopic iff 
$\opn{Q}(\phi_0) = \opn{Q}(\phi_1)$ in $\cat{D}(A)$. 
\end{enumerate}
\end{prop}

In other words, item (2) says that given a morphism $\psi$ in $\cat{D}(A)$, 
as in the first diagram below, there exists a homomorphism $\phi$ in 
$\cat{M}(A)$ making that diagram commutative; and such $\phi$ is unique up to 
homotopy. Item (3) says that the second diagram below, in $\cat{M}(A)$, is 
commutative up to homotopy, iff the third diagram below, in $\cat{D}(A)$, is 
commutative. 
\[ \UseTips \xymatrix @C=12ex @R=6ex {
P
\ar @/^1.5em/ [r]^{ \psi }
\ar@{-->} @/_1.5em/ [r]_{ \opn{Q}(\phi) }
&
I
}
\qquad \qquad 
\UseTips \xymatrix @C=10ex @R=6ex {
P
\ar @/^1.5em/ [r]^{ \phi_0 }
\ar @/_1.5em/ [r]_{ \phi_1 }
&
I
}
\qquad 
\UseTips \xymatrix @C=10ex @R=6ex {
P
\ar @/^1.5em/ [r]^{ \opn{Q}(\phi_0) }
\ar @/_1.5em/ [r]_{ \opn{Q}(\phi_1) }
&
I
} \]

\begin{proof}
(1) According to \cite[Propositions 1.4 and 1.5]{Sp}, the additive homomorphism 
\[ \opn{Hom}_{\cat{K}(A)}(P, I) = \opn{H}^0(\opn{Hom}_A(P, I)) \to 
\opn{Hom}_{\cat{D}(A)}(P, I) \]
is bijective. Note that in \cite{Sp} $A$ is a ring; but the definitions and 
proofs pass without change to the case of a DG ring.

\medskip \noindent 
(2, 3) These are immediate from (1). 
\end{proof}

A DG ring homomorphism $u : A \to B$ induces the forgetful (or restriction of 
scalars) functor 
$\opn{For}_{u} : \cat{M}(B) \to \cat{M}(A)$, 
which is exact, and hence gives rise to a triangulated functor
$\opn{For}_{u} : \cat{D} (B) \to \cat{D} (A)$. 
We usually do not mention the functor $\opn{For}_{u}$ explicitly, unless it is 
important for the discussion at hand. 

If $u : A \to B$ is a quasi-isomorphism then 
$\opn{For}_{u} : \cat{D} (B) \to \cat{D} (A)$ is an equivalence, with 
quasi-inverse  $M \mapsto B \ot^{\mrm{L}}_A M$.
For any $M_0 \in \cat{D}(A^{\mrm{op}})$ and $M_1, M_2 \in \cat{D} (A)$
there are bi-functorial isomorphisms 
$M_0 \ot^{\mrm{L}}_A M_1 \cong M_0 \ot^{\mrm{L}}_B M_1$ and 
$\opn{RHom}_A(M_1, M_2) \cong \opn{RHom}_B(M_1, M_2)$
in $\cat{D}(\Z)$; cf.\ \cite[Proposition 1.4]{YZ2}.

Here is a result that will be used a lot in our paper. The restriction 
functors are suppressed in it. The easy proof is left to the reader. 

\begin{prop} \label{prop:461} 
Let $A_0 \xar{u_1} A_1 \xar{u_2} A_2$ and 
$B_0 \xar{v_1} B_1 \xar{v_2} B_2$ 
be homomorphisms in $\cat{DGR}$.
For $k = 0, 1, 2$ let $M_k \in \cat{M}(B_k \ot_{\Z} A_k)$ and
$N_k \in \cat{M}(B_k)$; and for $k = 1, 2$ let 
$\ze_k : M_{k - 1} \to M_k$ and $\th_k : N_k \to N_{k - 1}$ 
be homomorphisms in $\cat{M}(B_{k - 1} \ot_{\Z} A_{k - 1})$ and 
$\cat{M}(B_{k - 1})$ respectively.
\begin{enumerate}
\item The abelian group $\opn{Hom}_{B_0}(M_0, N_0)$ is a DG 
$A_0^{\mrm{op}}$-module,  with action 
\[ (\phi \cd a)(m) :=  \phi(a \cd m) \]
for $\phi \in \opn{Hom}_{B_0}(M_0, N_0)$, $a \in A_0$  
and $m \in M_0$.

\item The function 
\[ \opn{Hom}_{v_1}(\ze_1, \th_1) : \opn{Hom}_{B_1}(M_1, N_1) \to
\opn{Hom}_{B_0}(M_0, N_0) , \]
with formula 
\[ \opn{Hom}_{v_1}(\ze_1, \th_1)(\phi) := \th_1 \circ \phi \circ \ze_1 \]
for $\phi \in \opn{Hom}_{B_1}(M_1, N_1)$,
is a homomorphism in $\cat{M}(A_0^{\mrm{op}})$.

\item There is equality 
\[ \opn{Hom}_{v_1}(\ze_1, \th_1) \circ \opn{Hom}_{v_2}(\ze_2, \th_2) =
\opn{Hom}_{v_2 \circ v_1}(\ze_2 \circ \ze_1, \th_1 \circ \th_2) . \]

\item Say 
$\ze'_1 : M_{0} \to M_1$ and $\th'_1 : N_1 \to N_{0}$ are also homomorphisms in 
\lb $\cat{M}(B_0 \ot_{\Z} A_0)$ and $\cat{M}(B_{0})$ respectively, and there 
are homotopies 
$\ze_1 \twoto \ze'_1$ and $\th_1 \twoto \th'_1$ in
$\cat{M}(B_0 \ot_{\Z} A_0)$ and $\cat{M}(B_{0})$ respectively.
Then there is a homotopy 
\[ \opn{Hom}_{v_1}(\ze_1, \th_1) \twoto \opn{Hom}_{v_1}(\ze'_1, \th'_1) \]
in $\cat{M}(A_0^{\mrm{op}})$.
\end{enumerate}
\end{prop}

Item (2) of the proposition is illustrated here:

\[ \UseTips \xymatrix @C=10ex @R=6ex {
M_0
\ar[r]^{\ze_1}
\ar[d]_{\opn{Hom}_{v_1}(\ze_1, \th_1)(\phi)}
&
M_1
\ar[d]^{\phi}
\\
N_0
&
N_1
\ar[l]_{\th_1}
} \]

\begin{prop} \label{prop:462}
Let $v : A' \to A$ be a quasi-isomorphism between DG rings.
Let $Q, I \in \cat{M}(A)$, $P \in \cat{M}(A^{\mrm{op}})$, 
$Q', I' \in \cat{M}(A')$ and $P' \in \cat{M}(A'^{\, \mrm{op}})$.
Let $\phi : P' \to P$ be a quasi-isomorphism in  
$\cat{M}(A'^{\, \mrm{op}})$, and let 
$\psi : Q' \to Q$ and $\chi : I \to I'$ be quasi-isomorphisms in  $\cat{M}(A')$.
\begin{enumerate}
\item If \tup{((}$P$ is K-flat over $A^{\mrm{op}}$ or $Q$ is K-flat over 
$A$\tup{)} and
\tup{(}$P'$ is K-flat over $A'^{\mrm{\, op}}$ or $Q'$ is K-flat over 
$A'$\tup{))} then 
\[ \phi \ot_v \psi : P' \ot_{A'} Q' \to P \ot_A Q \] 
is a quasi-isomorphism.

\item If \tup{((}$Q$ is K-projective or $I$ is K-injective over $A$\tup{)} 
and
\tup{(}$Q'$ is K-projective or $I'$ is K-injective over $A'$\tup{))} then
\[ \opn{Hom}_v(\psi, \chi) : \opn{Hom}_A(Q, I) \to \opn{Hom}_{A'}(Q', I') \]
is a quasi-isomorphism.
\end{enumerate}
\end{prop}

\begin{proof}
This is implicit in the proof of \cite[Proposition 1.4]{YZ2}. Here is a 
detailed proof. 

\medskip \noindent 
(1) By the symmetry of the situation we can assume that either $P'$ and $P$ 
are K-flat (over the respective DG rings), or that $Q'$ and $P$ are
K-flat. We begin by reducing the second case to the first case. Choose a K-flat 
resolution $\mu : \til{P}' \to P'$. Then 
$\til{P}' \ot_{A'} Q' \to P' \ot_{A'} Q'$ is a quasi-isomorphism, and it 
suffices to prove that 
\[ (\phi \circ \mu) \ot_v \psi : \til{P}' \ot_{A'} Q' \to P \ot_A Q \]
is a quasi-isomorphism, which is the first case.

Now we assume that  $P'$ and $P$ are K-flat. Consider the homomorphism 
$\phi \cd \opn{id}_A : P' \ot_{A'} A \to P$, 
$(\phi \cd \opn{id}_A)(p' \ot a)  := \phi(p') \cd a$. 
We claim this is a quasi-isomorphism. To see why, let us look at the first 
commutative diagram below. The homomorphism $\opn{id}_{P'} \ot v$ is a 
quasi-isomorphism since $v$ is a quasi-isomorphism and $P'$ is K-flat; 
the homomorphism $\opn{id}_{P'} \cd \opn{id}_{A'}$ is an isomorphism; and 
$\phi$ is a quasi-isomorphism; therefore $\phi \cd \opn{id}_A$ 
is a quasi-isomorphism.
\[ \UseTips \xymatrix @C=10ex @R=6ex {
P' \ot_{A'} A' 
\ar[r]^(0.6){ \opn{id}_{P'} \cd \opn{id}_{A'} }
\ar[d]_{ \opn{id}_{P'} \ot v }
&
P'
\ar[d]^{\phi}
\\
P' \ot_{A'} A
\ar[r]^(0.6){\phi \cd \opn{id}_{A}  }
&
P 
} 
\qquad 
\UseTips \xymatrix @C=8ex @R=6ex {
P' \ot_{A'} Q'
\ar[r]^{ \phi \ot_v \psi }
\ar[d]_{ \opn{id}_{P'} \ot \psi }
&
P \ot_A Q 
\\
P' \ot_{A'} Q
\ar[r]^(0.45){ \lambda }
&
P' \ot_{A'} A \ot_A Q 
\ar[u]_{ (\phi \cd \opn{id}_{A}) \ot \opn{id}_{Q} }
} \]

Next look at the second commutative diagram above.
The homomorphism  $\lambda$, defined by
$\lambda(p' \ot q) := p' \ot 1 \ot q$, is an isomorphism.
The homomorphism $\opn{id}_{P'} \ot \psi$ is a quasi-isomorphism because 
$\psi$ is a  quasi-isomorphism and $P'$ is K-flat; 
and $(\phi \cd \opn{id}_{A}) \ot \opn{id}_{Q}$ is a quasi-isomorphism because 
$\phi \cd \opn{id}_{A}$ a quasi-isomorphism between K-flat DG 
$A^{\mrm{op}}$-modules. Therefore $\phi \ot_v \psi$ is a quasi-isomorphism.

\medskip \noindent 
(2) Choose K-projective resolutions 
$\al: \til{Q} \to Q$ and $\al' : \til{Q}' \to Q'$,
and K-injective resolutions 
$\be : I \to \til{I}$ and $\be' : I' \to \til{I}'$,
in $\cat{M}(A)$ and $\cat{M}(A')$ respectively. 
By Proposition \ref{prop:460}, we can lift $\psi$ and $\chi$ 
to quasi-isomorphisms $\til{\psi}$ and $\til{\chi}$, such that the diagrams 
\[ \UseTips \xymatrix @C=6ex @R=6ex {
\til{Q}'
\ar[d]_{\al'}
\ar[r]^{\til{\psi}}
&
\til{Q}
\ar[d]^{\al}
\\
Q'
\ar[r]^{\psi}
&
Q
} 
\qquad \qquad 
\UseTips \xymatrix @C=8ex @R=6ex {
\til{I}
\ar[r]^{\til{\chi}}
&
\til{I}'
\\
I
\ar[u]^{\be}
\ar[r]^{\chi}
&
I'
\ar[u]_{\be'}
} \]
in $\cat{M}(A')$ are commutative up to homotopy. We get a diagram 
\begin{equation} \label{eqn:970}
\UseTips \xymatrix @C=14ex @R=6ex {
\opn{Hom}_A(Q, I)
\ar[r]^{\opn{Hom}_v(\psi, \chi)}
\ar[d]_{ \opn{Hom}_{\mrm{id}}(\al, \be) }
&
\opn{Hom}_{A'}(Q', I')
\ar[d]^ {\opn{Hom}_{\mrm{id}}(\al', \be') }
\\
\opn{Hom}_A(\til{Q}, \til{I})
\ar[r]^{\opn{Hom}_v(\til{\psi}, \til{\chi})}
&
\opn{Hom}_{A'}(\til{Q}', \til{I}')
} 
\end{equation}
in $\cat{M}(\Z)$ that is commutative up to homotopy. 
Now let us examine the commutative diagram 
\[ \UseTips \xymatrix @C=14ex @R=8ex {
\opn{Hom}_A(Q, I)
\ar[r]^{ \opn{Hom}_{\mrm{id}}(\al, \mrm{id}_I) }
\ar[d]_{ \opn{Hom}_{\mrm{id}}(\mrm{id}_Q, \be) }
\ar[dr]^{ \ \opn{Hom}_{\mrm{id}}(\al, \be) }
&
\opn{Hom}_{A}(\til{Q}, I)
\ar[d]^{ \opn{Hom}_{\mrm{id}}(\mrm{id}_{\til{Q}}, \be) }
\\
\opn{Hom}_A(Q, \til{I})
\ar[r]^{ \opn{Hom}_{\mrm{id}}(\al, \mrm{id}_{\til{I}}) }
&
\opn{Hom}_A(\til{Q}, \til{I})
} 
\]
in $\cat{M}(\Z)$.
The homomorphisms 
$\opn{Hom}_{\mrm{id}}(\al, \mrm{id}_{\til{I}})$ 
and 
$\opn{Hom}_{\mrm{id}}(\mrm{id}_{\til{Q}}, \be)$
are both quasi-isomorphisms. At least one of the homomorphisms
$\opn{Hom}_{\mrm{id}}(\mrm{id}_Q, \be)$
and 
$\opn{Hom}_{\mrm{id}}(\al, \mrm{id}_I)$
is a quasi-isomorphism. It follows that the diagonal homomorphism 
$\opn{Hom}_{\mrm{id}}(\al, \be)$ is a quasi-isomorphism. 
A similar argument shows that 
$\opn{Hom}_{\mrm{id}}(\al', \be')$ is a quasi-isomorphism. Going back to 
diagram (\ref{eqn:970}), we see that it is enough to prove that 
$\opn{Hom}_v(\til{\psi}, \til{\chi})$
is a quasi-isomorphism.

Let 
$\la : \til{I} \to \opn{Hom}_{A'}(A, \til{I}')$
be the homomorphism 
$\la(i)(a) := \til{\chi}(a \cd i)$
in $\cat{M}(A)$.
It fits into a commutative diagram 
\[ \UseTips \xymatrix @C=6ex @R=6ex {
\til{I}
\ar[r]^(0.3){\la}
\ar[d]_{\til{\chi}}
& 
\opn{Hom}_{A'}(A, \til{I}')
\ar[d]^{ \opn{Hom}_{\mrm{id}}(v, \mrm{id}) }
\\
\til{I}'
& 
\opn{Hom}_{A'}(A', \til{I}')
\ar[l]_(0.7){ \opn{can} }
} 
\]
in $\cat{M}(A')$, where $\opn{can}$ is the canonical isomorphism. 
Because $\til{\chi}$ and $v$ are quasi-isomorphisms, and $\til{I}'$ is 
K-injective, we see that $\la$ is a quasi-isomorphism. 

Finally, consider the commutative diagram 
\[ \UseTips \xymatrix @C=8ex @R=6ex {
\opn{Hom}_A(\til{Q}, \til{I})
\ar[r]^{ \opn{Hom}_v(\til{\psi}, \til{\chi}) }
\ar[d]_{ \opn{Hom}_{\mrm{id}}(\mrm{id}, \la) }
&
\opn{Hom}_{A'}(\til{Q}', \til{I}')
\\
\opn{Hom}_A(\til{Q}, \opn{Hom}_{A'}(A, \til{I}'))
\ar[r]^(0.6){ \opn{adj} }
&
\opn{Hom}_{A'}(\til{Q}, \til{I}')
\ar[u]_{ \opn{Hom}_{\mrm{id}}(\til{\psi}, \mrm{id}) }
} 
\]
in $\cat{M}(\Z)$, where $\opn{adj}$ is the adjunction isomorphism. 
Since $\la$ and $\til{\psi}$ are quasi-isomorphisms, $\til{Q}$ is K-projective, 
and  $\til{I}'$ is K-injective, it follows that 
$\opn{Hom}_v(\til{\psi}, \til{\chi})$
is a quasi-isomorphism.
\end{proof}

\section{Central Pairs of DG Rings and their Resolutions}
\label{sec:resolutions-rings}

In this section we introduce several special kinds of DG rings and 
homomorphisms. Recall that $\cat{DGR}$ is the category of DG rings. 

\begin{dfn}
A DG ring $A = \bigoplus_{i \in \Z} A^i$ is called {\em nonpositive} if 
$A^i = 0$ for all $i > 0$. 
The full subcategory of $\cat{DGR}$ consisting of nonpositive DG rings is 
denoted by $\cat{DGR}^{\leq 0}$.
\end{dfn}

\begin{dfn} \label{dfn:330}
A DG ring $A = \bigoplus_{i \in \Z} A^i$ is called {\em strongly commutative} 
if it satisfies these two conditions:
\begin{enumerate}
\rmitem{i} $A$ is weakly commutative (Definition \ref{dfn:750}(2)), namely 
$b \cd a = (-1)^{i j} \cd a \cd b$ for all $a \in A^{i}$ and $b \in A^{j}$. 

\rmitem{ii} $a \cd a = 0$ for all $a \in A^i$ and $i$ odd. 
\end{enumerate}
We denote by $\cat{DGR}_{\mrm{sc}}$ the full subcategory of 
$\cat{DGR}$ on the strongly commutative DG rings.
\end{dfn}

\begin{dfn} \label{dfn:876}
A DG ring $A$ is called a {\em commutative DG ring} if it is both nonpositive 
and strongly commutative. The full subcategory of $\cat{DGR}$ consisting of 
commutative DG rings is denoted by $\cat{DGR}_{\mrm{sc}}^{\leq 0}$. 
\end{dfn}

In other words, 
\[ \cat{DGR}_{\mrm{sc}}^{\leq 0} = \cat{DGR}^{\leq 0} \cap
\cat{DGR}_{\mrm{sc}} . \]

\begin{exa}
A commutative DG ring $A$ concentrated in degree $0$ (i.e.\ $A = A^0$) is just 
a commutative ring. 
\end{exa}

\begin{rem}
The name ``strongly commutative DG ring'' was suggested to us by J. 
Palmieri. In an earlier version of this paper we used the name 
``strictly commutative DG ring'', following \cite{ML}; and in \cite{YZ2}
it was ``super-commutative DG algebra''. 

The name ``weakly commutative DG ring'' is ``commutative DG algebra'' in 
\cite{ML}, and ``graded-commutative DG algebra'' in \cite{AILN}. 

Of course when the number $2$ is invertible in $A^0$ (e.g.\ if $A^0$  contains 
$\Q$), then condition (i) of Definition \ref{dfn:330} implies condition (ii), 
so that the distinction between weak and strong commutativity disappears. See 
also Remark \ref{rem:750}.
\end{rem}

Central homomorphisms were introduced in Definition \ref{dfn:750}(3). 

\begin{dfn} \label{dfn:910}
A {\em central pair of DG rings} is a central homomorphism 
$u : A \to B$ in $\cat{DGR}$, where the DG ring $A$ is commutative. 
We usually denote this pair by $B / A$, keeping $u$ implicit. 

Suppose $A \xar{u} B$ and $A' \xar{u'} B'$
are central pairs of DG rings. A {\em morphism of central pairs} 
$w / v : B' / A' \to B / A$
consists of homomorphisms 
$v : A' \to A$ and $w : B' \to B$
in $\cat{DGR}$, such that 
$u \circ v = w \circ u'$. 
The resulting category is called the {\em category of central 
pairs of DG rings}, and is denoted by $\cat{PDGR}$. 
\end{dfn}

A morphism $w / v$ in $\cat{PDGR}$ is shown in diagram (\ref{eqn:985}). 
We do not require the homomorphism $w : B' \to B$ to be 
central. The homomorphism $v : A' \to A$ is automatically central, because $A$ 
is commutative. 

\begin{dfn} \label{dfn:911}
\mbox{}
\begin{enumerate}
\item A {\em nonpositive central pair} is a central pair $B / A$ such that both 
$A$ and $B$ are nonpositive DG rings. The full subcategory of nonpositive 
central pairs is denoted by $\cat{PDGR}^{\leq 0}$.

\item A {\em commutative central pair} is a central pair $B / A$ such that the 
DG ring $B$ is also commutative. The full subcategory of commutative 
central pairs is denoted by $\cat{PDGR}_{\mrm{sc}}^{\leq 0}$.

\item A morphism $w / v : B' / A' \to B / A$ in $\cat{PDGR}$
is called {\em central} if the homomorphism $w : B' \to B$ is central.
The subcategory of $\cat{PDGR}$ on all objects, but only with central 
homomorphisms, is denoted by $\cat{PDGR}_{\vec{\mrm{ce}}}$.

\item A morphism $w / v : B' / A' \to B / A$ in $\cat{PDGR}$
is called a {\em quasi-iso\-morphism} if the homomorphisms 
$v : A' \to A$ and $w : B' \to B$ are both quasi-isomorphisms.
\end{enumerate}
\end{dfn}

\begin{dfn} \label{dfn:912}
Let $B / A$ be a central pair of DG rings. 
A {\em resolution} of $B / A$ is a morphism 
$s / r : \til{B} / \til{A} \to B / A$ in $\cat{PDGR}$,
such that:
\begin{itemize}
\item $r : \til{A} \to A$ is a quasi-isomorphism.
\item $s : \til{B} \to B$ is a surjective quasi-isomorphism.
\end{itemize}

The resolution is called {\em strict} if $\til{A} = A$ and $r = \opn{id}_A$, 
the identity automorphism of $A$.
\end{dfn}

See diagram (\ref{eqn:482}) for an illustration. 

\begin{dfn} \label{dfn:913}
Let $w / v : B' / A' \to B / A$ be a morphism in $\cat{PDGR}$,
let 
$s / r : \til{B} / \til{A} \to B / A$
be a resolution of $B / A$, and let 
$s' / r' : \til{B}' / \til{A}' \to B' / A'$
be a resolution of $B' / A'$. A morphism 
$\til{w} / \til{v} : \til{B}' / \til{A}' \to \til{B} / \til{A}$
in $\cat{PDGR}$ is said to be {\em morphism of resolutions above $w / v$}
if 
\[ (w / v) \circ (s' / r') = (s / r) \circ (\til{w} / \til{v}) \]
in $\cat{PDGR}$.
In this case, we also say that $\til{w} / \til{v}$ is a {\em resolution of 
$w / v$}.

If $B' / A' = B / A$ and $w / v$ is the identity automorphism, then 
we say that $\til{w} / \til{v}$ is a morphism of resolutions of $B / A$. 
\end{dfn}

A morphism $\til{w} / \til{v}$ is shown in diagram (\ref{eqn:986}).

By {\em graded set} we mean a set $X$ 
with a partition $X = \coprod_{i \in \Z} X^i$; the elements of $X^i$ are 
said to have degree $i$. We say that $X$ is a {\em nonpositive graded set} if 
$X^i = \varnothing$ for all $i > 0$. 
A {\em filtered graded set} is a graded set $X$, with 
an ascending filtration 
$F = \{ F_j(X) \}_{j \in \Z}$
by graded subsets, such that $X = \bigcup_j F_j(X)$ and 
$F_{-1}(X) = \varnothing$. 

Let $A$ be a DG ring. We denote by $A^{\natural}$ the graded ring gotten from 
$A$ by forgetting the differential. The same for DG modules. 
When talking about free modules, we assume the DG ring $A$ is nonzero (to avoid 
nonsense).

For a variable $x$ of degree $i$, we let 
$A \ot x := A[-i]$, the shift by $-i$ of the DG $A$-module $A$.
For any element $a \in A^j$ we denote by $a \ot x$ or $a \cd x$ 
the element $\opn{t}^{-i}(a) \in A[-i]^{j + i}$; cf.\ (\ref{eqn:1025}).
Note that $\d(x) = 0$, and $\d(a \cd x) = \d(a) \cd x$ for any $a \in A$.  
Now consider a graded set of variables $X$. The {\em free DG $A$-module on $X$} 
is 
\[ A \ot X := \bigoplus_{x \in X} \ A \ot x . \]
A {\em free DG $A$-module} is a DG $A$-module $P$ that is isomorphic to 
$A \ot X$ for some graded set of variables $X$. The image of $X$ in $P$ is 
called a {\em basis} of $P$. 

Let $P$ be a DG $A$-module. A {\em semi-free filtration} on $P$ 
is an ascending filtration 
$F = \{ F_j(P) \}_{j \in \Z}$
of $P$ by DG submodules, such that $P = \bigcup_j F_j(P)$, 
$F_{-1}(P) = 0$, and each 
$\opn{gr}^F_j(P) := F_{j}(P) / F_{j - 1}(P)$ 
is a free DG $A$-module. 
We say that $P$ is a {\em semi-free DG $A$-module} if it admits some semi-free 
filtration. See \cite{AFH} and \cite{Hi}.
A semi-free DG module is K-projective, and hence also K-flat. 

Let $P$ be a semi-free DG $A$-module, with semi-free filtration $F$. 
For each $j \geq 0$ choose a graded subset $Y_j \subseteq F_j(P)$
whose image $\bar{Y}_j$
in $\opn{gr}^F_j(P)$ is a basis of this free DG $A$-module.
The  graded set  
$X := \coprod_{i \in \Z} \, Y_j$
is filtered by 
$F_j(X) := \bigcup_{j' \leq j} Y_{j'}$. 
Such a filtered graded set is called a {\em semi-basis of $P$}.
Note that for every $j$ there is an isomorphism of graded 
$A^{\natural}$-modules 
$A^{\natural} \ot F_j(X) \cong F_j(P)^{\natural}$;
and so $A^{\natural} \ot X \cong P^{\natural}$. 
For any element $y \in Y_j$, its image
$\bar{y} \in \bar{Y}_j \subseteq \opn{gr}^F_j(P)$
satisfies $\d(\bar{y}) = 0$; and hence 
$\d(y) \in F_{j - 1}(X)$. 
We see that 
$\d(F_{j}(X)) \subseteq F_{j - 1}(X)$. 

Given a graded set of variables $X$, we can form the 
noncommutative polynomial ring $\Z \bra{X}$ over $\Z$, which is the free 
$\Z$-module on the set of monomials $x_1 \cdots x_m$ in the elements of $X$,
and with the obvious multiplication and grading. 

The next definition is not standard. Its first occurrence seems to have  been 
in \cite{YZ2}. 

\begin{dfn}  \label{dfn:751}
Let $X = \coprod_{i \in \Z} X^i$ be a graded set of variables.
The {\em strongly commutative polynomial ring} $\Z[X]$ is the quotient of 
$\Z \bra{X}$ by the two-sided ideal generated by the elements 
$y \cd x - (-1)^{i j} \cd x \cd y$ and $z^2$, for all $x \in X^{i}$, 
$y \in X^{j}$ and $z \in X^k$, with $i, j, k \in \Z$ and $k$ odd. 
\end{dfn}

\begin{dfn} \label{dfn:60}
Consider a nonpositive central pair of DG rings $A \xar{u} B$. 
\begin{enumerate}
\item We say that $B$ is a {\em commutative semi-free DG ring over $A$},  
that $u$ is a {\em commutative semi-free DG ring homomorphism}, and that 
$B / A$ is a {\em commutative semi-free pair of DG rings}, if there is 
an isomorphism of graded $A^{\natural}$-rings
$B^{\natural} \cong A^{\natural} \ot_{\Z} \Z[X]$,
where $X$ is a nonpositive graded set, and $\Z[X]$ is the strongly commutative 
polynomial ring. Such a graded set $X$ is called a {\em set of commutative 
semi-free ring generators} of $B$ over $A$. 

\item We say that $B$ is a {\em noncommutative semi-free DG ring over $A$}, 
that $u$ is a {\em noncommutative semi-free DG ring homomorphism}, and that 
$B / A$ is a {\em noncommutative semi-free pair of DG rings}, if there 
is an isomorphism of graded $A^{\natural}$-rings
$B^{\natural} \cong A^{\natural} \ot_{\Z} \Z \bra{X}$,
where $X$ is a nonpositive graded set, and $\Z\bra{X}$ is the noncommutative 
polynomial ring. Such a graded set $X$ is called a {\em set of noncommutative 
semi-free ring generators} of $B$ over $A$. 

\item  We say that  $B$ is a {\em K-flat DG ring over $A$}, that 
$u$ is a {\em K-flat DG ring homomorphism}, and that 
$B / A$ is a {\em K-flat pair of DG rings}, if $B$ is K-flat
as a DG $A$-module (left or right -- it does not matter). 
\end{enumerate}
\end{dfn}

\begin{prop} \label{prop:750}
Let $u : A \to B$ a nonpositive central pair of DG rings, which is either 
commutative semi-free or a noncommutative semi-free. Then $u : A \to B$ is 
K-flat. \end{prop}

\begin{proof}
In the noncommutative case, let $X$ be a set of noncommutative semi-free ring 
generators of $B$ over $A$. Then the monomials $x_1 \cdots x_m$, 
with  $x_1, \ldots, x_m \in X$, are a semi-basis for the DG $A$-module $B$.
We are using the fact that the graded set $X$ is nonpositive. 
 
In the commutative case, let $X$ be a set of commutative semi-free ring 
generators of $B$ over $A$. Choose some ordering of the set $X$. 
Then the monomials $x_1 \cdots x_m$, with  $x_1 \leq \cdots \leq x_m$
in $X$, and $x_i < x_{i + 1}$ if $x_i$ is odd, are a semi-basis for the DG 
$A$-module $B$. Again, we are using the fact that the graded set $X$ is 
nonpositive. 

We see that in both cases $B$ is a semi-free DG $A$-module, and thus it is 
K-flat. 
\end{proof}

\begin{exa} \label{exa:690}
Let $A$ be a commutative ring, and let $\bsym{a} = (a_1, \ldots, a_n)$ be a 
sequence of elements in $A$. The Koszul complex 
$B := \opn{K}(A; \bsym{a})$ associated to this sequence is a commutative
semi-free DG ring over $A$. It has a set of commutative 
semi-free ring generators $X = \{ x_1, \ldots, x_n \}$ of degree $-1$,
such that $\d_B(x_i) = a_i$.   
\end{exa}

\begin{rem} \label{rem:750}
The strongly commutative polynomial ring $\Z[X]$ is a free graded $\Z$-module. 
This is what makes Proposition \ref{prop:750} work. 

The ``weakly commutative polynomial ring'', which is the quotient of 
$\Z \bra{X}$ by the two-sided ideal generated by the elements 
$y \cd x - (-1)^{i j} \cd x \cd y$ for all $x \in X^{i}$ and $y \in X^{j}$,
{\em is not flat over $\Z$}. Indeed, for any odd element $z \in X$ we have 
$z^2 \neq 0$ but $2 \cd z^2 = 0$. 
\end{rem}

\begin{dfn} \label{dfn:52}
Let $B / A$ be a central pair of DG rings, and let 
$s / r : \til{B} / \til{A} \to B / A$ be a resolution of $B / A$
(Definition \ref{dfn:912}). 
\begin{enumerate}
\item Assume $B / A$ is a commutative pair. We say that $\til{B} / \til{A}$
is a {\em commutative semi-free resolution} of $B / A$ if the pair 
$\til{B} / \til{A}$ is commutative semi-free (Definition \ref{dfn:60}(1)). 

\item We say that $\til{B} / \til{A}$ is a {\em noncommutative semi-free 
resolution} of $B / A$ if the pair 
$\til{B} / \til{A}$ is noncommutative semi-free (Definition \ref{dfn:60}(2)).

\item We say that $\til{B} / \til{A}$ is a {\em K-flat resolution} of $B / A$ 
if the pair $\til{B} / \til{A}$ is K-flat (Definition \ref{dfn:60}(3)).
\end{enumerate}
\end{dfn}

\begin{exa}
If $B$ is K-flat over $A$, then $B / A$ is a K-flat ring resolution of itself. 
It is terminal among all K-flat ring resolutions of $B / A$.
\end{exa}

\begin{exa}
Let $u : A \to B$ be a surjective homomorphism of commutative rings, and let 
$\bsym{a} = (a_1, \ldots, a_n)$ be a sequence of elements of $A$ that generates 
the ideal $\opn{Ker}(u)$. Assume that the sequence $\bsym{a}$ is regular. 
Define $\til{B} := \opn{K}(A; \bsym{a})$, the Koszul complex. 
Then $\til{B} / A$ is a strict commutative semi-free DG ring 
resolution of $B / A$. 
\end{exa}

\begin{rem}
In \cite{YZ2} we did not require the quasi-isomorphism $w$ to be surjective. 
That omission was not significant there, because we were mostly concerned with 
the case when $B$ is a ring (cf.\ \cite[Proposition 1.8]{YZ2}). When $B$ has 
nonzero negative components, the surjectivity is needed (cf.\ proof of Theorem 
\ref{thm:305} below).
\end{rem}

The next results in this section are enhanced versions of 
\cite[Propositions 1.7 and 1.8]{YZ2}.
The category $\cat{DGR}^{\leq 0} \centover A$ was introduced in Definition 
\ref{dfn:750}(4). 

\begin{lem} \label{lem:305}
Let $A$ be a commutative DG ring, and let 
$B, B' \in  \cat{DGR}^{\leq 0} \centover A$. 
Assume either of these two conditions holds\tup{:}
\begin{enumerate}
\rmitem{i} $B$ is commutative, and $B'$ is commutative semi-free over $A$, 
with commutative semi-free ring generating set $X$. 

\rmitem{ii} $B'$ is noncommutative semi-free over $A$, with noncommutative 
semi-free ring generating set $X$. 
\end{enumerate}
Let $w : X \to B$ be a degree $0$ function. Then\tup{:}
\begin{enumerate}
\item The function $w$ extends uniquely to a graded $A^{\natural}$-ring 
homomorphism $w : B'^{\, \natural} \to B^{\natural}$.

\item The graded ring homomorphism $w : B'^{\, \natural} \to B^{\natural}$
extends to a DG ring homomorphism $w : B' \to B$ iff 
$\d(w(x)) = w(\d(x))$ for all $x \in X$. 
\end{enumerate}
\end{lem}

\begin{proof}
This is an easy exercise. 
\end{proof}

\begin{lem} \label{lem:306}
Let $X$ be a nonpositive graded set, and let $A$ be a commutative DG 
ring. Consider the graded ring $B := A^{\natural} \ot_{\Z} \Z[X]$ or 
$B := A^{\natural} \ot_{\Z} \Z\bra{X}$. 
\begin{enumerate}
\item Let $\d_X : X \to B$ be a degree $1$ function. The function $\d_X$ 
extends uniquely to a degree $1$ derivation $\d_B : B \to B$, such 
that $\d_B|_A = \d_A$, the differential of $A$. 

\item The derivation $\d_B$ is a differential on $B$, i.e.\ 
$\d_B \circ \d_B = 0$, iff  
$(\d_B \circ \d_B)(x) = 0$ for all $x \in X$. 
\end{enumerate}
\end{lem}

\begin{proof}
(1) Take elements $x_j \in X^{n_j}$ for $j = 1, \ldots, m$, and consider the 
monomial $x_1 \cdots x_m \in \Z[X]$ (resp.\ in $\Z \bra{X}$). 
We let 
\[ \d_B(x_1 \cdots x_m) := 
\sum_{j = 1}^m (-1)^{n_1 + \cdots n_{j-1}} \cd x_1 \cdots x_{j-1} \cd 
\d_X(x_{j}) \cd x_{j+1} \cdots x_m \in B .  \]
This is well-defined in the commutative case (and there is nothing to check in 
the noncommutative case). Extending it additively we obtain a degree $1$ 
$\Z$-linear homomorphism 
$\Z[X] \to B$ (resp.\ $ \Z \bra{X} \to B$). This extends to 
a degree $1$ $A$-linear homomorphism $\d_B : B \to B$
that satisfies the graded Leibniz rule. 

\medskip \noindent
(2) This is an easy exercise. 
\end{proof}

For a DG ring $A$ we denote by $\opn{Z}(A) = \bigoplus_i \opn{Z}^i(A)$ 
the set of cocycles, and by 
$\opn{B}(A) = \bigoplus_i \opn{B}^i(A)$ 
the set of coboundaries (i.e.\ the kernel and image of $\d_A$, 
respectively).
So $\opn{Z}(A)$ a DG subring of $A$ (with zero differential of course),
and $\opn{B}(A)$ is a two-sided graded ideal of $\opn{Z}(A)$.

\begin{thm} \label{thm:305}
Let $A \xar{u} B$ be a nonpositive central pair of DG rings.
\begin{enumerate}
\item If $B$ is commutative, then there exists a strict commutative semi-free 
DG ring resolution 
$A \xar{\til{u}} \til{B}$ of $A \xar{u}  B$.

\item There exists a strict noncommutative semi-free DG ring resolution 
$A \xar{\til{u}} \til{B}$ of $A \xar{u}  B$.
\end{enumerate} 
\end{thm}

\begin{proof}
(1) We are looking for a commutative semi-free DG $A$-ring $\til{B}$,
with a quasi-isomorphism $v : \til{B} \to B$ in 
$\cat{DGR}_{\mrm{sc}}^{\leq 0} / A$. 
We shall construct an increasing sequence of 
commutative DG $A$-rings 
$F_0(\til{B}) \subseteq F_1(\til{B}) \subseteq \cdots$. 
The differential of $F_i(\til{B})$ will be denoted by $\d_i$. At the same time 
we shall construct an increasing sequence of graded sets 
$F_0(X) \subseteq F_1(X) \subseteq \cdots$,
with a compatible sequence of isomorphisms 
$F_i(\til{B})^{\natural} \cong A^{\natural} \ot_{\Z} \Z[F_i(X)]$.
We shall also construct a compatible sequence of DG ring homomorphisms 
$v_i : F_i(\til{B}) \to B$.
The DG $A$-ring $\til{B} := \bigcup_{i} F_i(\til{B})$ and the 
homomorphism $v := \lim_{i \to} v_i$ will have the desired properties. 

The construction is by recursion on $i \in \N$. Moreover, for every $i$ the 
following conditions will hold:
\begin{enumerate}
\rmitem{i} The homomorphisms $v_i : F_i(\til{B}) \to B$,
$\opn{B}(v_i) : \opn{B}(F_i(\til{B})) \to \opn{B}(B)$ and 
$\opn{H}(v_i) : \opn{H}(F_i(\til{B})) \to \opn{H}(B)$ are surjective
in degrees $\geq -i$. 

\rmitem{ii} The homomorphism 
$\opn{H}(v_i) : \opn{H}(F_i(\til{B})) \to \opn{H}(B)$ is 
bijective in degrees $\geq -i + 1$.
\end{enumerate}

We start by choosing a set $Z_0''$ of degree $-1$ elements, and a function 
$v_0 : Z''_0 \to B^{-1}$, such that the set 
$\d_B(v_0(Z''_0))$ generates  $\opn{B}^{0}(B)$ as an $A^0$-module. 
Let $Y''_0$ be a set of degree $0$ elements, with a bijection 
$\d_0 : Z''_0 \to Y''_0$. 
Let $v_0 : Y''_0 \to B^{0}$ 
be the unique function such that 
$v_0(\d_0(z)) = \d_B(v_0(z))$ for all $z \in Z''_0$.

Choose a set $Y'_0$ of degree $0$ elements, and a function 
$v_0 : Y'_0 \to B^0$, such that the set $v_0(Y'_0 \sqcup Y''_0)$
generates $B^0$ as an $A^0$-ring. 

Define the graded set 
$F_0(X) := Y'_0 \sqcup Y''_0 \sqcup Z''_0$,
and the graded ring 
$F_0(\til{B})^{\natural} := A^{\natural} \ot_{\Z} \Z[F_0(X)]$.
Setting $\d_0(y) := 0$ for $y \in Y'_0 \sqcup Y''_0$
gives a degree $1$ function 
$\d_0 : F_0(X) \to F_0(\til{B})^{\natural}$.
According to Lemma \ref{lem:306} we get a differential $\d_0$ on 
the graded ring $F_0(\til{B})^{\natural}$, and it becomes the DG $A$-ring 
$F_0(\til{B})$. 
Lemma \ref{lem:305} says that the function $v_0 : F_0(X) \to B$ 
extends uniquely to a DG $A$-ring homomorphism
$v_0 : F_0(\til{B}) \to B$. Condition (i) holds, and condition (ii) holds 
trivially, for $i = 0$. 

Now take any $i \geq 0$, and assume that 
$v_i : F_i(\til{B}) \to B$ is already defined, and it satisfies conditions 
(i)-(ii). 
Choose a graded set $Y'_{i+1}$ of degree $-i-1$ elements, and a 
function $v_{i+1} : Y'_{i+1} \to \opn{Z}^{-i-1}(B)$, such that  
the cohomology classes of the elements of $v_{i+1}(Y'_{i+1})$
generate $\opn{H}^{-i-1}(B)$ as $\opn{H}^{0}(A)$-module. 
For $y \in Y'_{i+1}$ let $\d_{i+1}(y) := 0 \in F_i(\til{B})$.

Consider the $A^0$-module
\[ J_{i+1} := \bigl\{ b \in \opn{Z}^{-i} (F_i(\til{B})) \mid 
\mrm{H}^{-i}(v_i)(b) = 0 \bigr\} . \]
Choose a graded set $Y'''_{i+1}$ of degree $-i-1$ elements, and a 
function $\d_{i+1} : Y'''_{i+1} \to J_{i+1}$, such that the cohomology 
classes of the elements of $\d_{i+1}(Y'''_{i+1})$ generate the 
$\mrm{H}^{0}(A)$-module 
\[ \opn{Ker} \bigl( \mrm{H}^{-i}(v_i) : 
\mrm{H}^{-i}(F_i(\til{B})) \to \mrm{H}^{-i}(B) \bigr) . \]
For any $y \in Y'''_{i+1}$  there exists some $b \in B^{-i-1}$ such that 
$v_i(\d_{i+1}(y)) = \d(b)$, and we let $v_{i+1}(y) := b$. 

Choose a graded set $Z''_{i+1}$ of degree $-i-2$ elements, and a
function $v_{i+1} :  Z''_{i+1} \to B^{-i-2}$, such that the set 
$\d_B(v_{i+1}(Z''_{i+1}))$ generates $\opn{B}^{-i-1}(B)$ as 
$A^0$-module. 
Let $Y''_{i+1}$ be a set of degree $-i-1$ elements, with a bijection 
$\d_{i+1} : Z''_{i+1} \to Y''_{i+1}$. 
Let 
$v_{i+1} : Y''_{i+1} \to B^{-i-1}$ 
be the unique function such that 
$v_{i+1}(\d_{i+1}(z)) = \d_B(v_{i+1}(z))$
for all $z \in Z''_{i+1}$.
For $y \in Y''_{i+1}$ we define $\d_{i+1}(y) := 0$. 

Lastly choose a graded set $Y'_{i+1}$ of degree $-i-1$ elements, and a 
function $v_{i+1} : Y'_{i+1} \to B^{-i-1}$, such that, 
letting 
\[ Y_{i+1} := Y'_{i+1} \sqcup Y''_{i+1} \sqcup Y'''_{i+1}  , \] 
the set $v_{i+1}(Y_{i+1})$ generates $B^{-i-1}$ as an $A^0$-module. 
Since 
$v_i : \opn{B}^{-i}(F_i(\til{B})) \to \opn{B}^{-i}(B)$ 
is surjective, for any $y \in Y'_{i+1}$ there exists 
$b \in F_i(\til{B})^{-i}$ such that 
$v_i(b) = \d_B(v_{i+1}(y))$, and we define $\d_{i+1}(y) := b$.

Define the graded set 
\[ F_{i+1}(X) := F_{i}(X) \sqcup Y_{i+1} \sqcup Z''_{i+1} \]
and the commutative graded ring 
\[ F_{i+1}(\til{B})^{\natural} := A^{\natural} \ot_{\Z} \Z[F_{i+1}(X)] . \]
There is a degree $1$ function 
$\d_{i+1} : F_{i+1}(X) \to F_{i+1}(\til{B})^{\natural}$
such that $\d_{i+1}|_{F_i(X)} = d_i$. 
According to Lemma \ref{lem:306} the function $\d_{i+1}$ induces 
a differential $\d_{i+1}$ on $F_{i+1}(\til{B})^{\natural}$;
so we get a DG ring $F_{i+1}(\til{B})$. 
Likewise we have a degree $0$ function 
$v_{i+1} : F_{i+1}(X) \to B$ such that 
$v_{i+1}|_{F_{i}(X)} = v_i|_{F_{i}(X)}$.
By Lemma \ref{lem:305} there is an induced DG 
$A$-ring homomorphism $v_{i+1} : F_{i+1}(\til{B}) \to B$, 
and it satisfies 
$v_{i+1}|_{F_{i}(\til{B})} = v_i$.

\medskip \noindent 
(2) Here we are looking for a noncommutative semi-free DG $A$-ring $\til{B}$,
with a quasi-isomorphism $v : \til{B} \to B$ in 
$\cat{DGR}^{\leq 0} \centover A$. 
The proof here is the same as in part (1), except that we now replace 
$\Z[F_i(X)]$ with $\Z \bra{F_{i}(X)}$ everywhere.
\end{proof}

\begin{thm} \label{thm:306}
Let $A \xar{u} B$ be a central pair of nonpositive DG rings. 
Suppose we are given two factorizations 
$A \xar{\til{u}} \til{B} \xar{v} B$ and 
$A \xar{\til{u}'} \til{B}' \xar{v'} B$
of $u$ in $\cat{DGR}^{\leq 0} \centover A$, such that
$v$ is a surjective quasi-isomorphism, and either of these two conditions 
holds\tup{:}
\begin{enumerate}
\rmitem{i} The DG ring $\til{B}$ is commutative, and the DG ring 
$\til{B}'$ is commutative semi-free over $A$.
\rmitem{ii} The DG ring $\til{B}'$ is noncommutative semi-free over $A$.
\end{enumerate}
Then there exists a homomorphism 
$w : \til{B}' \to \til{B}$ in $\cat{DGR}^{\leq 0} \centover A$ such that 
$w \circ \til{u}' = \til{u}$ and $v \circ w = v'$.
\end{thm}

The situation is shown in this commutative diagram in 
$\cat{DGR} \centover A$ below. 

\begin{equation} \label{eqn:23}
\UseTips \xymatrix @C=5ex @R=5ex {
& 
\til{B}'
\ar[dr]^{v'}
\\
A
\ar[rr]^{u}
\ar[ur]^{\til{u}'} ="utilp"
\ar[dr]_{\til{u}}
{} \save 
[]+<-6ex,6ex> *+[F-:<3pt>]{\scriptstyle \text{semi-free}} 
\ar@{..} "utilp" 
\restore
& 
&
B
\\
&
\til{B}
\ar[ur]_{v} ="v"
{} \save 
[]+<10ex,-2ex> *+[F-:<3pt>]{\scriptstyle \text{surj qu-isom}} 
\ar@{..} "v" 
\restore
}
\qquad  
\xymatrix @C=5ex @R=5ex {
& 
\til{B}'
\ar[dr]^{v'}
\ar[dd]_{w}
\\
A
\ar[ur]^{\til{u}'}
\ar[dr]_{\til{u}}
& 
&
B
\\
&
\til{B}
\ar[ur]_{v}
}
\end{equation}

\begin{proof}
The proof is similar to that of \cite[Proposition 1.8]{YZ2}.
Let $X = \coprod_{i \leq 0} X^i$ be a set of commutative (resp.\ 
noncommutative) 
semi-free DG $A$-ring generators of $\til{B}'$ over $A$. Define 
$F_i(X) := \bigcup_{-i \leq j \leq 0} X^i$, 
and let $F_i(\til{B}')$ be the $A$-subring of $\til{B}'$ generated by the set 
$F_i(X)$. So $F_i(\til{B}')$ is a DG $A$-ring, 
$F_i(\til{B}')^{\natural} \cong A^{\natural} \ot_{\Z} \Z[F_i(X)]$
(resp.\ $F_i(\til{B}')^{\natural} \cong A^{\natural} \ot_{\Z} \Z \bra{F_i(X)}$),
and $\til{B}' = \bigcup_i F_i(\til{B}')$. 
We will construct a consistent sequence of homomorphisms 
$w_i : F_i(\til{B}') \to \til{B}$ in 
$\cat{DGR}^{\leq 0} \centover A$, 
satisfying 
$w_i \circ \til{u}' = \til{u}$ and $v \circ w_i = v'$.
The construction is by recursion on $i \in \N$.
Then $w := \lim_{i \to} w_i$ will have the required properties. 

We start with $i = 0$. Take any $x \in X^0$. Since $v$ is surjective, 
there exists $b \in \til{B}^{0}$ such that 
$v(b) = v'(x)$ in $B^{0}$. We let 
$w_0(x) := b$. The resulting function $w_0 : X^0 \to \til{B}$
extends uniquely to a DG $A$-ring homomorphism 
$w_0 : F_i(\til{B}') \to \til{B}$, by Lemma \ref{lem:305}. 

Next consider any $i \in \N$, and 
assume a DG $A$-ring homomorphism $w_{i} : F_i(\til{B}') \to \til{B}$ is 
defined, satisfying the required conditions. 
Take any element $x \in X^{-i-1}$. 
Since $v$ is surjective, there exists $b \in \til{B}^{-i-1}$ such 
that 
$v(b) = v'(x)$ in $B^{-i-1}$. 
Now $\d(x) \in F_i(\til{B}')^{-i}$ is a cocycle in $F_i(\til{B}')$, so 
$w_i(\d(x)) \in \til{B}^{-i}$ is a cocycle. 
Let $c := w_i(\d(x)) - \d(b) \in \til{B}^{-i}$, which is also a cocycle. 
We have 
\[ v(c) = v(w_i(\d(x)) - v(\d(b)) = 
v'(\d(x)) - v(\d(b)) = 
\d(v'(x) - v(b)) = 0 . \]
Thus the cohomology class $[c]$ satisfies 
$\opn{H}(v)([c]) = [v(c)] = 0$.
Because $ \opn{H}(v)$ is bijective we conclude that $[c] = 0$ in 
$\opn{H}^{-i}(\til{B})$. Hence there is some $b' \in \til{B}^{-i-1}$ such that 
$\d(b') = c$. We define $w_{i+1}(x) := b + b'$. Then 
\begin{equation} \label{eqn:305}
 \d(w_{i+1}(x)) = \d(b + b') = \d(b) + c = w_{i}(\d(x)) . 
\end{equation}

In this way we obtain a function 
$w_{i+1} : F_{i+1}(X) \to \til{B}$ that extends $w_i$.
According to Lemma \ref{lem:305}(1), the function $w_{i+1}$ extends uniquely 
to a homomorphism of graded $A^{\natural}$-rings
$F_{i+1}(w) : F_{i+1}(\til{B}')^{\natural} \to \til{B}^{\natural}$. 
Equation (\ref{eqn:305}) and Lemma \ref{lem:305}(2) imply that 
$w_{i+1} : F_{i+1}(\til{B}')^{} \to \til{B}^{}$
is in fact a homomorphism of DG rings. 
\end{proof}

\begin{cor} \label{cor:980} \mbox{}
\begin{enumerate}
\item Any object $B / A$ of $\cat{PDGR}^{\leq 0}$ admits a nonpositive 
K-flat resolution $\til{B} / \til{A}$.

\item Any morphism 
$w / v : B' / A' \to B / A$ in $\cat{PDGR}^{\leq 0}$ admits a nonpositive 
K-flat resolution $\til{w} / \til{v}$.
\end{enumerate}
\end{cor}

\begin{proof}
(1) Let $\til{B}$ be a noncommutative semi-free DG ring resolution of $B$ over 
$A$. See Theorem \ref{thm:305}.
Then $\til{B} / A$ is strict nonpositive K-flat resolution of $B / A$. 

\medskip \noindent
(2) Let $\til{B}'$ be a noncommutative semi-free DG ring resolution of $B'$ 
over 
$A'$, and let $\til{A}$ be a commutative semi-free DG ring resolution of $A$ 
over $A'$. Next, let $\til{B}$ be a noncommutative semi-free DG ring resolution 
of $B$ over $\til{A}$. So $\til{B}' / A'$ is a K-flat resolution of $B' / A'$, 
and $\til{B} / \til{A}$ is a K-flat resolution of $B / A$. Also, we have a 
homomorphism $\til{v} : A' \to \til{A}$ above $v : A' \to A$. Finally, 
according to Theorem \ref{thm:306}, we can find a 
homomorphism $\til{w} : \til{B}' \to \til{B}$
that respects all other homomorphisms. 
\end{proof}

\section{DG Ring Homotopies}
\label{sec:homotopies-rings}

In this section we introduce the concept of homotopy between DG ring 
homomorphisms. Many of the ideas in this section were  communicated to us 
privately by B. Keller (but of course the responsibility for correctness is 
ours). 

\begin{dfn} \label{dfn:70}
Suppose $u_0, u_1 : A \to B$ are homomorphisms of DG rings.  
\begin{enumerate}
\item An additive homomorphism $\ga : A \to B$ of degree $-1$ is called a {\em 
$u_0$-$u_1$-derivation} if it satisfies this twisted graded Leibniz formula
\[ \ga(a_0 \cd a_1) = \ga(a_0) \cd u_1(a_1) 
+ (-1)^{k_0} \cd u_0(a_0) \cd \ga(a_1) \]
for all $a_i \in A^{k_i}$.

\item A {\em DG ring homotopy} $\ga : u_0 \twoto u_1$ is 
$u_0$-$u_1$-derivation $\ga : A \to B$ of degree $-1$, satisfying the homotopy 
formula
\[ \d_B \circ \ga + \ga \circ \d_A = u_1 - u_0 . \]
\end{enumerate}
\end{dfn}

Here are homotopy results, that are only valid for noncommutative semi-free DG 
rings (Definition \ref{dfn:60}(2)). 
Recall the notions of commutative DG ring, central homomorphism, 
and the category $\cat{DGR}^{\leq 0} \centover A$, from 
Definitions \ref{dfn:750} and \ref{dfn:876}. 

\begin{lem} \label{lem:300}
Let $A$ be a commutative DG ring, and let
$w_0, w_1 : \til{B}' \to \til{B}$ be homomorphisms in 
$\cat{DGR}^{\leq 0} \centover A$. Assume that $\til{B}'$ 
is noncommutative semi-free over $A$, with set of noncommutative semi-free ring 
generators $X$. Let $\ga : X \to \til{B}$ be a degree $-1$ function.
\begin{enumerate}
\item There is a unique $A$-linear  $w_0$-$w_1$-derivation 
$\ga : \til{B}' \to \til{B}$ of degree $-1$ that extends $\ga : X \to \til{B}$.

\item The homomorphism $\ga : \til{B}' \to \til{B}$ is a DG ring homotopy 
$w_0 \twoto w_1$ iff 
\[ (\d \circ \ga + \ga \circ \d)(x) = (w_1 - w_0)(x) \]
for all $x \in X$. 
\end{enumerate}
\end{lem}

\begin{proof}
(1) Define an additive homomorphism 
$\ga : \Z \bra{X} \to \til{B}$ of degree $-1$ 
by letting 
\[ \ali{ 
& \ga(x_1 \cdots x_m) :=
\\ & \qquad 
\sum_{i = 1}^m (-1)^{k_1 + \cdots + k_{i - 1}} \cd
w_0(x_1) \cdots w_0(x_{i - 1}) \cd \ga(x_i) \cd 
w_1(x_{i + 1}) \cdots w_1(x_{m})
} \]
for elements $x_i \in X^{k_i}$.
(Here is where we need $\til{B}'$ to be noncommutative semi-free -- this won't 
work for $\Z[X]$ unless $w_0 = w_1$.)
Then extend $\ga$ $A$-linearly to all elements of $\til{B}'$, using the  
isomorphism 
$\til{B}'^{\, \natural} \cong A^{\natural} \ot_{\Z} \Z \bra{X}$. 

\medskip \noindent
(2) An easy  calculation, using induction on $m$, shows that 
\[ (\d \circ \ga + \ga \circ \d)(a \cd x_1 \cdots x_m) = 
(w_1 - w_0)(a \cd x_1 \cdots x_m)  \]
for all $a \in A$ and $x_i \in X$.
\end{proof}

\begin{thm} \label{thm:300}
Let $A$ be a commutative DG ring, and let 
$u : A \to B$ be a central homomorphism in $\cat{DGR}^{\leq 0}$.
Suppose we are given two factorizations 
$A \xar{\til{u}} \til{B} \xar{v} B$ and 
$A \xar{\til{u}'} \til{B}' \xar{v'} B$
of $u$ in $\cat{DGR}^{\leq 0} \centover A$, such that 
$v$ is a surjective quasi-isomorphism, and $\til{u}'$ is noncommutative 
semi-free. Let $w_0, w_1 : \til{B}' \to \til{B}$
be homomorphisms in $\cat{DGR}^{\leq 0} \centover A$
that satisfy 
$w_i \circ \til{u}' = \til{u}$ and $v \circ w_i = v'$. 
Then there is an $A$-linear DG ring homotopy $\ga : w_0 \twoto w_1$,
that satisfies $v \circ \ga = 0$. 
\end{thm}

The situation is depicted in the commutative diagrams below, where $i = 0, 1$. 
\begin{equation} \label{eqn:358}
\UseTips \xymatrix @C=5ex @R=5ex {
& 
\til{B}'
\ar[dr]^{v'}
\\
A
\ar[rr]^{u}
\ar[ur]^{\til{u}'} ="utilp"
{} \save 
[]+<-6ex,8ex> *+[F-:<3pt>]{\scriptstyle \text{NC semi-free}} 
\ar@{..} "utilp" 
\restore
\ar[dr]_{\til{u}}
& 
&
B
\\
&
\til{B}
\ar[ur]_{v} ="v"
{} \save 
[]+<10ex,-2ex> *+[F-:<3pt>]{\scriptstyle \text{surj qu-isom}} 
\ar@{..} "v" 
\restore
}
\qquad  
\xymatrix @C=5ex @R=5ex {
& 
\til{B}' 
\ar[dr]^{v'}
\ar[dd]_{w_i} ="wi"
\\
A
\ar[ur]^{\til{u}'}
\ar[dr]_{\til{u}}
{} \save 
[]+<-6ex,6ex> *+[F-:<3pt>]{\scriptstyle i = 0, 1} 
\ar@{..} "wi" 
\restore
& 
&
B
\\
&
\til{B}
\ar[ur]_{v}
} 
\end{equation}

\begin{proof}
Choose a set of noncommutative semi-free ring generators  
$X = \coprod_{i \leq 0} X^i$ for $\til{B}'$; so  
$\til{B}'^{\, \natural} \cong A^{\natural} \ot_{\Z} \Z \bra{X}$
as graded rings. 
For any $k \geq 0$ let 
$F_k(X) :=  \bigcup_{-k \leq i \leq 0} X^i$, and let
$F_k(\til{B}')$ be the $A$-subring of $\til{B}'$ generated by $F_k(X)$.
Thus $F_k(\til{B}')$ is in $\cat{DGR}^{\leq 0} \centover A$, and  
$F_k(\til{B}')^{\natural} \cong A^{\natural} \ot_{\Z} \Z \bra{F_k(X)}$. 
We will define an $A$-linear homomorphism 
$\ga_k : F_k(\til{B}') \to \til{B}$ of degree $-1$,
which is a DG ring homotopy 
$w_0|_{F_k(\til{B}')} \twoto w_1|_{F_k(\til{B}')}$,
recursively on $k$, such that 
$\ga_k|_{F_{k-1}(\til{B}')} = \ga_{k-1}$ if $k \geq 1$, 
and $v \circ \ga_k = 0$. Then we take $\ga := \lim_{k \to} \ga_k$. 

Since $v \circ w_i = v'$ it follows that 
$\opn{H}(v) \circ \opn{H}(w_i) = \opn{H}(v')$. Because 
$\opn{H}(v)$ is an isomorphism of graded rings, we see that 
$\opn{H}(w_0) = \opn{H}(w_1)$.  

Let us construct $\ga_0$. We will use the fact that all elements of 
$\til{B}'^{\, 0}$, $\til{B}^0$ and $B^0$ are cocycles.
Take any $x \in X^{0}$, and define 
$c := w_0(x) - w_1(x) \in \til{B}^{0}$.
Since $v(c) = 0$, so is its cohomology class 
$[v(c)] \in \mrm{H}^0(B)$. 
Because $\opn{H}(v)$ is an isomorphism, we see that 
$[c] = 0$ in $\mrm{H}^0(\til{B})$. Therefore 
$c = \d(b)$ for some $b \in \til{B}^{-1}$.
Consider the element $v(b) \in B^{-1}$. We have
$\d(v(b)) = v(c) = 0$, namely  $v(b)$ is a cocycle. 
Again using the fact that $\opn{H}(v)$ is bijective, 
we can find  a cocycle $b' \in \til{B}^{-1}$ such that 
$\opn{H}(v)([b']) = [v(b)]$ in $\opn{H}^{-1}(B)$.
But then $[v(b - b')] = 0$ in $\opn{H}^{-1}(B)$, 
so there exists some $a \in B^{-2}$ such that 
$\d(a) = v(b - b')$.  The surjectivity of $v$ says that there is
some $a' \in \til{B}^{-2}$ with $v(a') = a$. 
We define 
\[ \ga_0(x) := b - b' - \d(a') \in \til{B}^{-1} . \]
Then
\[ \d(\ga_0(x))  = \d(b) = c = w_0(x) - w_1(x) , \]
and 
\[ v(\ga_0(x)) = v(b - b') - \d(a) = 0 . \]

In this way we get a function 
$\ga_0 : X^0 \to \til{B}^{-1}$. Lemma \ref{lem:300} shows that 
this function extends uniquely to a homomorphism 
$\ga_0 : F_0(\til{B}') \to \til{B}$, which is a DG ring homotopy
$w_0|_{F_0(\til{B}')} \twoto w_1|_{F_0(\til{B}')}$, and 
$v \circ \ga_0 = 0$.

Now consider $k \geq 1$, and assume we already have $\ga_{k - 1}$. Take any 
$x \in X^{-k}$. Note that $\d(x) \in F_{k - 1}(\til{B}')$. 
We claim that the element 
\[ c := w_0(x) - w_1(x) - \ga_{k - 1}(\d(x)) \in \til{B}^{-k} \]
is a cocycle. Indeed, since $\ga_{k - 1}$ is a DG ring homotopy
$w_0|_{F_{k - 1}(\til{B}')} \twoto w_1|_{F_{k - 1}(\til{B}')}$,
we have 
\[ (\d \circ \ga_{k - 1} + \ga_{k - 1} \circ \d)(\d(x)) = 
(w_0 - w_1)(\d(x)) . \]
But $\d(\d(x)) = 0$ and $w_i(\d(x)) = \d(w_i(x))$, so 
$\d(c) = 0$ as claimed.

{}From here the proof is just like in the case $k = 0$. 
Because $v \circ w_i = v'$ and 
$v \circ \ga_{k - 1} = 0$, we see that $v(c) = 0$.
Therefore the cohomology class 
$[v(c)] \in  \opn{H}^{-k}(B)$ is $0$. 
Using the fact that $\opn{H}(v)$ is an isomorphism, we conclude that 
$[c] = 0$ in $\opn{H}^{-k}(\til{B})$.
Hence there is some $b \in \til{B}^{-k - 1}$ such that 
$\d(b) = c$. 
Now 
$\d(v(b)) = v(\d(b)) =  v(c) = 0$,
so $v(b) \in B^{-k - 1}$ is a cocycle.
Again using the fact that $\opn{H}(v)$ is bijective, 
we can find  a cocycle $b' \in \til{B}^{-k - 1}$ such that 
$\opn{H}(v)([b']) = [v(b)]$ in $\opn{H}^{-k - 1}(B)$.
But then $[v(b - b')] = 0$ in $\opn{H}^{-k  - 1}(B)$, 
so there exists some $a \in B^{-k - 2}$ such that 
$\d(a) = v(b - b')$.  The surjectivity of $v$ says that there is
some $a' \in \til{B}^{-k - 2}$ with $v(a') = a$. 
We define 
\[ \ga_k(x) := b - b' - \d(a') \in \til{B}^{-k - 1} . \]
Then
\[ (\d \circ \ga_k + \ga_{k - 1} \circ \d)(x)  = 
c  + \ga_{k - 1}(\d(x))
= w_0(x) - w_1(x) , \]
and 
$v(\ga_k(x)) = 0$.

In this way we get a function 
$\ga_k : X^{-k} \to \til{B}^{-k - 1}$. We extend it to a function 
$\ga_k : F_k(X) \to \til{B}$ by defining 
$\ga_k(x) := \ga_{k - 1}(x)$ for $x \in F_{k - 1}(X)$. 
Lemma \ref{lem:300} shows that this function extends uniquely to an $A$-linear 
homomorphism 
$\ga_k : F_k(\til{B}') \to \til{B}$, which is a DG ring homotopy
$w_k|_{F_k(\til{B}')} \twoto w_1|_{F_k(\til{B}')}$, and 
$v \circ \ga_k = 0$.
\end{proof}

DG ring homotopies can be expressed using the cylinder construction
(see Definition \ref{dfn:700}). The next result is similar to 
\cite[Theorem 4.3(c)]{Ke3}.

\begin{prop} \label{prop:300}
Let $u_0, u_1 : A \to B$ be DG ring homomorphisms, and let 
$\ga : A \to B$ be a $\Z$-linear homomorphism of degree $-1$. 
The following conditions are equivalent. 
\begin{enumerate}
\rmitem{i} The homomorphism $\ga$ is a DG ring homotopy $\ga : u_0 \twoto u_1$.

\rmitem{ii} The $\Z$-linear homomorphism 
$u_{\mrm{cyl}} : A \to \opn{Cyl}(B)$ 
with formula 
\[ u_{\mrm{cyl}}(a) := 
\bmat{u_0(a) & \ \xi \cd \ga(a) \\[0.3em] 0 & u_1(a)} \]
is a DG ring homomorphism. 
\end{enumerate}
\end{prop}

\begin{proof}
This is a straightforward calculation.
\end{proof}

\begin{rem} \label{rem:780}
Let $A$ be a commutative DG ring. If $A$ is a ring (i.e.\ $A = A^0$), then the 
categories  $\cat{DGR}^{\leq 0} \centover A$ and
$\cat{DGR} \centover A$ admit Quillen model structures, in which the 
quasi-isomorphisms are the weak equivalences, and the
noncommutative semi-free DG rings are cofibrant; see \cite[Theorem A.3.1]{BP}. 
We think it is quite plausible that the 
same is true even when $A$ is not a ring (i.e.\ it has a nontrivial negative 
part). Evidence for this is provided by Theorems \ref{thm:305}, 
\ref{thm:306} and \ref{thm:300}.

On the other hand, we do not know  whether the categories 
$\cat{DGR}_{\mrm{sc}}^{\leq 0} / A$ and $\cat{DGR}_{\mrm{sc}} / A$ 
admit similar Quillen model structures. A negative indication is that Theorem 
\ref{thm:300} does not seem to hold for commutative semi-free DG rings.
Another negative indication is that even when $A$ is a ring, this is not known
(except when $A$ contains $\Q$; see \cite{Hi}). 
 
Recall \cite[Theorem 3.2]{AILN}, that was discussed in Subsection 
\ref{subsec:discussion} of the Introduction. It deals with a commutative 
base ring $A$. The proof of this theorem hinges on the Quillen model 
structure on $\cat{DGR}^{\leq 0} \centover A$ that was produced in \cite{BP}.  
If such a model structure does exist in the more 
general case of a commutative DG ring $A$ (as we predict above), then it would 
most likely imply Theorem \ref{thm:965} (by a proof similar to that in 
\cite{AILN}). 

However, even if we had model structures on the categories 
$\cat{DGR}^{\leq 0} \centover A$ and $\cat{DGR}_{\mrm{sc}}^{\leq 0} / A$, 
that would probably not be sufficient to imply Theorem \ref{thm:966}.
This is because the proof of Theorem \ref{thm:966} (actually, the proof of the 
noncommutative variant Theorem \ref{thm:840}) requires a very delicate 
treatment of central morphisms of central pairs (see Setup \ref{set:930} and 
Definition \ref{dfn:911}), 
something that a Quillen model structure (being a rather coarse structure) does 
not appear able to provide.
\end{rem}

\section{Pairs of DG Modules, Compound Resolutions, and Rectangles} 
\label{sec:pairs-DG-modules}

Recall that $\cat{PDGR}$ is the category of central pairs DG rings.
See Definition \ref{dfn:910}. 
Resolutions in $\cat{PDGR}$ were defined in Definition \ref{dfn:912} and 
\ref{dfn:52}. The category of central DG $A$-rings  
$\cat{DGR} \centover A$
was defined in Definition \ref{dfn:750}(4). 

\begin{dfn} \label{dfn:914}
Let $\til{B} / \til{A}$ be an object of $\cat{PDGR}$. 
The {\em enveloping DG ring} of $\til{B}$ over $\til{A}$ is the DG ring 
\[ \til{B}^{\mrm{en}} := 
\til{B}^{} \ot_{\til{A}^{}} \til{B}^{\mrm{op}}
\in \cat{DGR} \centover \til{A} . \]

Let 
$\til{w} / \til{v} : \til{B}_1 / \til{A}_1 \to \til{B}_0 / \til{A}_0$
be a morphism in $\cat{PDGR}$. There is an induced DG ring homomorphism 
\[ \til{w}^{\mrm{en}} := \til{w} \ot_{\til{v}} \til{w}^{\mrm{op}} : 
\til{B}^{\mrm{en}}_1 \to \til{B}^{\mrm{en}}_0  \]
in $\cat{DGR} \centover \til{A}_1$. 
\end{dfn}

In this way we obtain a functor 
$\cat{PDGR} \to \cat{DGR}$.

\begin{lem} \label{lem:911}
Let $w / v : B_1 / A_1 \to B_0 / A_0$ be a quasi-isomorphism in $\cat{PDGR}$,
and let 
$\til{w} / \til{v} : \til{B}_1 / \til{A}_1 \to \til{B}_0 / \til{A}_0$
be a K-flat resolution of $w / v$. Then 
$\til{w}^{\mrm{en}} : \til{B}^{\mrm{en}}_1 \to \til{B}^{\mrm{en}}_0$
is a quasi-isomorphism.
\end{lem}

\begin{proof}
The homomorphisms 
$\til{v} : \til{A}_1 \to \til{A}_0$
and 
$\til{w} : \til{B}_1 \to \til{B}_0$
are also quasi-iso\-morphisms. 
We can forget the DG rings structures of $\til{B}_0$ and $\til{B}_1$, and just 
view them as K-flat DG modules over $\til{A}_0$ and $\til{A}_1$ respectively. 
Then Proposition \ref{prop:462}(1) applies. 
\end{proof}
 
\begin{dfn} \label{dfn:915}
Let $\til{B} / \til{A}$ be an object of $\cat{PDGR}$, and let 
$(\til{M}^{\mrm{l}}, \til{M}^{\mrm{r}})$
be an object of 
$\cat{M}(\til{B}) \times \cat{M}(\til{B}^{\mrm{op}})$;
namely $\til{M}^{\mrm{l}}$ is a left DG $\til{B}$-module, and 
$\til{M}^{\mrm{r}}$ is a right DG $\til{B}$-module.
We define the DG module
\[ \til{M}^{\mrm{en}} := \til{M}^{\mrm{l}} \ot_{\til{A}} 
\til{M}^{\mrm{r}} \in \cat{M}(\til{B}^{\mrm{en}}) . \]

Let 
$\til{w} / \til{v} : \til{B}_1 / \til{A}_1 \to \til{B}_0 / \til{A}_0$
be a morphism in $\cat{PDGR}$, let 
\[ (\til{M}^{\mrm{l}}_k, \til{M}^{\mrm{r}}_k) \in 
\cat{M}(\til{B}_k) \times \cat{M}(\til{B}^{\mrm{op}}_k) \]
for $k = 0, 1$, and let 
\[ (\til{\th}^{\mrm{l}}, \til{\th}^{\mrm{r}}) : 
(\til{M}^{\mrm{l}}_0, \til{M}^{\mrm{r}}_0) \to
(\til{M}^{\mrm{l}}_1, \til{M}^{\mrm{r}}_1) \]
be a morphism in 
$\cat{M}(\til{B}_1) \times \cat{M}(\til{B}^{\mrm{op}}_1)$.
We define the morphism 
\[ \til{\th}^{\mrm{en}} := 
\til{\th}^{\mrm{l}} \ot_{\til{v}} \til{\th}^{\mrm{r}} : 
\til{M}^{\mrm{en}}_0 \to \til{M}^{\mrm{en}}_1  \]
in $\cat{M}(\til{B}^{\mrm{en}}_1)$.
\end{dfn}

\begin{rem}
We have omitted reference to the forgetful functors in the definition 
above. Thus, the homomorphism 
$\til{\th}^{\mrm{l}}$ in $\cat{M}(\til{B}_1)$ 
is actually a homomorphism 
$\til{\th}^{\mrm{l}} : \opn{For}_{w}(\til{M}^{\mrm{l}}_0) \to M^{\mrm{l}}_1$;
and the homomorphism 
$\til{\th}^{\mrm{r}}$ in $\cat{M}(\til{B}^{\mrm{op}}_1)$ 
is actually a homomorphism 
$\til{\th}^{\mrm{r}} : \opn{For}_{w^{\mrm{op}}}(\til{M}^{\mrm{r}}_0) \to 
M^{\mrm{r}}_1$.

We shall continue with these omissions, for the sake of clarity. 
The forgetful functors will only be made explicit when there is more than one 
given homomorphism between two DG rings. We hope this shortcut will not be a 
cause for confusion. 
\end{rem}

\begin{dfn} \label{dfn:917}
Let $B/ A$ be an object of $\cat{PDGR}$, let 
$(M^{\mrm{l}}, M^{\mrm{r}})$ be an object of 
$\cat{D}(B) \times \cat{D}(B^{\mrm{op}})$, 
and let $\til{B} / \til{A}$ be a K-flat resolution of $B / A$. 
A {\em compound resolution of $(M^{\mrm{l}}, M^{\mrm{r}})$ over 
$\til{B} / \til{A}$} is the following data:
\begin{itemize}
\item A K-flat resolution 
$\al^{\mrm{l}} : \til{P}^{\mrm{l}} \to M^{\mrm{l}}$
in $\cat{M}(\til{B})$. 

\item A K-flat resolution 
$\al^{\mrm{r}} : \til{P}^{\mrm{r}} \to M^{\mrm{r}}$
in $\cat{M}(\til{B}^{\mrm{op}})$. 

\item A K-injective resolution 
$\be : \til{P}^{\mrm{en}} \to \til{I}$
in $\cat{M}(\til{B}^{\mrm{en}})$.
\end{itemize}
We denote this compound resolution by 
\[ \bsym{P} := (\til{P}^{\mrm{l}}, \til{P}^{\mrm{r}}, \til{I}; 
\al^{\mrm{l}}, \al^{\mrm{r}}, \be ) . \]
\end{dfn}

It is easy to see that compound resolutions exist. 

Consider a K-flat resolution
$s / r : \til{B} / \til{A} \to B / A$ in $\cat{PDGR}$.
The left and right actions of $\til{B}$ on $B$ commute with each other, and 
they commute with the action of $B^{\mrm{ce}}$ on $B$ (all in the graded 
sense). Thus $B$ is a DG module over
$\til{B}^{\mrm{en}} \ot_{\til{A}} B^{\mrm{ce}}$. Hence we can make the next 
definition, in the spirit of Proposition \ref{prop:461}. 

\begin{dfn} \label{dfn:918}
In the situation of Definition \ref{dfn:917}, we define the DG module
\[ \opn{Rect}_{B / A}^{\til{B} / \til{A}}(\bsym{P}) :=
\opn{Hom}_{B^{\mrm{en}}}(B, \til{I}) \in \cat{M}(B^{\mrm{ce}}) . \]
\end{dfn}

The next setup shall be used in the rest of this section. 

\begin{setup} \label{set:940}
For $k = 0, 1, 2$ we are given a central pair
$B_k / A_k \in \cat{PDGR}$, 
a K-flat resolution
$\til{B}_k / \til{A}_k \to B_k / A_k$ in $\cat{PDGR}$, 
and a pair of DG modules 
\[ (M^{\mrm{l}}_k, M^{\mrm{r}}_k) \in
\cat{D}(B_k) \times \cat{D}(B^{\mrm{op}}_k) . \]

For $k = 1, 2$ we are given a central morphism
\[ w_k / v_k : B_k / A_k \to B_{k - 1} / A_{k - 1}  \] 
in $\cat{PDGR}$ (i.e.\ a morphism in $\cat{PDGR}_{\vec{\mrm{ce}}}$, see 
Definition \ref{dfn:911}(3)), a morphism 
\[ (\th_{k}^{\mrm{l}},  \th_{k}^{\mrm{r}}) : 
(M_{k - 1}^{\mrm{l}},  M_{k - 1}^{\mrm{r}}) \to 
(M_{k}^{\mrm{l}},  M_{k}^{\mrm{r}}) \]
in $\cat{D}(B_k) \times \cat{D}(B_k^{\mrm{op}})$,
and a morphism 
\[ \til{w}_k / \til{v}_k : \til{B}_k / \til{A}_k \to 
\til{B}_{k - 1} / \til{A}_{k - 1} \]
above $w_k / v_k$. 

When we only consider $k = 0, 1$, the index $1$ on the morphisms 
$w_1$, $v_1$, $\th_{1}^{\mrm{l}}$, $\th_{1}^{\mrm{r}}$,  
$\til{w}_1$ and $\til{v}_1$ is redundant, and hence we will suppress it. 
\end{setup}

\begin{dfn} \label{dfn:919}
Consider Setup \ref{set:940}, with $k = 0, 1$ only. 
Let $\bsym{P}_k$ be a compound resolution of 
$(M^{\mrm{l}}_k, M^{\mrm{r}}_k)$
over $\til{B}_k / \til{A}_k$. 
A {\em compound morphism}
$\bsym{\eta} : \bsym{P}_0 \to \bsym{P}_1$
above $(\th_{}^{\mrm{l}},  \th_{}^{\mrm{r}})$
and $\til{w} / \til{v}$
is the following data:
\begin{itemize}
\item K-projective resolutions 
$\ga_{0}^{\mrm{l}} : \til{Q}_{0}^{\mrm{l}} \to \til{P}_{0}^{\mrm{l}}$
and 
$\ga_{0}^{\mrm{r}} : \til{Q}_{0}^{\mrm{r}} \to \til{P}_{0}^{\mrm{r}}$
in $\cat{M}(\til{B}_1)$ and $\cat{M}(\til{B}_1^{\mrm{op}})$
respectively. 

\item  Homomorphisms 
$\til{\th}^{\mrm{l}} : \til{Q}_{0}^{\mrm{l}} \to \til{P}_1^{\mrm{l}}$
and 
$\til{\th}^{\mrm{r}} : \til{Q}_{0}^{\mrm{r}} \to \til{P}_1^{\mrm{r}}$
in $\cat{M}(\til{B}_1)$ and
$\cat{M}(\til{B}_1^{\mrm{op}})$ respectively,
such that 
\[ \th_{}^{\mrm{l}} \circ \opn{Q}(\al_{0}^{\mrm{l}})
\circ \opn{Q}(\ga_{0}^{\mrm{l}}) =
\opn{Q}(\al_{1}^{\mrm{l}}) \circ \opn{Q}(\til{\th}_{}^{\mrm{l}}) \]
and
\[ \th_{}^{\mrm{r}} \circ \opn{Q}(\al_{0}^{\mrm{r}})
\circ \opn{Q}(\ga_{0}^{\mrm{r}}) =
\opn{Q}(\al_{1}^{\mrm{r}}) \circ \opn{Q}(\til{\th}_{}^{\mrm{r}}) \]
in $\cat{D}(\til{B}_1)$ and $\cat{D}(\til{B}_1^{\mrm{op}})$ respectively.

\item A homomorphism $\eta : \til{I}_{0} \to \til{I}_1$
in $\cat{M}(\til{B}_1^{\mrm{en}})$, such that 
\[ \opn{Q}(\eta) \circ \opn{Q}(\be_{0}) 
\circ \opn{Q}( \ga_{0}^{\mrm{l}} \ot_{\til{v}} \ga_{0}^{\mrm{r}} ) = 
\opn{Q}(\be_{1}) \circ \opn{Q}(\til{\th}_{}^{\mrm{en}}) , \]
as morphisms 
$\til{Q}_{0}^{\mrm{l}} \ot_{\til{A}_1} \til{Q}_{0}^{\mrm{r}}
\to \til{I}_1$
in $\cat{D}(\til{B}_1^{\mrm{en}})$.
\end{itemize}
We denote this compound morphism by 
\[ \bsym{\eta} = (\til{Q}_{0}^{\mrm{l}}, \til{Q}_{0}^{\mrm{r}};
\ga_{0}^{\mrm{l}}, \ga_{0}^{\mrm{r}},
\til{\th}^{\mrm{l}}, \til{\th}^{\mrm{r}}, \eta ) . \]
We also say that $\bsym{\eta}$ is a {\em compound resolution of 
$(\th_{}^{\mrm{l}},  \th_{}^{\mrm{r}})$}.
\end{dfn}

This definition is illustrated in the next diagrams, that are commutative
diagrams in the categories $\cat{D}(\til{B}_1)$,
$\cat{D}(\til{B}_1^{\mrm{op}})$
and $\cat{D}(\til{B}_1^{\mrm{en}})$ respectively.
The vertical arrows in these diagrams (namely those whose names do not contain 
the letters ``$\th$'' or ``$\eta$'') are isomorphisms. 
(Note that the letter $\opn{Q}$ denotes the localization functor, 
not to be confused with the DG modules $\til{Q}_0^{\mrm{l}}$ etc.)

\[ \UseTips \xymatrix @C=7ex @R=6ex {
\til{Q}_0^{\mrm{l}}
\ar[d]_{ \opn{Q}(\ga_0^{\mrm{l}}) }
\ar[dr]^{ \opn{Q}(\til{\th}^{\mrm{l}}) }
\\
\til{P}_0^{\mrm{l}}
\ar[d]_{ \opn{Q}(\al_0^{\mrm{l}}) }
&
\til{P}_1^{\mrm{l}}
\ar[d]_{ \opn{Q}(\al_1^{\mrm{l}}) }
\\
M_0^{\mrm{l}}
\ar[r]_{ \th^{\mrm{l}} }
&
M_1^{\mrm{r}}
} 
\qquad 
\xymatrix @C=6ex @R=7ex {
\til{Q}_0^{\mrm{r}}
\ar[d]_{ \opn{Q}(\ga_0^{\mrm{r}}) }
\ar[dr]^{ \opn{Q}(\til{\th}^{\mrm{r}}) }
\\
\til{P}_0^{\mrm{r}}
\ar[d]_{ \opn{Q}(\al_0^{\mrm{r}}) }
&
\til{P}_1^{\mrm{r}}
\ar[d]_{ \opn{Q}(\al_1^{\mrm{r}}) }
\\
M_0^{\mrm{r}}
\ar[r]_{ \th^{\mrm{r}} }
&
M_1^{\mrm{r}}
} 
\qquad 
\UseTips \xymatrix @C=6ex @R=6ex {
\til{Q}_{0}^{\mrm{l}} \ot_{\til{A}_1} \til{Q}_{0}^{\mrm{r}}
\ar[d]_{ \opn{Q}(\ga_{0}^{\mrm{l}} \ot_{\til{v}} \ga_{0}^{\mrm{r}}) }
\ar[dr]^{ \opn{Q}(\til{\th}^{\mrm{en}}) }
\\
\til{P}_0^{\mrm{en}}
\ar[d]_{ \opn{Q}(\be_0) }
&
\til{P}_1^{\mrm{en}}
\ar[d]_{ \opn{Q}(\be_1) }
\\
\til{I}_0^{}
\ar[r]_{ \opn{Q}(\eta) }
&
\til{I}_1^{}
} 
\]

There is no composition rule for the compound morphisms $\bsym{\eta}$; 
but this is taken care of by the following lemmas. 

\begin{lem} \label{lem:915}
In the situation of Definition \tup{\ref{dfn:919}}, compound
resolutions $\bsym{\eta}$ of the morphism 
$(\th_{}^{\mrm{l}},  \th_{}^{\mrm{r}})$ exist.
\end{lem}

\begin{proof}
The homomorphisms $\til{\th}^{\mrm{l}}$ and $\til{\th}^{\mrm{r}}$ exist because 
$\til{Q}_0^{\mrm{l}}$ and $\til{Q}_0^{\mrm{r}}$ are K-projective in the 
respective categories. The homomorphism $\eta$ exists because 
$\til{I}_1$ is K-injective. See Proposition \ref{prop:460}. 
\end{proof}

In the next definition, it is crucial that the homomorphism 
$w : B_1 \to B_0$ is central; this is guaranteed by Setup \ref{set:940}, in 
which we demand that $w / v$ is a morphism in $\cat{PDGR}_{\vec{\mrm{ce}}}$.
Since $w$ is central, we get an induced DG ring homomorphism 
$w^{\mrm{ce}} : B^{\mrm{ce}}_1 \to B^{\mrm{ce}}_0$. 
So Proposition \ref{prop:461}(2) applies.

\begin{dfn} \label{dfn:920}
In the situation of Definition \tup{\ref{dfn:919}}, we let 
\[ \opn{Rect}_{w / v}^{\til{w} / \til{v}}(\bsym{\eta}) : 
\opn{Rect}_{B_0 / A_0}^{\til{B}_{0} / \til{A}_{0}}(\bsym{P}_{0}) \to
\opn{Rect}_{B_1 / A_1}^{\til{B}_{1} / \til{A}_{1}}(\bsym{P}_{1}) \]
be the morphism 
\[ \opn{Rect}_{w / v}^{\til{w} / \til{v}}(\bsym{\eta})  :=
\opn{Q} \bigl( \opn{Hom}_{w^{\mrm{en}}}(w, \eta) \bigr) \]
in $\cat{D}(B_1^{\mrm{ce}})$.  
\end{dfn}

\begin{lem} \label{lem:916}
The morphism 
$\opn{Rect}_{w / v}^{\til{w} / \til{v}}(\bsym{\eta})$
is independent of the compound resolution $\bsym{\eta}$. Namely, if 
$\bsym{\eta}'$ is another compound resolution of 
$(\th^{\mrm{l}},  \th^{\mrm{r}})$, 
then there is equality
\[ \opn{Rect}_{w / v}^{\til{w} / \til{v}}(\bsym{\eta}) = 
\opn{Rect}_{w / v}^{\til{w} / \til{v}}(\bsym{\eta}') \]
of morphisms 
\[ \opn{Rect}_{B_0 / A_0}^{\til{B}_0 / \til{A}_{0}}(\bsym{P}_{0}) \to
\opn{Rect}_{B_1 / A_1}^{\til{B}_{1} / \til{A}_{1}}(\bsym{P}_{1}) \]
in $\cat{D}(B_1^{\mrm{ce}})$.
\end{lem}

\begin{proof}
Say 
\[ \bsym{\eta}' = (\til{Q}_{0}'^{\, \mrm{l}}, \til{Q}_{0}'^{\, \mrm{r}};
\ga_{0}'^{\, \mrm{l}}, \ga_{0}'^{\, \mrm{r}},
\til{\th}'^{\, \mrm{l}}, \til{\th}'^{\, \mrm{r}}, \eta' ) . \]
The DG modules $\til{Q}_{0}'^{\, \mrm{l}}$ and $\til{Q}_{0}^{\mrm{l}}$
are homotopy equivalent in $\cat{M}(\til{B}_1)$, 
and under this equivalence the homomorphisms 
$\til{\th}'^{\, \mrm{l}} : \til{Q}_{0}'^{\, \mrm{l}} \to
\til{P}_{1}^{\mrm{l}}$
and 
$\til{\th}^{\mrm{l}} : \til{Q}_{0}^{\mrm{l}} \to \til{P}_{1}^{\mrm{l}}$
are homotopic. Likewise for ``r''. Therefore the DG modules 
$\til{Q}_{0}'^{\, \mrm{en}}$ and $\til{Q}_{0}^{\mrm{en}}$
are homotopy equivalent in $\cat{M}(\til{B}_1^{\mrm{en}})$, 
and under this equivalence the homomorphisms 
$\til{\th}'^{\, \mrm{en}} : \til{Q}_{0}'^{\, \mrm{en}} \to
\til{P}_{1}^{\mrm{en}}$
and 
$\til{\th}^{\mrm{en}} : \til{Q}_{0}^{\mrm{en}} \to \til{P}_{1}^{\mrm{en}}$
are homotopic.
This implies that the homomorphisms $\eta'$ and $\eta$ are homotopic.
\end{proof}

\begin{lem} \label{lem:925}
Assume that $B_1 / A_1 = B_0 / A_0$, 
$\til{B}_1 / \til{A}_1 = \til{B}_0 / \til{A}_0$,
$(M^{\mrm{l}}_1, M^{\mrm{r}}_1) = (M^{\mrm{l}}_0, M^{\mrm{r}}_0)$, 
and all the morphisms between them are the identities. 
Also assume that $\bsym{P}_1 = \bsym{P}_0$. Then for any compound morphism
$\bsym{\eta} : \bsym{P}_{0} \to \bsym{P}_{0}$
above 
$(\opn{id}_{M^{\mrm{l}}_0}, \opn{id}_{M^{\mrm{r}}_0})$,
the morphism 
$\opn{Rect}_{\mrm{id} / \mrm{id}}^{\mrm{id} / \mrm{id}}(\bsym{\eta})$
is the identity automorphism of
$\opn{Rect}_{B_0 / A_0}^{\til{B}_0 / \til{A}_{0}}(\bsym{P}_{0})$.
\end{lem}

\begin{proof}
Clear. 
\end{proof}

\begin{lem} \label{lem:919}
Consider Setup \tup{\ref{set:940}}, with $k = 0, 1, 2$. 
We are given a compound resolution
$\bsym{P}_k$ of $(M^{\mrm{l}}_k, M^{\mrm{r}}_k)$
over $\til{B} / \til{A}$. 
For $k = 1, 2$ we are given a compound morphism  
$\bsym{\eta}_k : \bsym{P}_{k - 1} \to \bsym{P}_{k}$
above $(\th_{k}^{\mrm{l}},  \th_{k}^{\mrm{r}})$ 
and $\til{w}_k / \til{v}_k$.
Also, we are given a compound morphism 
$\bsym{\eta}_{0, 2} : \bsym{P}_{0} \to \bsym{P}_{2}$
above 
$(\th_{2}^{\mrm{l}}, \th_{2}^{\mrm{r}}) \circ 
(\th_{1}^{\mrm{l}}, \th_{1}^{\mrm{r}})$
and 
$(\til{w}_1 / \til{v}_1) \circ  (\til{w}_2 / \til{v}_2)$.
Then there is equality
\[ \opn{Rect}_{(w_1 / v_1) \circ (w_2 / v_2)}
^{(\til{w}_1 / \til{v}_1) \circ (\til{w}_2 / \til{v}_2)}(\bsym{\eta}_{0, 2}) 
= \opn{Rect}_{w_2 / v_2}^{\til{w}_2 / \til{v}_2}(\bsym{\eta}_2) 
\circ 
\opn{Rect}_{w_1 / v_1}^{\til{w}_1 / \til{v}_1}(\bsym{\eta}_1) \]
of morphisms 
\[ \opn{Rect}_{B_0 / A_0}^{\til{B}_0 / \til{A}_{0}}(\bsym{P}_{0}) \to
\opn{Rect}_{B_2 / A_2}^{\til{B}_{2} / \til{A}_{2}}(\bsym{P}_{2})  \]
in $\cat{D}(B_2^{\mrm{ce}})$.
\end{lem}

\begin{proof}
By choosing homomorphisms 
$\til{Q}_{0}^{\mrm{l}} \to \til{Q}_{1}^{\mrm{l}}$
and 
$\til{Q}_{0}^{\mrm{r}} \to \til{Q}_{1}^{\mrm{r}}$
that lift $\th_1^{\mrm{l}}$ and $\th_1^{\mrm{r}}$
respectively, we can concoct a ``composed'' morphism 
$\bsym{\eta}_2 \circ \bsym{\eta}_1$,
well-defined up to homotopy, whose component 
$\til{I}_0 \to \til{I}_2$ is  $\eta_2 \circ \eta_1$. 
Thus 
\[ \opn{Rect}_{(w_1 / v_1) \circ (w_2 / v_2)}
^{(\til{w}_1 / \til{v}_1) \circ (\til{w}_2 / \til{v}_2)}
(\bsym{\eta}_2 \circ \bsym{\eta}_1) = 
\opn{Rect}_{w_2 / v_2}^{\til{w}_2 / \til{v}_2}(\bsym{\eta}_2) 
\circ 
\opn{Rect}_{w_1 / v_1}^{\til{w}_1 / \til{v}_1}(\bsym{\eta}_1) . \]
On the other hand, by Lemma \ref{lem:916} we know that 
\[ \opn{Rect}_{(w_1 / v_1) \circ (w_2 / v_2)}
^{(\til{w}_1 / \til{v}_1) \circ (\til{w}_2 / \til{v}_2)}
(\bsym{\eta}_2 \circ \bsym{\eta}_1) = 
\opn{Rect}_{(w_1 / v_1) \circ (w_2 / v_2)}
^{(\til{w}_1 / \til{v}_1) \circ (\til{w}_2 / \til{v}_2)}
(\bsym{\eta}_{0, 2}) . \] 
\end{proof}

\begin{lem} \label{lem:917}
If 
$w / v : B_1 / A_1 \to B_0 / A_0$ is a quasi-isomorphism in $\cat{PDGR}$,
and if 
\[ (\th_{}^{\mrm{l}},  \th_{}^{\mrm{r}}) : 
(M_{0}^{\mrm{l}},  M_{0}^{\mrm{r}}) \to 
(M_{1}^{\mrm{l}},  M_{1}^{\mrm{r}}) \]
is an isomorphism in $\cat{D}(B_1) \times \cat{D}(B_1^{\mrm{op}})$,
then 
$\opn{Rect}_{w / v}^{\til{w} / \til{v}}(\bsym{\eta})$
is an isomorphism in $\cat{D}(B_1^{\mrm{ce}})$.
\end{lem}

\begin{proof}
Here 
$w : B_1 \to B_0$,
$\til{w}^{\mrm{en}} : B_1^{\mrm{en}} \to B_0^{\mrm{en}}$,
$\til{\th}^{\mrm{l}}$, $\til{\th}^{\mrm{r}}$,
$\til{\th}^{\mrm{en}}$ and
$\eta : \til{I}_0 \to \til{I}_1$ are quasi-iso\-morphism.
According to Proposition \ref{prop:462}(2), 
$\opn{Hom}_{w^{\mrm{en}}}(w, \eta)$
is a quasi-isomorphism.
\end{proof}

\begin{prop} \label{prop:915}
Let $B / A$ be a pair in $\cat{PDGR}$, with K-flat resolution
$\til{B} / \til{A}$. 
Let $(M^{\mrm{l}}, M^{\mrm{r}})$ be a pair in  
$\cat{D}(B) \times \cat{D}(B^{\mrm{op}})$. 
There is an object 
\[ \opn{Rect}_{B / A}^{\til{B} / \til{A}}(M^{\mrm{l}}, M^{\mrm{r}})
\in \cat{D}(B^{\mrm{ce}}) , \]
unique up to a unique isomorphism, together with an isomorphism 
\[ \opn{rect}(\bsym{P}) : 
\opn{Rect}_{B / A}^{\til{B} / \til{A}}(M^{\mrm{l}}, M^{\mrm{r}})
\iso \opn{Rect}_{B / A}^{\til{B} / \til{A}}(\bsym{P}) \]
in $\cat{D}(B^{\mrm{ce}})$ 
for every compound resolution $\bsym{P}$ of the pair 
$(M^{\mrm{l}}, M^{\mrm{r}})$ over $\til{B} / \til{A}$,
satisfying this condition\tup{:}
\begin{enumerate}
\item[($\dag$)] Let $\bsym{P}_0$ and $\bsym{P}_1$ be compound resolutions of 
the pair $(M^{\mrm{l}}, M^{\mrm{r}})$ over $\til{B} / \til{A}$,
and let 
$\bsym{\eta} : \bsym{P}_0 \to \bsym{P}_1$ 
be a compound morphism above 
$(\opn{id}_{M^{\mrm{l}}}, \opn{id}_{M^{\mrm{r}}})$
and
$\opn{id}_{\til{B}} / \opn{id}_{\til{A}}$. 
Then  the diagram 
\[ \UseTips \xymatrix @C=12ex @R=6ex {
\opn{Rect}_{B / A}^{\til{B} / \til{A}}(M^{\mrm{l}}, M^{\mrm{r}})
\ar[d]_{ \opn{rect}(\bsym{P}_0) }
\ar[dr]^{ \opn{rect}(\bsym{P}_1) }
\\
\opn{Rect}_{B / A}^{\til{B} / \til{A}}(\bsym{P}_0) 
\ar[r]_{ \opn{Rect}_{\mrm{id} / \mrm{id}}^{\mrm{id} / \mrm{id}}(\bsym{\eta}) }
&
\opn{Rect}_{B / A}^{\til{B} / \til{A}}(\bsym{P}_1) 
} \]
of isomorphisms in $\cat{D}(B^{\mrm{ce}})$ is commutative. 
\end{enumerate}
\end{prop}

\begin{proof}
The uniqueness is clear. As for existence, let us fix some 
resolution $\bsym{P}$ of $(M^{\mrm{l}}, M^{\mrm{r}})$, 
and define 
\[ \opn{Rect}_{B / A}^{\til{B} / \til{A}}(M^{\mrm{l}}, M^{\mrm{r}}) :=
\opn{Rect}_{B / A}^{\til{B} / \til{A}}(\bsym{P}) . \]
The isomorphism $\opn{rect}(\bsym{P})$ is the identity. 

Given any resolution $\bsym{P}_{0}$ of $(M^{\mrm{l}}, M^{\mrm{r}})$, 
let $\bsym{\eta}_0 : \bsym{P} \to \bsym{P}_0$ 
be any morphism above the identity. Define the isomorphism
\[ \opn{rect}(\bsym{P}_0) :=
\opn{Rect}_{\mrm{id} / \mrm{id}}^{\mrm{id} / \mrm{id}}(\bsym{\eta}_0) . \]
By Lemma \ref{lem:916} this does not depend on the choice of 
$\bsym{\eta}_0$.

To verify condition ($\dag$), let 
$\bsym{\eta}_1 : \bsym{P} \to \bsym{P}_1$
be any morphism above the identity. Then, according to Lemma \ref{lem:919},
we have
\[ \opn{Rect}_{\mrm{id} / \mrm{id}}^{\mrm{id} / \mrm{id}}(\bsym{\eta}_1) =
\opn{Rect}_{\mrm{id} / \mrm{id}}^{\mrm{id} / \mrm{id}}(\bsym{\eta}) \circ
\opn{Rect}_{\mrm{id} / \mrm{id}}^{\mrm{id} / \mrm{id}}(\bsym{\eta}_0) . \]
\end{proof}

\begin{prop} \label{prop:916}
In the situation of Setup \tup{\ref{set:940}}, with $k = 0, 1$ only, there is a 
unique morphism 
\[  \opn{Rect}_{w / v}^{\til{w} / \til{v}}
(\th^{\mrm{l}}, \th^{\mrm{r}}) :
\opn{Rect}_{B_0 / A_0}^{\til{B}_0 / \til{A}_0}(M_0^{\mrm{l}}, M_0^{\mrm{r}}) \to
\opn{Rect}_{B_1 / A_1}^{\til{B}_1 / \til{A}_1}(M_1^{\mrm{l}}, M_1^{\mrm{r}}) \]
in $\cat{D}(B_1^{\mrm{ce}})$
satisfying this condition\tup{:}
\begin{enumerate}
\item[($\dag \dag$)]
For $k = 0, 1$ let $\bsym{P}_k$ be a compound resolution of 
$(M_k^{\mrm{l}},  M_k^{\mrm{r}})$ over $\til{B}_k / \til{A}_k$, 
and let  
$\bsym{\eta} : \bsym{P}_0 \to \bsym{P}_1$ 
be a compound morphism above 
$(\th^{\mrm{l}},  \th^{\mrm{r}})$ and $\til{w} / \til{v}$. 
Then the diagram 
\[ \UseTips \xymatrix @C=16ex @R=6ex {
\opn{Rect}_{B_0 / A_0}^{\til{B}_0 / \til{A}_0}(M_0^{\mrm{l}}, M_0^{\mrm{r}})
\ar[d]_{ \opn{rect}(\bsym{P}_0) }
\ar[r]^{ \opn{Rect}_{w / v}^{\til{w} / \til{v}}
(\th^{\mrm{l}}, \th^{\mrm{r}}) }
&
\opn{Rect}_{B_1 / A_1}^{\til{B}_1 / \til{A}_1}(M_1^{\mrm{l}}, M_1^{\mrm{r}})
\ar[d]^{ \opn{rect}(\bsym{P}_1) }
\\
\opn{Rect}_{B_0 / A_0}^{\til{B}_0 / \til{A}_0}(\bsym{P}_0) 
\ar[r]^{ \opn{Rect}_{w / v}^{\til{w} / \til{v}}(\bsym{\eta}) }
&
\opn{Rect}_{B_1 / A_1}^{\til{B}_1 / \til{A}_1}(\bsym{P}_1) 
} \]
of morphisms in $\cat{D}(B_1^{\mrm{ce}})$ is commutative.
\end{enumerate} 
\end{prop}

\begin{proof}
Choose a compound resolution $\bsym{P}_k^{\mrm{bs}}$ of  
$(M_k^{\mrm{l}},  M_k^{\mrm{r}})$ over $\til{B}_k / \til{A}_k$, 
for $k = 0, 1$. Then choose a compound morphism 
$\bsym{\eta}^{\mrm{bs}} : \bsym{P}^{\mrm{bs}}_0 \to \bsym{P}^{\mrm{bs}}_1$.
(The superscript ``bs'' stands for ``basic''.)
This can be done by Lemma \ref{lem:915}. 
Let 
$\opn{Rect}_{w / v}^{\til{w} / \til{v}}(\th^{\mrm{l}}, \th^{\mrm{r}})$
be the the unique morphism for which condition ($\dag \dag$) holds, 
with respect to these choices.

We have to prove that condition ($\dag \dag$) holds for an arbitrary choice of 
compound resolutions $\bsym{P}_k$ and a compound morphism 
$\bsym{\eta}$. Let us choose compound morphisms
$\bsym{\eta}_k :  \bsym{P}^{\mrm{bs}}_k \to \bsym{P}_k$
above $(\opn{id}_{M^{\mrm{l}}}, \opn{id}_{M^{\mrm{r}}})$
and
$\opn{id}_{\til{B}_k} / \opn{id}_{\til{A}_k}$. 
Consider the following diagram in $\cat{D}(B_1^{\mrm{ce}})$.

\[ \UseTips \xymatrix @C=16ex @R=8ex {
\opn{Rect}_{B_0 / A_0}^{\til{B}_0 / \til{A}_0}(M_0^{\mrm{l}}, M_0^{\mrm{r}})
\ar[d]^{ \opn{rect}(\bsym{P}^{\mrm{bs}}_0) }
\ar[r]^{ \opn{Rect}_{w / v}^{\til{w} / \til{v}}
(\th^{\mrm{l}}, \th^{\mrm{r}}) }
\ar @/_4.5em/ [dd]_(0.7){ \opn{rect}(\bsym{P}_0) }
&
\opn{Rect}_{B_1 / A_1}^{\til{B}_1 / \til{A}_1}(M_1^{\mrm{l}}, M_1^{\mrm{r}})
\ar[d]_{ \opn{rect}(\bsym{P}^{\mrm{bs}}_1) }
\ar @/^4.5em/ [dd]^(0.7){ \opn{rect}(\bsym{P}_1) }
\\
\opn{Rect}_{B_0 / A_0}^{\til{B}_0 / \til{A}_0}(\bsym{P}^{\mrm{bs}}_0) 
\ar[r]^{ \opn{Rect}_{w / v}^{\til{w} / \til{v}}(\bsym{\eta}^{\mrm{bs}}) }
\ar[d]^{ \opn{Rect}_{\mrm{id} / \mrm{id}}^{\mrm{id} / \mrm{id}} (\bsym{\eta}_0) 
}
&
\opn{Rect}_{B_1 / A_1}^{\til{B}_1 / \til{A}_1}(\bsym{P}^{\mrm{bs}}_1)
\ar[d]_{ \opn{Rect}_{\mrm{id} / \mrm{id}}^{\mrm{id} / \mrm{id}}(\bsym{\eta}_1)}
\\
\opn{Rect}_{B_0 / A_0}^{\til{B}_0 / \til{A}_0}(\bsym{P}_0) 
\ar[r]^{ \opn{Rect}_{w / v}^{\til{w} / \til{v}}(\bsym{\eta}) }
&
\opn{Rect}_{B_1 / A_1}^{\til{B}_1 / \til{A}_1}(\bsym{P}_1)
} \]

\noindent 
The top square is commutative by definition. The bottom square is commutative 
by Lemmas \ref{lem:915} and \ref{lem:919}. The half-moons are commutative by 
condition ($\dag$) of Proposition \ref{prop:915}. Therefore the outer paths are 
equal; and this is what we had to prove. 
\end{proof}

\begin{prop} \label{prop:940}
In the situation of Setup \tup{\ref{set:940}}, with $k = 0, 1, 2$,
the diagram 
\[ \UseTips \xymatrix @C=18ex @R=8ex {
\opn{Rect}_{B_0 / A_0}^{\til{B}_0 / \til{A}_0}(M_0^{\mrm{l}}, M_0^{\mrm{r}})
\ar[dr]_{ \opn{Rect}_{(w_1 / v_1) \circ (w_2 / v_2)}
^{(\til{w}_1 / \til{v}_1) \circ (\til{w}_2 / \til{v}_2)}
(\th_2^{\mrm{l}} \circ \th_1^{\mrm{l}}, \th_2^{\mrm{r}}  \circ 
\th_1^{\mrm{r}}) \quad \quad \quad \quad }
\ar[r]^{ \opn{Rect}_{w_1 / v_1}^{\til{w}_1 / \til{v}_1}
(\th_1^{\mrm{l}}, \th_1^{\mrm{r}}) }
&
\opn{Rect}_{B_1 / A_1}^{\til{B}_1 / \til{A}_1}(M_1^{\mrm{l}}, M_1^{\mrm{r}})
\ar[d]^{ \opn{Rect}_{w_2 / v_2}^{\til{w}_2 / \til{v}_2}
(\th_2^{\mrm{l}}, \th_2^{\mrm{r}}) }
\\
&
\opn{Rect}_{B_2 / A_2}^{\til{B}_2 / \til{A}_2}(M_2^{\mrm{l}}, M_2^{\mrm{r}})
} \]
of morphisms in $\cat{D}(B_2^{\mrm{ce}})$ is commutative.

If $B_1 / A_1 = B_0 / A_0$, 
$\til{B}_1 / \til{A}_1 = \til{B}_0 / \til{A}_0$,
$(M^{\mrm{l}}_1, M^{\mrm{r}}_1) = (M^{\mrm{l}}_0, M^{\mrm{r}}_0)$, 
and the morphisms 
$w_1 / v_1$, 
$\til{w}_1 / \til{v}_1$ and 
$(\th_1^{\mrm{l}}, \th_1^{\mrm{r}})$
are the identity automorphisms, then 
$\opn{Rect}_{w_1 / v_1}^{\til{w}_1 / \til{v}_1}
(\th_1^{\mrm{l}}, \th_1^{\mrm{r}})$
is the identity automorphism of
$\opn{Rect}_{B_0 / A_0}^{\til{B}_0 / \til{A}_0}(M_0^{\mrm{l}}, M_0^{\mrm{r}})$. 
\end{prop}

\begin{proof}
The triangular diagram is commutative by Lemmas \ref{lem:915} and \ref{lem:919},
with condition ($\dag \dag$) of Proposition \ref{prop:916}.
The assertion about the identity automorphisms is true by Lemma \ref{lem:925}.
\end{proof}

\begin{rem} \label{rem:950}
We could have defined the rectangle object directly as 
\[ \opn{Rect}_{B / A}^{\til{B} / \til{A}}(M^{\mrm{l}},  M^{\mrm{r}}) := 
\opn{RHom}_{\til{B}^{\mrm{en}}}(B, M^{\mrm{l}} \ot_{\til{A}}^{\mrm{L}}  
M^{\mrm{r}}) \in \cat{D}(B^{\mrm{ce}}) , \]
i.e.\ as the composition of these three functors: 
\[ \opn{For}_{} \times \opn{For} :
\cat{D}(B) \times \cat{D}(B^{\mrm{op}}) \to 
\cat{D}(\til{B}) \times \cat{D}(\til{B}^{\mrm{op}}) , \]
\[ - \ot_{\til{A}}^{\mrm{L}} - : 
\cat{D}(\til{B}) \times \cat{D}(\til{B}^{\mrm{op}}) \to 
\cat{D}(\til{B}^{\mrm{en}}) \]
and 
\[ \opn{RHom}_{\til{B}^{\mrm{en}}}(B, -) : \cat{D}(\til{B}^{\mrm{en}})
\to \cat{D}(B^{\mrm{ce}}) . \]
In other words, Proposition \ref{prop:915} would have become a definition. 

However, this approach would have made it hard to make precise sense of the 
morphism 
$\opn{Rect}_{w / v}^{\til{w} / \til{v}}(\th^{\mrm{l}}, \th^{\mrm{r}})$,
and to prove the functoriality result (Proposition \ref{prop:940}).
\end{rem}

\section{The Rectangle Operation}
\label{sec:rect-oper}

In this section we continue with the material of Section 
\ref{sec:pairs-DG-modules}, but with slightly different notation.
Recall that $\cat{PDGR}$ is the category of central pairs of DG rings, and 
$\cat{PDGR}^{\leq 0}$ is the full subcategory of nonpositive central pairs. 
See Definitions \ref{dfn:910} and \ref{dfn:911}.

The main result here is Theorem \ref{thm:840}. 
Throughout this section we work in the following setup.

\begin{setup} \label{set:930}
We are given:
\begin{enumerate}
\item A central morphism 
$w / v : B' / A' \to B / A$ in $\cat{PDGR}^{\leq 0}$
(namely a morphism in $\cat{PDGR}^{\leq 0}_{\vec{\mrm{ce}}}$, 
see Definition \ref{dfn:911}(3)).

\item K-flat resolutions 
$s / r : \til{B} / \til{A} \to B / A$
and 
$s' / r' : \til{B}' / \til{A}' \to B' / A'$
in $\cat{PDGR}^{\leq 0}$, and morphisms 
\[ \til{w}_0 / \til{v}_0 , \, \til{w}_1 / \til{v}_1 : 
\til{B}' / \til{A}' \to \til{B} / \til{A} \]
in $\cat{PDGR}^{\leq 0}$ above $w / v$. 
See Definitions \ref{dfn:912}, \ref{dfn:913} and \ref{dfn:52}.

\item An object 
$(M^{\mrm{l}}, M^{\mrm{r}})$ in 
$\cat{D}(B) \times \cat{D}(B^{\mrm{op}})$,
an  object 
$(M'^{\, \mrm{l}}, M'^{\, \mrm{r}})$ in 
$\cat{D}(B') \times \cat{D}(B'^{\mrm{\, op}})$,
and a morphism 
\[ (\th^{\mrm{l}},  \th^{\mrm{r}}) : (M^{\mrm{l}}, M^{\mrm{r}}) \to 
(M'^{\, \mrm{l}}, M'^{\, \mrm{r}}) \]
in 
$\cat{D}(B') \times \cat{D}(B'^{\mrm{\, op}})$.
\end{enumerate}
\end{setup}

The DG ring input is depicted in the following commutative diagram in the 
category $\cat{PDGR}^{\leq 0}$, for $i = 0, 1$. 

\[ \UseTips \xymatrix @C=8ex @R=6ex {
\til{B}' / \til{A}'
\ar[r]^{ \til{w}_i / \til{v}_i }
\ar[d]_{ s' / r' }
&
\til{B} / \til{A}
\ar[d]^{ s / r }
\\
B' / A'
\ar[r]^{ w / v }
&
B / A
} \]

The four central homomorphisms, belonging to the four central pairs above, are
$u : A \to B$, $u' : A' \to B'$,
$\til{u} : \til{A} \to \til{B}$ and 
$\til{u}' : \til{A}' \to \til{B}'$.

From here until Lemma \ref{lem:827} (inclusive), we also assume the next 
condition. 

\begin{cond} \label{cond:930} \mbox{}
\begin{enumerate}
\rmitem{i} The homomorphism $\til{u}' : \til{A}' \to \til{B}'$ is 
noncommutative semi-free (Definition \ref{dfn:60}(2)). 
\rmitem{ii} The homomorphisms 
$\til{v}_0, \til{v}_1 : \til{A}' \to \til{A}$ are equal. 
\end{enumerate}
\end{cond}

In view of item (ii) in Condition \ref{cond:930}, we may write 
$\til{v} := \til{v}_0 = \til{v}_1$; this is a homomorphism 
$\til{v} : \til{A}' \to \til{A}$
in $\cat{DGR}^{\leq 0}_{\mrm{sc}}$,
lying above $v : A' \to A$. 
The DG ring input of Setup \ref{set:930} and Condition \ref{cond:930} can be 
summarized as the following commutative diagrams in the category 
$\cat{DGR}^{\leq 0}\centover \til{A}'$,
for $i = 0, 1$:

\[ \UseTips \xymatrix @C=5ex @R=5ex {
& 
\til{B}'
\ar[dr]^{w \circ s'}
\\
\til{A}'
\ar[rr]
\ar[ur]^{\til{u}'} 
\ar[dr]_{\til{u} \circ \til{v}}
& 
&
B
\\
&
\til{B}
\ar[ur]_{s} 
}
\qquad  \qquad 
\xymatrix @C=5ex @R=5ex {
& 
\til{B}' 
\ar[dr]^{w \circ s'}
\ar[dd]_{\til{w}_i} 
\\
\til{A}'
\ar[ur]^{\til{u}'}
\ar[dr]_{\til{u} \circ \til{v}}
& 
&
B
\\
&
\til{B}
\ar[ur]_{s}
} \]

Recall the cylinder DG ring from Definition \ref{dfn:700}. 

\begin{lem} \label{lem:640}
Under Condition \tup{\ref{cond:930}}, there is a homomorphism 
\[ \til{w}_{\mrm{cyl}} : \til{B}' \to \opn{Cyl}(\til{B}) \] 
in $\cat{DGR} \centover \til{A}'$, such that diagram \tup{(\ref{eqn:359})}
is commutative for $i = 0, 1$.
\end{lem}

\begin{equation} \label{eqn:359}
\UseTips \xymatrix @C=8ex @R=6ex {
\til{B}'
\ar[r]^(0.4){ \til{w}_{\mrm{cyl}} }
\ar[d]_{s'}
\ar@(u,u)[rr]^{ \til{w}_i }
&
\opn{Cyl}(\til{B})
\ar[r]^(0.6){ \eta_i }
\ar[d]_{ \opn{Cyl}(s) }
&
\til{B}
\ar[d]_{s}
\\
B'
\ar[r]^(0.4){\ep \circ w}
&
\opn{Cyl}(B)
\ar[r]^(0.6){\eta_i}
&
B
}
\end{equation}
 
\begin{proof}
We know that $\til{u}' : \til{A}' \to \til{B}'$ is noncommutative semi-free, 
and $s$ is a surjective quasi-iso\-morphism. So Theorem \ref{thm:300} applies.
We obtain an $\til{A}'$-linear DG ring homotopy 
$\ga : \til{w}_0 \twoto \til{w}_1$ such that $s \circ \ga = 0$.
By Proposition \ref{prop:300} we deduce the 
existence of a DG ring homomorphism 
$\til{w}_{\mrm{cyl}} : \til{B}' \to \opn{Cyl}(\til{B})$. 
The formula, for $b \in \til{B}'$, is 
\[ \til{w}_{\mrm{cyl}}(b) = 
\bmat{\til{w}_0(b) & \ \xi \cd \ga(b) \\[0.3em] 0 & \til{w}_1(b)} 
. \]
This shows the commutativity of the half-moon in the diagram. 
Next, because the element $\xi$ commutes with the homomorphism $s$, and 
because $s \circ \ga = 0$, we have  
\[ \opn{Cyl}(s) (\til{w}_{\mrm{cyl}}(b)) = 
\bmat{ s(\til{w}_0(b)) & \ s(\xi \cd \ga(b)) 
\\[0.3em] 0 & s(\til{w}_1(b))} =  
\bmat{ w(s'(b)) & 0 \\ 0 & w(s'(b)) } . \]
We deduce the commutativity of the first square in the diagram. 
The second square is trivially commutative.
\end{proof}

There are central homomorphisms 
$\ep \circ u : A \to \opn{Cyl}(B)$
and 
$\ep \circ \til{u} : \til{A} \to \opn{Cyl}(\til{B})$, 
and thus two more central pairs
$\opn{Cyl}(B) / A$ and $\opn{Cyl}(\til{B}) / \til{A}$. 
These central pairs fit into a commutative diagram 

\begin{equation} \label{eqn:942}
\UseTips \xymatrix @C=10ex @R=8ex {
\til{B}' / \til{A}'
\ar[r]^(0.45){ \til{w}_{\mrm{cyl}} / \til{v} }
\ar[d]_{ s' / r' }
\ar@(u,u)[rr]^{ \til{w}_i / \til{v} }
&
\opn{Cyl}(\til{B}) / \til{A}
\ar[r]^(0.55){ \eta_i / \opn{id} }
\ar[d]_{ \opn{Cyl}(s) / r }
&
\til{B} / \til{A}
\ar[d]_{ s / r }
\\
B' / A'
\ar[r]^(0.45){ (\ep \circ w) / v }
\ar@(d,d)[rr]_{ w / v }
&
\opn{Cyl}(B) / A
\ar[r]^(0.55){ \eta_i / \opn{id} }
&
B / A
} 
\end{equation}

\medskip \noindent
in the category $\cat{PDGR}$, for $i = 0, 1$. Moreover, the morphisms between 
the pairs in the lower row are central, namely they are in 
$\cat{PDGR}_{\vec{\mrm{ce}}}$. 
The pairs in the upper row are K-flat, and the vertical morphisms are 
quasi-isomorphisms. 

Consider the DG module 
$\opn{Cyl}(M^{\mrm{l}}) \in \cat{M}(\opn{Cyl}(B))$. 
Define homomorphisms
\[ \eta_0^{\vee} : M^{\mrm{l}} \to \opn{Cyl}(M^{\mrm{l}}) , \quad
\eta_0^{\vee}(m) := \bmat{ m & 0 \\ 0 & 0 } \]
and 
\[ \eta_1^{\vee} : M^{\mrm{l}} \to \opn{Cyl}(M^{\mrm{l}}) , \quad
\eta_1^{\vee}(m) := \bmat{ 0 & 0 \\ 0 & m } . \]
Then 
$\eta_i^{\vee} : \opn{For}_{\eta_i}(M^{\mrm{l}}) \to \opn{Cyl}(M^{\mrm{l}})$
is a homomorphism in $\cat{M}(\opn{Cyl}(B))$. 
Next, define a homomorphism
\[ \ep^{\vee} : \opn{Cyl}(M^{\mrm{l}}) \to M^{\mrm{l}} , \quad
\ep^{\vee}(\bmat{ m_0 & n \\ 0 & m_1 }) :=  m_0 + m_1 . \]
Then 
$\ep^{\vee} : \opn{For}_{\ep}(\opn{Cyl}(M^{\mrm{l}})) \to M^{\mrm{l}}$
is a homomorphism in $\cat{M}(B)$. 
And 
$\ep^{\vee} \circ \eta_i^{\vee} = \opn{id}$. 
All this can be done for 
$M^{\mrm{r}}$ too. 

Consider the next commutative diagram, for $i = 0, 1$, 
in the category 
$\cat{D}(B') \times \cat{D}(B'^{\mrm{\, op}})$.

\begin{equation} \label{eqn:943}
 \UseTips \xymatrix @C=14ex @R=6ex {
(M'^{\, \mrm{l}}, M'^{\, \mrm{r}})
&
\bigl( \opn{Cyl}(M^{\mrm{l}}), \opn{Cyl}(M^{\mrm{r}}) \bigr)
\ar[l]_(0.55){ ( \th^{\mrm{l}} \circ \ep^{\vee}, 
\th^{\mrm{r}} \circ \ep^{\vee}) } 
&
(M^{\mrm{l}}, M^{\mrm{r}})
\ar[l]_(0.4){ (\eta_i^{\vee}, \eta_i^{\vee}) } 
\ar@(d,d)[ll]^{ (\th^{\mrm{l}},  \th^{\mrm{r}}) }
} 
\end{equation}

\begin{lem} \label{lem:827}
Under Condition \tup{\ref{cond:930}}, there is equality 
\begin{equation} \label{eqn:944}
\opn{Rect}_{w / v}^{\til{w}_0 / \til{v}}(\th^{\mrm{l}},  \th^{\mrm{r}}) 
= 
\opn{Rect}_{w / v}^{\til{w}_1 / \til{v}}(\th^{\mrm{l}},  \th^{\mrm{r}})
\end{equation}
of morphisms 
\[ \opn{Rect}_{B / A}^{\til{B} / \til{A}}(M^{\mrm{l}}, M^{\mrm{r}}) \to
\opn{Rect}_{B' / A'}^{\til{B}' / \til{A}'}(M'^{\, \mrm{l}}, M'^{\, \mrm{r}}) \]
in $\cat{D}(B'^{\mrm{\, ce}})$.
\end{lem}

\begin{proof}
Step 1. Here we assume that
$B' / A' = B / A$, 
$\til{B}' / \til{A}' = \til{B} / \til{A}$, 
the morphisms $w / v$,  $\til{w}_0 / \til{v}$ and $\til{w}_1 / \til{v}$
are the identity automorphisms,
and there is equality $s' / r' = s / r$.  
Also we assume that 
$(M'^{\, \mrm{l}}, M'^{\, \mrm{r}}) = (M^{\mrm{l}}, M^{\mrm{r}})$, 
and the morphism 
$(\th^{\mrm{l}},  \th^{\mrm{r}})$ is the identity automorphism. 
We may also assume that 
$\til{w}_{\mrm{cyl}} = \ep \circ \til{w}$,
but that is not important. 
By functoriality (Proposition \ref{prop:940}), we have
\begin{equation} \label{eqn:940}
\opn{Rect}_{\ep \circ w / v}^{\til{w}_{\mrm{cyl}} / \til{v}}
(\th^{\mrm{l}} \circ \ep^{\vee},  \th^{\mrm{r}} \circ \ep^{\vee})
\circ 
\opn{Rect}_{\eta_i / \opn{id}}^{\eta_i / \opn{id}}
(\eta_i^{\vee}, \eta_i^{\vee})
= \opn{Rect}_{w / v}^{\til{w}_i / \til{v}}(\th^{\mrm{l}},  \th^{\mrm{r}}) ,
\end{equation}
in $\cat{D}(B^{\mrm{ce}})$, for $i = 0, 1$. 
Compare to diagrams (\ref{eqn:942}) and (\ref{eqn:943}). 
By Lemma \ref{lem:917}, all three morphisms above are isomorphisms.
Some caution is needed here: the morphism 
$\opn{Rect}_{\eta_i / \opn{id}}^{\eta_i / \opn{id}}
(\eta_i^{\vee}, \eta_i^{\vee})$
is actually in the category 
$\cat{D}(\opn{Cyl}(B)^{\mrm{ce}})$, and the forgetful functor 
\[ \opn{For}_{\ep^{\mrm{ce}}} : \cat{D}(\opn{Cyl}(B)^{\mrm{ce}}) \to 
\cat{D}(B^{\mrm{ce}}) \]
has been hidden. But but we know (see Proposition \ref{prop:700}) that the DG 
ring homomorphism 
$\ep : B \to \opn{Cyl}(B)$ induces an isomorphism 
$\ep^{\mrm{ce}} :  B^{\mrm{ce}} \to \opn{Cyl}(B)^{\mrm{ce}}$.
Therefore the functor 
$\opn{For}_{\ep^{\mrm{ce}}}$ is an equivalence. 

Since $\til{w}_1 = \til{w}_0$ here, formula (\ref{eqn:940}) for $i = 0, 1$
leads us to conclude that 
\begin{equation} \label{eqn:930}
\opn{Rect}_{\eta_1 / \opn{id}}^{\eta_1 / \opn{id}}
(\eta_1^{\vee}, \eta_1^{\vee}) = 
\opn{Rect}_{\eta_0 / \opn{id}}^{\eta_0 / \opn{id}} 
(\eta_0^{\vee}, \eta_0^{\vee}), 
\end{equation}
as isomorphisms 
\[ \opn{Rect}_{B / A}^{\til{B} / \til{A}}(M^{\mrm{l}}, M^{\mrm{r}}) \to
\opn{Rect}_{\opn{Cyl}(B) / A}^{\opn{Cyl}(\til{B}) / \til{A}}
\bigl( \opn{Cyl}(M^{\mrm{l}}), \opn{Cyl}(M^{\mrm{r}}) \bigr) \]
in $\cat{D}(\opn{Cyl}(B)^{\mrm{ce}})$.
 
\medskip \noindent 
Step 2. Now there are no special assumptions. We choose a homomorphism 
$\til{w}_{\mrm{cyl}}$ as in Lemma \ref{lem:640}.
By functoriality, we again have the equalities (\ref{eqn:940}), for 
$i = 0, 1$. Formula (\ref{eqn:930}) from part (1) applies here, because it does 
not depend on $B' / A'$. We deduce that (\ref{eqn:944}) holds. 
\end{proof}

Here is the key technical result of the paper (the noncommutative version). 

\begin{thm}[Homotopy Invariance] \label{thm:840}
In the situation of Setup \tup{\ref{set:930}}, there is equality 
\[ \opn{Rect}_{w / v}^{\til{w}_0 / \til{v}_0}(\th^{\mrm{l}},  \th^{\mrm{r}}) 
= 
\opn{Rect}_{w / v}^{\til{w}_1 / \til{v}_1}(\th^{\mrm{l}},  \th^{\mrm{r}}) \]
of morphisms 
\[ \opn{Rect}_{B / A}^{\til{B} / \til{A}}(M^{\mrm{l}}, M^{\mrm{r}}) \to
\opn{Rect}_{B' / A'}^{\til{B}' / \til{A}'}(M'^{\, \mrm{l}}, M'^{\, \mrm{r}}) \]
in $\cat{D}(B'^{\mrm{\, ce}})$.
\end{thm}

\begin{proof}
Choose the following resolutions in $\cat{PDGR}^{\leq 0}$: 
a strict K-flat resolution \lb
$s^{\diamondsuit} / \mrm{id} : \til{B}^{\diamondsuit} / A \to B / A$, 
a strict noncommutative semi-free resolution 
$s^{\heartsuit} / \mrm{id} : \til{B}^{\heartsuit} / \til{A} \to 
\til{B} / \til{A}$, 
and a strict noncommutative semi-free resolution
$s'^{\, \diamondsuit} / \mrm{id} : 
\til{B}'^{\, \diamondsuit} / \til{A}' \lb \to \til{B}' / \til{A}'$. 
These exist by Theorem \ref{thm:305}. 

There is a homomorphism 
$s^{\dag} : \til{B}^{\heartsuit} \to \til{B}^{\diamondsuit}$
in $\cat{DGR}^{\leq 0} \centover \til{A}$
such that 
$s^{\diamondsuit} \circ s^{\dag} = s \circ s^{\heartsuit}$.
For $i = 0, 1$ there is a homomorphism 
$\til{w}_i^{\diamondsuit} : \til{B}'^{\, \diamondsuit} \to 
\til{B}^{\heartsuit}$
in $\cat{DGR}^{\leq 0} \centover \til{A}'$, such that 
$s^{\heartsuit} \circ \til{w}_i^{\diamondsuit} = \til{w}_i 
\circ s'^{\, \diamondsuit}$.
These homomorphisms exist by Theorem \ref{thm:306}. 

The resolutions and the homomorphisms between them, from the two paragraphs 
above, fit into the following commutative diagram in $\cat{PDGR}^{\leq 0}$, for 
$i = 0, 1$. 
The central pairs
$\til{B}'^{\, \diamondsuit} / \til{A}'$ and 
$\til{B}' / \til{A}'$ are K-flat resolutions of $B' / A'$, and 
the central pairs
$\til{B}^{\heartsuit} / \til{A}$, 
$\til{B}^{\diamondsuit} / A$ and $\til{B} / \til{A}$
are K-flat resolutions of $B / A$.

\begin{equation} \label{eqn:946}
\UseTips \xymatrix @C=10ex @R=8ex {
&
&
\til{B}^{\heartsuit} / \til{A}
\ar[d]_{ s^{\heartsuit} / \mrm{id} }
\ar[dr]^{ s^{\dag} / r }
\\
\til{B}'^{\, \diamondsuit} / \til{A}'
\ar @/^1.5em/ [urr]^{ \til{w}_i^{\diamondsuit} / \til{v}_i }
\ar[r]^{ s'^{\, \diamondsuit} / \mrm{id} }
\ar[dr]_{ (s' \circ s'^{\, \diamondsuit}) / r' }
&
\til{B}' / \til{A}'
\ar[r]^{ \til{w}_i / \til{v}_i }
\ar[d]^{ s' / r' }
&
\til{B} / \til{A}
\ar[d]^{ s / r }
&
\til{B}^{\diamondsuit} / A
\ar[dl]^{ s^{\diamondsuit} / \mrm{id} }
\\
&
B' / A'
\ar[r]^{ w / v }
&
B / A
}
\end{equation}

\medskip
Let us introduce the temporary abbreviations 
\[ \opn{R}(\til{B} / \til{A}) := 
\opn{Rect}_{B / A}^{\til{B} / \til{A}}(M^{\mrm{l}}, M^{\mrm{r}}) , \quad
\opn{R}(\til{B}' / \til{A}') := 
\opn{Rect}_{B' / A'}^{\til{B}' / \til{A}'}(M'^{\, \mrm{l}}, M'^{\, \mrm{r}}) , 
\]
\[ \opn{R}(\til{w}_i / \til{v}_i) :=
\opn{Rect}_{w / v}^{\til{w}_i / \til{v}_i}(\th^{\mrm{l}},  \th^{\mrm{r}}) , 
\quad
\opn{R}(s^{\dag} / r) :=
\opn{Rect}_{\mrm{id} / \mrm{id}}^{s^{\dag} / r}(\mrm{id} , \mrm{id}) , \]
etc. Applying functoriality (Proposition \ref{prop:940}) to the commutative 
diagram (\ref{eqn:946}), we obtain a commutative diagram 

\begin{equation} \label{eqn:947}
\UseTips \xymatrix @C=10ex @R=8ex {
&
&
\opn{R}(\til{B}^{\heartsuit} / \til{A})
\ar @/_1.5em/ [dll]_{ \opn{R}(\til{w}_i^{\diamondsuit} / \til{v}_i) }
\\
\opn{R}(\til{B}'^{\, \diamondsuit} / \til{A}')
&
\opn{R}(\til{B}' / \til{A}')
\ar[l]_(0.45){ \opn{R}(s'^{\, \diamondsuit} / \mrm{id}) }
&
\opn{R}(\til{B} / \til{A})
\ar[u]^{ \opn{R}(s^{\heartsuit} / \mrm{id}) }
\ar[l]_{ \opn{R}(\til{w}_i / \til{v}_i) }
&
\opn{R}(\til{B}^{\diamondsuit} / A)
\ar[ul]_{ \opn{R}(s^{\dag} / r) }
}
\end{equation}

\noindent
in $\cat{D}(B'^{\, \mrm{ce}})$, for $i = 0, 1$. All the morphisms whose names 
contain the letter ``s'' are isomorphisms. 

Note that 
$r \circ \til{v}_i = v \circ r'$. We get a commutative diagram 

\begin{equation} \label{eqn:948}
\UseTips \xymatrix @C=10ex @R=8ex {
&
&
\opn{R}(\til{B}^{\heartsuit} / \til{A})
\ar @/_1.5em/ [dll]_{ \opn{R}(\til{w}_i^{\diamondsuit} / \til{v}_i) }
\\
\opn{R}(\til{B}'^{\, \diamondsuit} / \til{A}')
&
&
&
\opn{R}(\til{B}^{\diamondsuit} / A)
\ar[ul]_{ \opn{R}(s^{\dag} / r) }
\ar[lll]_{ \opn{R}( (s^{\dag} \circ \til{w}_i^{\diamondsuit}) / (v \circ r')) } 
}
\end{equation}

\noindent
in $\cat{D}(B'^{\, \mrm{ce}})$, for $i = 0, 1$.
Because Condition \ref{cond:930} is satisfied for the morphisms of resolutions 
\[ (s^{\dag} \circ \til{w}_i^{\diamondsuit}) / (v \circ r') : 
\til{B}'^{\, \diamondsuit} / \til{A}' \to 
\til{B}^{\diamondsuit} / A , \]
Lemma \ref{lem:827} tells us that 
\[ \opn{R}( (s^{\dag} \circ \til{w}_0^{\diamondsuit}) / (v \circ r')) =
\opn{R}( (s^{\dag} \circ \til{w}_1^{\diamondsuit}) / (v \circ r')) . \]
Since $\opn{R}(y^{\dag} / r)$ is an isomorphism, we conclude that 
\[ \opn{R}(\til{w}_0^{\diamondsuit} / \til{v}_0) = 
\opn{R}(\til{w}_1^{\diamondsuit} / \til{v}_1) . \]
Going back to diagram (\ref{eqn:947}), and using the fact that both 
$\opn{R}(s^{\heartsuit} / \mrm{id})$ and 
$\opn{R}(s'^{\, \diamondsuit} / \mrm{id})$
are isomorphisms, we see that 
\[ \opn{R}(\til{w}_0^{} / \til{v}_0) = \opn{R}(\til{w}_1^{} / \til{v}_1) . \]
\end{proof}

\begin{thm}[Existence of Rectangles] \label{thm:985}
Let $A$ be a commutative DG ring, and let $A \to B$ be a central homomorphism 
of nonpositive DG rings. Given a pair 
\[ (M^{\mrm{l}}, M^{\mrm{r}}) \in \cat{D}(B) \times \cat{D}(B^{\mrm{op}}) , \]
there is a DG module 
\[ \opn{Rect}_{B / A}(M^{\mrm{l}}, M^{\mrm{r}}) \in \cat{D}(B^{\mrm{ce}}) , \]
unique up to a unique isomorphism, 
together with an isomorphism
\[ \opn{rect}^{\til{B} / \til{A}} : 
\opn{Rect}_{B / A}(M^{\mrm{l}}, M^{\mrm{r}}) \iso 
\opn{Rect}_{B / A}^{\til{B} / \til{A}}(M^{\mrm{l}}, M^{\mrm{r}})  \]
for every K-flat resolution 
$\til{B} / \til{A}$ of $B / A$ in  $\cat{PDGR}^{\leq 0}$,
such that the following condition holds. 

\begin{enumerate}
\item[($*$)] Let 
$\til{w} / \til{v} : \til{B}' / \til{A}' \to  \til{B} / \til{A}$
be a morphism of K-flat resolutions of $B / A$ in  $\cat{PDGR}^{\leq 0}$. Then  
the diagram 
\[ \UseTips \xymatrix @C=20ex @R=8ex {
\opn{Rect}_{B / A}(M^{\mrm{l}}, M^{\mrm{r}})
\ar[d]_{ \opn{rect}^{\til{B} / \til{A}} }
\ar[dr]^{ \quad \opn{rect}^{\til{B}' / \til{A}'} }
\\
\opn{Rect}_{B / A}^{\til{B} / \til{A}}(M^{\mrm{l}}, M^{\mrm{r}})
\ar[r]_{ \opn{Rect}_{w / v}^{\til{w} / \til{v}}
(\opn{id}_{M^{\mrm{l}}}, \opn{id}_{M^{\mrm{r}}}) }
&
\opn{Rect}_{B / A}^{\til{B}' / \til{A}'}(M^{\mrm{l}}, M^{\mrm{r}}) 
} \]
of isomorphisms in $\cat{D}(B^{\mrm{ce}})$ is commutative. 
\end{enumerate}
\end{thm}

\begin{proof}
Choose a commutative semi-free resolution $\til{A}_{\mrm{un}} / \Z$ 
of $A / \Z$, 
and then choose a noncommutative semi-free resolution
$\til{B}_{\mrm{un}} / \til{A}_{\mrm{un}}$
of $B / \til{A}_{\mrm{un}}$, both in $\cat{PDGR}^{\leq 0}$.
This is possible by Theorem \ref{thm:305}. (The subscript ``un'' stands for 
``universal''.) Note that 
$\til{B}_{\mrm{un}} / \til{A}_{\mrm{un}}$
is a K-flat resolution of $B / A$. Define 
\[ \opn{Rect}_{B / A}(M^{\mrm{l}}, M^{\mrm{r}}) :=
\opn{Rect}_{B / A}^{\til{B}_{\mrm{un}} / \til{A}_{\mrm{un}}}
(M^{\mrm{l}}, M^{\mrm{r}}) . \]

Consider any K-flat resolution $\til{B} / \til{A}$ of $B / A$. According to 
Theorem \ref{thm:306} there is a morphism 
\[ \til{w}_{\mrm{un}} / \til{v}_{\mrm{un}} :
\til{B}_{\mrm{un}} / \til{A}_{\mrm{un}} \to \til{B} / \til{A}  \]
of resolutions of $B / A$. We define
\[ \opn{rect}^{\til{B} / \til{A}} : 
\opn{Rect}_{B / A}^{\til{B}_{\mrm{un}} / \til{A}_{\mrm{un}}}
(M^{\mrm{l}}, M^{\mrm{r}}) \iso 
\opn{Rect}_{B / A}^{\til{B} / \til{A}}(M^{\mrm{l}}, M^{\mrm{r}}) \]
to be 
\[ \opn{rect}^{\til{B} / \til{A}} := 
\opn{Rect}_{\mrm{id} / \mrm{id}}
^{\til{w}_{\mrm{un}} / \til{v}_{\mrm{un}}}
(\opn{id}_{M^{\mrm{l}}}, \opn{id}_{M^{\mrm{r}}})^{-1} . \]
Theorem \ref{thm:840} tells us that the isomorphism
$\opn{rect}^{\til{B} / \til{A}}$
does not depend on the choice of 
$\til{w}_{\mrm{un}} / \til{v}_{\mrm{un}}$. 

It remains to verify condition ($*$). Suppose 
$\til{B}' / \til{A}'$ is another K-flat resolution of $B / A$, and 
$\til{w} / \til{v} : \til{B}' / \til{A}' \to  \til{B} / \til{A}$
is a morphism of resolutions of $B / A$. Choose a morphism of resolutions 
\[ \til{w}'_{\mrm{un}} / \til{v}'_{\mrm{un}} :
\til{B}_{\mrm{un}} / \til{A}_{\mrm{un}} \to \til{B}' / \til{A}'  . \]
We have a diagram 
\[ \UseTips \xymatrix @C=20ex @R=8ex {
\opn{Rect}_{B / A}^{\til{B}_{\mrm{un}} / \til{A}_{\mrm{un}}}
(M^{\mrm{l}}, M^{\mrm{r}}) 
\\
\opn{Rect}_{B / A}^{\til{B} / \til{A}}(M^{\mrm{l}}, M^{\mrm{r}})
\ar[u]^{ \opn{Rect}_{\mrm{id} / \mrm{id}}
^{\til{w}_{\mrm{un}} / \til{v}_{\mrm{un}}}(\opn{id}, \opn{id}) }
\ar[r]_{ \opn{Rect}_{w / v}^{\til{w} / \til{v}}(\opn{id}, \opn{id}) }
&
\opn{Rect}_{B / A}^{\til{B}' / \til{A}'}(M^{\mrm{l}}, M^{\mrm{r}}) 
\ar[ul]_{ \quad \opn{Rect}_{\mrm{id} / \mrm{id}}
^{\til{w}'_{\mrm{un}} / \til{v}'_{\mrm{un}}}(\opn{id}, \opn{id}) }
} \]
of isomorphisms in $\cat{D}(B^{\mrm{ce}})$.
Proposition \ref{prop:940} says that 
\[ \opn{Rect}_{\mrm{id} / \mrm{id}}
^{\til{w}'_{\mrm{un}} / \til{v}'_{\mrm{un}}}(\opn{id}, \opn{id})
\circ 
\opn{Rect}_{w / v}^{\til{w} / \til{v}}(\opn{id}, \opn{id}) = 
\opn{Rect}_{\mrm{id} / \mrm{id}}
^{ (\til{w} \circ \til{w}'_{\mrm{un}}) / 
(\til{v} \circ \til{v}'_{\mrm{un}})}(\opn{id}, \opn{id}) . \]
Finally, Theorem \ref{thm:840} says that 
\[ \opn{Rect}_{\mrm{id} / \mrm{id}}
^{ (\til{w} \circ \til{w}'_{\mrm{un}}) / 
(\til{v} \circ \til{v}'_{\mrm{un}})}(\opn{id}, \opn{id}) = 
\opn{Rect}_{\mrm{id} / \mrm{id}}
^{\til{w}_{\mrm{un}} / \til{v}_{\mrm{un}}}(\opn{id}, \opn{id}) . \]
We that the diagram in condition ($*$) is indeed commutative.
\end{proof}

\begin{rem} \label{rem:720}
Let $\K$ be a regular finite dimensional noetherian commutative ring, and let 
$A$ be a cohomologically pseudo-noetherian commutative DG ring 
(see \cite[Section 3]{Ye2}). Suppose $\K \to A$ is a DG ring homomorphism, 
such that $\K \to \opn{H}^0(A)$ is essentially finite type.
Let $\cat{D}^+_{\mrm{f}}(A)$ be the full subcategory of $\cat{D}(A)$
consisting of DG modules whose cohomology is bounded below, 
and whose cohomology modules are finite  over $\opn{H}^0(A)$.

For $M, N \in \cat{D}^+_{\mrm{f}}(A)$ let us write 
\[ M \ot_{A / \K}^! N := \opn{Rect}_{A / \K}(M, N) \in \cat{D}(A) . \]
In \cite{Ga}, D. Gaitsgory states that when $\K$ is a field of 
characteristic $0$, this is operation makes $\cat{D}^+_{\mrm{f}}(A)$ into a 
{\em monoidal category}. No proof of this statement is given in \cite{Ga}.

Now assume $A$ is a ring (so it is just a noetherian commutative ring, 
essentially finite type over $\K$). According to \cite{YZ3} and its corrections 
in \cite{Ye5}, there is a {\em rigid dualizing complex} $R_{A / \K}$ over $A$ 
relative to $\K$. The functor $D := \opn{RHom}_A(-, R_{A / \K})$ is a duality 
of the category 
$\cat{D}_{\mrm{f}}(A)$, and it exchanges $\cat{D}^+_{\mrm{f}}(A)$
with $\cat{D}^-_{\mrm{f}}(A)$.
Recently L. Shaul \cite{Sh1} showed that there is a bifunctorial isomorphism 
\begin{equation} \label{eqn:390}
 M \ot_{A / \K}^! N \cong D \bigl( D(M) \ot^{\mrm{L}}_A D(N) \bigr)
\end{equation}
for $M, N \in \cat{D}^{\mrm{b}}_{\mrm{f}}(A)$. The proof relies on the 
reduction formula for Hochschild cohomology \cite[Theorem 1]{AILN}, \cite{Sh2}. 
The isomorphism (\ref{eqn:390}) implies that the operation 
$- \ot_{A / \K}^! -$ is a monoidal structure on the subcategory
$\cat{D}^{\mrm{b}}_{\mrm{f}}(A)_{\mrm{fid}}$
of $\cat{D}^{\mrm{b}}_{\mrm{f}}(A)$
consisting of complexes with finite injective dimension. The monoidal 
unit is $R_{A / \K}$. 

Work in progress by Shaul indicates that the result above 
should hold in greater generality: $A$ can be as in the first paragraph of 
the remark, and $M, N$ can be any objects in  $\cat{D}^+_{\mrm{f}}(A)$.
\end{rem}

\section{The Squaring Operation}
\label{sec:squaring}

In this section all DG rings are commutative; namely, we work 
inside the category $\cat{DGR}_{\mrm{sc}}^{\leq 0}$. See Definition 
\ref{dfn:876}. 
The category $\cat{PDGR}_{\mrm{sc}}^{\leq 0}$ of commutative pairs of DG rings 
was introduced in Definition \ref{dfn:911}(2); it is a full subcategory of the 
category $\cat{PDGR}$ of central pairs of DG rings. 

\begin{dfn} \label{dfn:955}
\begin{enumerate}
\item Let $B / A$ be an object of $\cat{PDGR}_{\mrm{sc}}^{\leq 0}$.
A {\em K-flat resolution} of $B / A$ in 
$\cat{PDGR}_{\mrm{sc}}^{\leq 0}$
is a resolution 
$s / r : \til{B} / \til{A} \to B / A$
in the sense of Definition \ref{dfn:912}, such that the pair
$\til{B} / \til{A}$ is commutative and K-flat (Definition \ref{dfn:60}(3)). 

\item Morphisms of resolutions in $\cat{PDGR}_{\mrm{sc}}^{\leq 0}$,
and resolutions of morphisms in $\cat{PDGR}_{\mrm{sc}}^{\leq 0}$,
are as in Definition \ref{dfn:913}, but all the pairs are now commutative. 
\end{enumerate}
\end{dfn}

\begin{prop} \label{prop:1010}
Any object and any morphism in $\cat{PDGR}_{\mrm{sc}}^{\leq 0}$
admits a K-flat resolution.
\end{prop}

\begin{proof}
Like the proof of Corollary \ref{cor:980}, but using commutative semi-free 
resolutions. 
\end{proof}

For a commutative DG ring $B$ we have $B^{\mrm{ce}} = B$ of course. This is 
used in the following definitions. 

\begin{dfn} \label{dfn:950}
Let $B / A$ be a pair of commutative DG rings, and let $M$ be a DG 
$B$-module. Given a K-flat resolution $\til{B} / \til{A}$ of $B / A$ in  
$\cat{PDGR}_{\mrm{sc}}^{\leq 0}$, 
we let 
\[ \opn{Sq}_{B / A}^{\til{B} / \til{A}}(M) := 
\opn{Rect}_{B / A}^{\til{B} / \til{A}}(M, M) \in \cat{D}(B) . \]
See Proposition \ref{prop:915}. 
\end{dfn}

\begin{dfn} \label{dfn:956}
Let $w / v : B' / A' \to B / A$ be a morphism in  
$\cat{PDGR}_{\mrm{sc}}^{\leq 0}$, let 
$M \in \cat{D}(B)$, let $M' \in \cat{D}(B')$, and let 
$\th : M \to M'$
be a morphism in $\cat{D}(B')$.
Given a K-flat resolution
$\til{w} / \til{v} : \til{B}' / \til{A}' \to \til{B} / \til{A}$
of $w / v$ in $\cat{PDGR}_{\mrm{sc}}^{\leq 0}$,
we define
\[ \opn{Sq}_{w / v}^{\til{w} / \til{v}}(\th) :
\opn{Sq}_{B / A}^{\til{B} / \til{A}}(M) \to 
\opn{Sq}_{B' / A'}^{\til{B}' / \til{A}'}(M') \]
to be the morphism 
\[ \opn{Sq}_{w / v}^{\til{w} / \til{v}}(\th) :=
\opn{Rect}_{w / v}^{\til{w} / \til{v}}(\th, \th) \]
in $\cat{D}(B')$ from Proposition \ref{prop:916}.
\end{dfn}

\begin{rem} \label{rem:965}
As explained in Remark \ref{rem:950}, we can write 
\[ \opn{Sq}_{B / A}^{\til{B} / \til{A}}(M) = 
\opn{RHom}_{\til{B}^{\mrm{en}}}(B, M \ot_{\til{A}}^{\mrm{L}} M) . \]
However, the morphism $\opn{Sq}_{w / v}^{\til{w} / \til{v}}(\th)$
in Definition \ref{dfn:956} is not standard, so we 
have to use the more explicit approach with compound resolutions, as was 
done in Section \ref{sec:pairs-DG-modules}. 

Still, here we can simplify matters by using symmetric compound resolutions.
A symmetric compound resolution of a DG $B$-module $M$ is data 
$\bsym{P} = (\til{P}, \til{I}; \al, \be)$, 
consisting of a K-flat resolution 
$\al : \til{P} \to M$ in $\cat{M}(\til{B})$, 
and a K-injective resolution 
$\be : \til{P}^{\mrm{en}} = \til{P} \ot_{\til{A}} \til{P} \to \til{I}$
in $\cat{M}(\til{B}^{\mrm{en}})$.

Likewise, a morphism $\th : M \to M'$ can be resolved by a symmetric compound
morphism $\bsym{\eta} : \bsym{P} \to \bsym{P}'$,
where 
$\bsym{\eta} = (\til{Q}; \ga, \til{\th}, \eta)$
consists of  a K-projective resolution
$\ga : \til{Q} \to \til{P}$ in $\cat{M}(\til{B}')$;
a homomorphism
$\til{\th} : \til{Q} \to \til{P}'$
in $\cat{M}(\til{B}')$ lifting $\th$; and a homomorphism
$\eta : \til{I} \to \til{I}'$ in 
$\cat{M}(\til{B}'^{\, \mrm{en}})$ lifting 
\[ \til{\th} \ot \til{\th} : \til{Q} \ot_{\til{A}'} \til{Q} \to 
\til{P}' \ot_{\til{A}'} \til{P}' . \]
\end{rem}

Here is the key technical result of our paper (Theorem \ref{thm:960} in the 
Introduction). Actually, it is just a special case of Theorem \ref{thm:840}. 

\begin{thm}[Homotopy Invariance] \label{thm:950}
Let $w / v : B' / A' \to B / A$ be a morphism in  
$\cat{PDGR}_{\mrm{sc}}^{\leq 0}$, let 
$M \in \cat{D}(B)$, let $M' \in \cat{D}(B')$, and let 
$\th : M \to M'$ be a morphism in $\cat{D}(B')$.
Suppose
$\til{B} / \til{A}$ and $\til{B}' / \til{A}'$
are K-flat resolutions of 
$B / A$ and $B' / A'$ respectively in $\cat{PDGR}_{\mrm{sc}}^{\leq 0}$,
and  
\[ \til{w}_0 / \til{v}_0,  \, \til{w}_1 / \til{v}_1 :
\til{B}' / \til{A}' \to \til{B} / \til{A} \]
are morphisms of resolutions above $w / v$. Then the morphisms 
\[ \opn{Sq}_{w / v}^{\til{w}_0 / \til{v}_0}(\th), \,
\opn{Sq}_{w / v}^{\til{w}_1 / \til{v}_1}(\th) :
\opn{Sq}_{B / A}^{\til{B} / \til{A}}(M) \to 
\opn{Sq}_{B' / A'}^{\til{B}' / \til{A}'}(M') \]
in $\cat{D}(B')$ are equal. 
\end{thm}
 
\begin{proof}
Take $M^{\mrm{l}} = M^{\mrm{r}} := M$, 
$M'^{\, \mrm{l}} = M'^{\, \mrm{r}} := M'$, and
$\th^{\mrm{l}} = \th^{\mrm{r}} := \th$
in Theorem \ref{thm:840}.
\end{proof}

This brings us to the first main theorem of the paper (which is Theorem 
\ref{thm:965} in the Introduction). 

\begin{thm}[Existence of Squares] \label{thm:860} 
Let $A \to B$ be a homomorphism of commutative DG rings, and let $M$ be a DG 
$B$-module. There is a DG module 
$\opn{Sq}_{B / A}(M)$, 
unique up to a unique isomorphism in $\cat{D}(B)$, 
together with an isomorphism 
\[ \opn{sq}^{\til{B} / \til{A}} : \opn{Sq}_{B / A}(M) \iso 
\opn{Sq}_{B / A}^{\til{B} / \til{A}}(M) \]
in $\cat{D}(B)$ for each K-flat resolution 
$\til{B} / \til{A}$ of $B / A$ in $\cat{PDGR}_{\mrm{sc}}^{\leq 0}$,
satisfying the following condition. 

\begin{enumerate}
\item[($*$)] Let 
$\til{w} / \til{v} : \til{B}' / \til{A}' \to  \til{B} / \til{A}$
be a morphism of K-flat resolutions of $B / A$ in 
$\cat{PDGR}_{\mrm{sc}}^{\leq 0}$. 
Then there is equality
\[ \opn{Sq}_{\mrm{id} / \mrm{id}}^{\til{w} / \til{v}}(\opn{id}_{M}) = 
\opn{sq}_{}^{\til{B}' / \til{A}'} \circ \,
(\opn{sq}_{}^{\til{B} / \til{A}})^{-1} \]
of isomorphisms
$\opn{Sq}_{B / A}^{\til{B} / \til{A}}(M) \iso 
\opn{Sq}_{B / A}^{\til{B}' / \til{A}'}(M)$
in $\cat{D}(B)$.
\end{enumerate}
\end{thm}

\begin{proof}
The proof is very similar to that of Theorem \ref{thm:985}: we fix a 
universal K-flat resolution of $B / A$; but now it has to be a commutative 
resolution. So let $\til{A}_{\mrm{un}} / \Z$ be a commutative semi-free 
resolution of $A / \Z$, and let 
$\til{B}_{\mrm{un}} / \til{A}_{\mrm{un}}$
be a commutative semi-free resolution
of $B / \til{A}_{\mrm{un}}$, both in $\cat{PDGR}_{\mrm{sc}}^{\leq 0}$.
This is possible by Theorem \ref{thm:305}. Define 
\[ \opn{Sq}_{B / A}(M) :=
\opn{Sq}_{B / A}^{\til{B}_{\mrm{un}} / \til{A}_{\mrm{un}}}(M) . \]

Given any K-flat resolution 
$\til{B} / \til{A}$ of $B / A$ in $\cat{PDGR}_{\mrm{sc}}^{\leq 0}$,
Theorem \ref{thm:306} says there is a morphism of resolutions 
\[ \til{w}_{\mrm{un}} / \til{v}_{\mrm{un}} :
\til{B}_{\mrm{un}} / \til{A}_{\mrm{un}} \to \til{B} / \til{A} . \]
We define the isomorphism
\[ \opn{sq}^{\til{B} / \til{A}} : 
\opn{Sq}_{B / A}^{\til{B}_{\mrm{un}} / \til{A}_{\mrm{un}}}(M) \iso 
\opn{Sq}_{B / A}^{\til{B} / \til{A}}(M) \]
to be 
\[ \opn{sq}^{\til{B} / \til{A}} := 
\opn{Sq}_{\mrm{id} / \mrm{id}}
^{\til{w}_{\mrm{un}} / \til{v}_{\mrm{un}}}(\opn{id}_{M})^{-1} . \]

The verification of condition ($*$) is the same as in the proof of Theorem 
\ref{thm:985}, except that here we use Theorem \ref{thm:950} instead of 
Theorem \ref{thm:840}.
\end{proof}

\begin{dfn} \label{dfn:1000}
A {\em triple in $\cat{DGR}_{\mrm{sc}}^{\leq 0}$} is the data of 
homomorphisms $A \xar{u} B \xar{v} C$ in $\cat{DGR}_{\mrm{sc}}^{\leq 0}$.
We refer to this triple as $C / B / A$.  

A {\em commutative K-flat resolution of $C / B / A$} is a 
triple 
$\til{A} \xar{\til{u}} \til{B} \xar{\til{v}} \til{C}$
in $\cat{DGR}_{\mrm{sc}}^{\leq 0}$,
together with homomorphisms 
$r : \til{A} \to A$, $s : \til{B} \to B$ and $t : \til{C} \to C$
in $\cat{DGR}_{\mrm{sc}}^{\leq 0}$, such that:
\begin{itemize}
\item $r$ is a quasi-isomorphism, and $s, t$ are surjective quasi-isomorphisms.
\item $s \circ \til{u} = u \circ r$ and $t \circ \til{v} = v \circ s$.
\item $\til{u}$ and $\til{v} \circ \til{u}$ are K-flat. 
\end{itemize}
\end{dfn}

See diagram (\ref{eqn:966}) in the Introduction. 
A commutative K-flat resolution 
$\til{A} \xar{\til{u}} \til{B} \xar{\til{v}} \til{C}$
can be viewed as follows:  K-flat resolutions
$s / r : \til{B} / \til{A} \to B / A$ and 
$t / r : \til{C} / \til{A} \to C / A$
in $\cat{PDGR}_{\mrm{sc}}^{\leq 0}$, 
and a morphism of resolutions 
$\til{v} / \mrm{id}_{\til{A}} : \til{B} / \til{A} \to \til{C} / \til{A}$
above the morphism of pairs 
$v / \mrm{id}_{A} : B / A \to C / A$. 

\begin{prop} \label{prop:1011}
Commutative K-flat resolutions of triples exist. 
\end{prop}

\begin{proof}
This follows from Theorems \ref{thm:305} and \ref{thm:306}. Compare to the 
proof of Corollary \ref{cor:980}.
\end{proof}

Here is the second main theorem of the paper (which is Theorem \ref{thm:966} in 
the Introduction). 

\begin{thm}[Trace Functoriality] \label{thm:871}
Let $A \xar{} B \xar{v} C$ 
be homomorphisms of commutative DG rings, 
let $M \in \cat{D}(B)$, let $N \in \cat{D}(C)$, and let 
$\th : N \to M$ be a morphism in $\cat{D}(B)$.
There is a unique morphism 
\[ \opn{Sq}_{v / \mrm{id}_A}(\th) :
\opn{Sq}_{C / A}(N) \to \opn{Sq}_{B / A}(M) \]
in $\cat{D}(B)$, satisfying the condition\tup{:}
\begin{enumerate}
\item[($**$)] For any commutative K-flat resolution
$\til{A} \to \til{B} \xar{\til{v}} \til{C}$
of the triple $A \to B \xar{v} C$, there is equality 
\[ \opn{Sq}_{v / \mrm{id}_A}^{\til{v} / \mrm{id}_{\til{A}}}(\th) = 
 \opn{sq}^{\til{B} / \til{A}} \circ \opn{Sq}_{v / \mrm{id}_A}(\th) 
\circ (\opn{sq}^{\til{C} / \til{A}})^{-1} \]
of morphisms 
$\opn{Sq}_{C / A}^{\til{C} / \til{A}}(N) \to 
\opn{Sq}_{B / A}^{\til{B} / \til{A}}(M)$
in $\cat{D}(B)$.
\end{enumerate}
\end{thm}

\begin{proof}
We begin by choosing a universal K-flat resolution of the triple 
$A \xar{u} B \xar{v} C$.
We choose, in this order, commutative semi-free resolutions 
$r_{\mrm{un}} / \mrm{id} : \til{A}_{\mrm{un}} / \Z \to A / \Z$, 
$s_{\mrm{un}} / r_{\mrm{un}} : \til{B}_{\mrm{un}} / \til{A}_{\mrm{un}} \to B / 
\til{A}_{\mrm{un}}$
and
$t_{\mrm{un}} / s_{\mrm{un}} : \til{C}_{\mrm{un}} / \til{B}_{\mrm{un}} \to C / 
\til{B}_{\mrm{un}}$.
In this way we obtain a K-flat resolving triple 
$\til{A}_{\mrm{un}} \xar{\til{u}_{\mrm{un}}} \til{B}_{\mrm{un}} 
\xar{\til{v}_{\mrm{un}}} \til{C}_{\mrm{un}}$
of $A \xar{u} B \xar{v} C$.
Now we define 
\[ \opn{Sq}_{v / \mrm{id}_A}(\th) := 
(\opn{sq}^{\til{B}_{\mrm{un}} / \til{A}_{\mrm{un}}})^{-1} \circ 
\opn{Sq}_{v / \mrm{id}_A}^{\til{v}_{\mrm{un}} / 
\mrm{id}_{\til{A}_{\mrm{un}}}}(\th)
\circ \opn{sq}^{\til{C}_{\mrm{un}} / \til{A}_{\mrm{un}}} . \]

Let us verify condition ($**$). We are given an arbitrary commutative K-flat 
resolution 
$\til{A} \xar{\til{u}} \til{B} \xar{\til{v}} \til{C}$
of the triple $A \to B \xar{v} C$.
According to Theorem \ref{thm:306} we can find homomorphisms 
$\til{r} : \til{A}_{\mrm{un}} \to \til{A}$,
$\til{s} : \til{B}_{\mrm{un}} \to \til{B}$ and
$\til{t} : \til{C}_{\mrm{un}} \to \til{C}$, 
such that the diagram below in $\cat{PDGR}_{\mrm{sc}}^{\leq 0}$
is commutative.  

\begin{equation} \label{eqn:1000}
\UseTips \xymatrix @C=10ex @R=6ex {
\til{B}_{\mrm{un}} / \til{A}_{\mrm{un}}
\ar[r]^{ \til{v}_{\mrm{un}} / \mrm{id} }
\ar[d]^{ \til{s} / \til{r} }
\ar @/_3em/ [dd]_(0.7){ s_{\mrm{un}} / r_{\mrm{un}} }
&
\til{C}_{\mrm{un}} / \til{A}_{\mrm{un}}
\ar[d]_{ \til{t} / \til{r} }
\ar @/^3em/ [dd]^(0.7){ t_{\mrm{un}} / r_{\mrm{un}} }
\\
\til{B} / \til{A}
\ar[r]^{ \til{v} / \mrm{id} }
\ar[d]^{ s / r }
&
\til{C} / \til{A}
\ar[d]_{ t / r }
\\
B / A
\ar[r]^{ v / \mrm{id} }
&
C / A
} 
\end{equation}

Now let us look at the corresponding diagram with the squaring operation. It is 
a diagram in $\cat{D}(B)$. 

\[ \UseTips \xymatrix @C=16ex @R=8ex {
\opn{Sq}_{B / A}^{\til{B}_{\mrm{un}} / \til{A}_{\mrm{un}}}(M)
&
\opn{Sq}_{C / A}^{\til{C}_{\mrm{un}} / \til{A}_{\mrm{un}}}(N)
\ar[l]_(0.5){ \opn{Sq}_{v / \mrm{id}}^{\til{v}_{\mrm{un}} / \mrm{id}}(\th) }
\\
\opn{Sq}_{B / A}^{\til{B} / \til{A}}(M)
\ar[u]_{ \opn{Sq}_{\mrm{id} / \mrm{id}}^{\til{s} / \til{r}}(\mrm{id}) }
&
\opn{Sq}_{C / A}^{\til{C} / \til{A}}(N)
\ar[u]^{ \opn{Sq}_{\mrm{id} / \mrm{id}}^{\til{t} / \til{r}}(\mrm{id}) }
\ar[l]_(0.5){ \opn{Sq}_{v / \mrm{id}}^{\til{v} / \mrm{id}}(\th) }
\\
\opn{Sq}_{B / A}(M)
\ar @/^4em/ [uu]^(0.3){ \opn{sq}^{\til{B}_{\mrm{un}} / \til{A}_{\mrm{un}}} }
\ar[u]_{ \opn{sq}^{\til{B} / \til{A}} }
&
\opn{Sq}_{C / A}(N)
\ar[u]^{ \opn{sq}^{\til{C} / \til{A}} }
\ar[l]_{ \opn{Sq}_{v / \mrm{id}}(\th) }
\ar @/_3em/ [uu]_(0.3){ \opn{sq}^{\til{C}_{\mrm{un}} / \til{A}_{\mrm{un}}} }
} \]

The top square is commutative by Proposition \ref{prop:940} (applied twice). 
The two half-moons are commutative by condition ($*$) of Theorem \ref{thm:860}.
The outer paths are equal by definition of 
$\opn{Sq}_{v / \mrm{id}}(\th)$. Because the vertical arrows are isomorphisms, 
we conclude that the bottom square is commutative. This is what we had to 
prove. 
\end{proof}

\begin{cor} \label{cor:970}
Let $B / A$ be a pair of commutative DG rings.
The assignments 
\[ M \mapsto \opn{Sq}_{B / A}(M) \quad \tup{and} \quad 
\th \mapsto \opn{Sq}_{\mrm{id}_B / \mrm{id}_A}(\th) \]
are a functor 
\[ \opn{Sq}_{B / A} : \cat{D}(B) \to  \cat{D}(B) . \]
\end{cor}

\begin{proof}
Consider morphisms 
$M_0 \xar{\th_1} M_1 \xar{\th_2} M_2$
in $\cat{D}(B)$. A combination of Proposition \ref{prop:940} and condition 
($**$) of Theorem \ref{thm:871} shows that 
\[ \opn{Sq}_{\mrm{id} / \mrm{id}}(\th_2) \circ 
\opn{Sq}_{\mrm{id} / \mrm{id}}(\th_1) = 
\opn{Sq}_{\mrm{id} / \mrm{id}}(\th_2 \circ \th_1) . \]
And by Lemma \ref{lem:925}, 
$\opn{Sq}_{\mrm{id} / \mrm{id}}(\mrm{id}_M) = 
\mrm{id}_{\opn{Sq}_{B / A}(M)}$.
\end{proof}

\begin{dfn} \label{dfn:1001}
The object $\opn{Sq}_{B / A}(M)$ from Theorem \ref{thm:860}  is called the 
{\em square of $M$ over $B$ relative to $A$}. 
The functor $\opn{Sq}_{B / A}$ from Corollary \ref{cor:970}
is called the {\em squaring operation  over $B$ relative to $A$}. 
\end{dfn}

The following result says that the squaring operation $\opn{Sq}_{B / A}$ is 
also functorial in the DG ring $B$. 

\begin{prop} \label{prop:1000}
We are given homomorphisms of commutative DG rings 
$A \to B_2 \xar{v_2} B_1 \xar{v_1} B_0$,
DG modules $M_k \in \cat{D}(B_k)$,
and morphisms 
$\th_k : M_{k - 1} \to  M_k$ in $\cat{D}(B_k)$.
Then 
\[ \opn{Sq}_{v_2 / \mrm{id}_A}(\th_2) \circ 
\opn{Sq}_{v_1 / \mrm{id}_A}(\th_1) = 
\opn{Sq}_{(v_1 \circ v_2) / \mrm{id}_A}(\th_2 \circ \th_1) , \]
as morphisms 
$\opn{Sq}_{B_0 / A}(M_0) \to \opn{Sq}_{B_2 / A}(M_2)$ in $\cat{D}(B_2)$.  
\end{prop}

\begin{proof}
Choose a K-flat resolution 
$\til{A} \xar{} \til{B}_2 \xar{\til{v}_2} \til{B}_1 \xar{\til{v}_1} \til{B}_0$
of 
$A \to B_2 \xar{v_2} B_1 \xar{v_1} B_0$, 
in the obvious sense that generalizes 
Definition \ref{dfn:1000}. According to condition ($**$) of Theorem 
\ref{thm:871}, it suffices to prove that 
\[ \opn{Sq}_{v_2 / \mrm{id}}^{\til{v}_2 / \mrm{id}} (\th_2) \circ 
\opn{Sq}_{v_1 / \mrm{id}}^{\til{v}_1 / \mrm{id}}(\th_1) = 
\opn{Sq}_{(v_1 \circ v_2) / \mrm{id}}^{(\til{v}_1 \circ \til{v}_2) / \mrm{id}}
(\th_2 \circ \th_1) , \]
as morphisms 
$\opn{Sq}_{B_0 / A}^{\til{B}_0 / \til{A}}(M_0) \to 
\opn{Sq}_{B_2 / A}^{\til{B}_2 / \til{A}}(M_2)$ 
in $\cat{D}(B_2)$. This is true by Proposition \ref{prop:940}.
\end{proof}

\begin{prop} \label{prop:1015}
Let $A$ be a commutative DG ring.
The action from Proposition \tup{\ref{prop:461}(1)}
makes $\cat{D}(A)$ into an $\opn{H}^0(A)$-linear category. 
\end{prop}

\begin{proof}
Because $A^0$ is in the center of $A$, and $\d(A^0) = 0$, we see that 
$\cat{M}(A)$ is an $A^0$-linear category. See formula (\ref{eqn:1015}). 
Next, from formula 
(\ref{eqn:1016}) we see that $\cat{K}(A)$ is an 
$\opn{H}^0(A)$-linear category. Hence its localization $\cat{D}(A)$ is an
$\opn{H}^0(A)$-linear category.
\end{proof}

Given a homomorphism $B \to C$ of commutative DG rings, and objects 
$M \in \cat{D}(B)$ and $N \in \cat{D}(C)$, the set 
$\opn{Hom}_{\cat{D}(B)}(N, M)$
is an $\opn{H}^0(C)$-module, with action coming the action of 
$\opn{H}^0(C)$ on $N$ as an object of $\cat{D}(C)$, as in the proposition 
above. 

We end the paper with the next result.

\begin{thm} \label{thm:1010}
In the situation of Theorem \tup{\ref{thm:871}}, let $c \in \opn{H}^0(C)$. 
Then 
\[ \opn{Sq}_{v / \mrm{id}}(c \cd \th) =
c^2 \cd \opn{Sq}_{v / \mrm{id}}(\th) , \]
as morphisms 
\[ \opn{Sq}_{C / A} (N)  \to \opn{Sq}_{B / A} (M) \]
in $\cat{D}(B)$.
\end{thm}

\begin{proof}
Choose a K-flat resolution
$\til{B} / \til{A}$ of $B / A$. Then choose a strict commutative semi-free 
resolution $\til{C} / \til{B}$ of $C / \til{B}$. 
Putting them together we obtain a K-flat resolution 
$\til{A} \xar{} \til{B} \xar{\til{v}} \til{C}$
of $A \to B \xar{v} C$, where $\til{v}$ is commutative semi-free. 
Let $\bsym{P}_M = (\til{P}_M, \til{I}_M; \al_M, \be_M)$
be a symmetric compound resolution of $M$ over $\til{B} / \til{A}$, 
and let  
$\bsym{P}_N = (\til{P}_N, \til{I}_N; \al_N, \be_N)$
be a symmetric compound resolution of $N$ over $\til{C} / \til{A}$.
See Remark \ref{rem:965} and Definition \ref{dfn:919}.

Because $\til{P}_N$ is K-projective over $\til{B}$,
we can find a homomorphism 
$\til{\th} : \til{P}_N \to \til{P}_M$
in $\cat{M}(\til{B})$ that lifts $\th$. 
We obtain a symmetric compound morphism
$\bsym{\eta} : \bsym{P}_N \to \bsym{P}_M$, 
$\bsym{\eta} := (\til{P}_N; \mrm{id}_{\til{P}_N}, \til{\th}, \eta)$
above $\th$ and  $\til{v} / \mrm{id}_{\til{A}}$.
Specifically, the homomorphism 
$\eta : \til{I}_N \to \til{I}_M$
in $\cat{M}(\til{B}^{\mrm{en}})$ is a lift of 
\[ \til{\th} \ot \til{\th} : \til{P}_N \ot_{\til{A}} \til{P}_N \to 
\til{P}_M \ot_{\til{A}} \til{P}_M . \]
By Proposition \ref{prop:916},  
$\opn{Sq}_{v / \mrm{id}}(\th)$
is represented by 
\[ \opn{Hom}_{v^{\mrm{en}}}(v, \eta) : 
\opn{Hom}_{C^{\mrm{en}}}(C, \til{I}_N) \to 
\opn{Hom}_{B^{\mrm{en}}}(B, \til{I}_M) . \]
We remind that 
$\til{B}^{\mrm{en}} = \til{B} \ot_{\til{A}} \til{B}^{\mrm{op}}$
and
$\til{C}^{\mrm{en}} = \til{C} \ot_{\til{A}} \til{C}^{\mrm{op}}$.

Now choose some $\til{c} \in \til{C}^0$ that represents the cohomology class 
$c \in \opn{H}^0(C) \cong \opn{H}^0(\til{C})$. 
Then 
$\til{c} \cd \til{\th} : \til{P}_N \to \til{P}_M$ 
lifts $c \cd \th : N \to M$, and 
$(\til{c} \ot \til{c}) \cd \eta : \til{I}_N \to \til{I}_M$
lifts 
$(\til{c} \cd \til{\th}) \ot (\til{c} \cd \til{\th})$. 
We get a symmetric compound morphism 
$\bsym{\eta}_c : \bsym{P}_N \to \bsym{P}_M$, 
\[ \bsym{\eta}_c := \bigl( \til{P}_N; \mrm{id}_{\til{P}_N} , \til{c} \cd 
\til{\th}, (\til{c} \ot \til{c}) \cd \eta \bigr)  \]
resolving $c \cd \th$. But then the homomorphism 
\[ \opn{Hom}_{v^{\mrm{en}}}(v, (\til{c} \ot \til{c}) \cd \eta) = 
\opn{Hom}_{v^{\mrm{en}}}( \til{c}^2 \cd v, \eta) = 
\til{c}^2 \cd \opn{Hom}_{v^{\mrm{en}}}(v, \eta)  \]
represents 
$\opn{Sq}_{v / \mrm{id}}(c \cd \th)$.
\end{proof}

This theorem shows that the functor $\opn{Sq}_{B / A}$ is a 
{\em quadratic functor}.


\end{document}